\documentclass[10pt]{amsart}
\usepackage{graphicx}

\usepackage{amsmath}
\usepackage{amsfonts}
\usepackage{amssymb}
\usepackage{hyperref}
\usepackage{dsfont}

\usepackage{cite}

\hypersetup{
    colorlinks=true,
    urlcolor=blue,
    linkcolor=blue,
    breaklinks=true,
    citecolor=darkgreen
}
\usepackage{url}
\usepackage{xcolor,color}
\colorlet{darkgreen}{green!50!black}
\usepackage{amsbsy}
\usepackage{amsmath}

\usepackage[utf8]{inputenc} 
\usepackage[T1]{fontenc}
\usepackage[french,english]{babel}

\usepackage{geometry}
\usepackage{caption}
\usepackage{subcaption}

\def\R{{\mathbb R}}

\def\C{{\mathbb C}}

\def\1{{1\!\!\!1}}

\def\a{{\alpha}}
\def\E{{\mathbb E}}

\def\P{{\mathbb P}}

\def\cal{\mathcal}

\newcommand{\fc}{\mathds{1}}

\newcommand{\nor}[1]{ 
	\left\lvert\mkern-1.2mu\left\lvert#1%
	\right\rvert\mkern-1.2mu\right\rvert
}

\let \phi=\varphi

\newcommand{\be}{\begin{equation}}
\newcommand{\ee}{\end{equation}}
\numberwithin{equation}{section}

\newtheorem{theorem}{Theorem}
\newtheorem{prop}{Proposition}[section]

\newtheorem{defi}[prop]{Definition}
\newtheorem*{defi*}{Definition}
\newtheorem{lemma}[prop]{Lemma}
\newtheorem{rem}[prop]{Remark}
\newtheorem{assumption}{Assumption}

\usepackage{oubraces}

\newcommand{\nb}[1]{#1}
\newcommand{\nr}[1]{}

\begin{document}

\title[Asymptotics for the Green's functions of a reflected Brownian motion in a wedge]{Asymptotics for the Green's functions of a transient reflected Brownian motion in a wedge}

\author{Sandro Franceschi}
\address{Institut Polytechnique de Paris, T\'el\'ecom SudParis, Laboratoire SAMOVAR, 19 place Marguerite Perey, 91120 Palaiseau, France}
\email{sandro.franceschi@telecom-sudparis.eu}

\author{Irina Kourkova}
\address{Sorbonne Universite, Laboratoire de Probabilités, Statistiques et Modélisation, 
UMR 8001, 4 place Jussieu, 75005 Paris, France. 
}
\email{irina.kourkova@sorbonne-universite.fr}

\author{Maxence Petit}
\address{Sorbonne Universite, Laboratoire de Probabilités, Statistiques et Modélisation, 
UMR 8001, 4 place Jussieu, 75005 Paris, France. }
\email{maxence.petit@ens-rennes.fr}

\thanks{This project has received funding from Agence Nationale de la Recherche, ANR JCJC programme 
under the Grant Agreement ANR-22-CE40-0002.}

\begin{abstract}
We consider a transient Brownian motion reflected obliquely in a two-dimensional wedge. A precise asymptotic expansion of Green's functions is found in all directions. 

To this end, we first determine a kernel functional equation connecting the Laplace transforms of the Green's functions. We then extend the Laplace transforms analytically and study its singularities. We obtain the asymptotics applying the saddle point method to the inverse Laplace transform on the Riemann surface generated by the kernel.
\end{abstract}

\maketitle


\section{Introduction}

\subsection*{Context}

Since its introduction in the 1980s, reflected Brownian motion in a cone has been extensively studied \cite{HaRe-81,HaRe-81b,varadhan_brownian_1985}, particularly due to its deep links with queuing systems as an approximate model in heavy traffic \cite{harrison_78_diffusion,reiman_84_open}. Seminal work has determined the recurrent or transient nature of this process in dimension two
\cite{williams_recurrence_1985,
hobson_recurrence_1993}
\nr{but also in higher dimension which is a much more complex issue} \nb{and in higher dimensions}
\cite{Ch-96,
BrDaHa-10,
Br-11,
DaHa-12}.
\nr{The literature is full of studies of its stationary distribution in the recurrent case, such as the study of its asymptotics, which has generated a great deal of work \cite{harrison_reflected_2009,
 dai_reflecting_2011,
 DaMi-13,
 franceschi_asymptotic_2016,miyazawa_conjectures_2011,Sa-+1}, numerical methods developed to compute it \cite{dai_steady-state_1990,dai_reflected_1992} or the determination of explicit expressions of its stationary density
\cite{foddy_analysis_1984,Foschini,
baccelli_analysis_1987,
harrison_multidimensional_1987,
dieker_reflected_2009,
franceschi_tuttes_2016,
BoElFrHaRa_algebraic_2018,franceschi_explicit_2017}. }
\nb{The literature on the stationary distribution in the recurrent case, in particular the study of the asymptotics, is wide and vast \cite{harrison_reflected_2009,
 dai_reflecting_2011,
 DaMi-13,
 franceschi_asymptotic_2016,miyazawa_conjectures_2011,Sa-+1}.}
\nb{Numerical methods have been explored in \cite{dai_steady-state_1990,dai_reflected_1992} and explicit expressions for the stationary density have been given in \cite{foddy_analysis_1984,Foschini,
baccelli_analysis_1987,
harrison_multidimensional_1987,
dieker_reflected_2009,
franceschi_tuttes_2016,
BoElFrHaRa_algebraic_2018,franceschi_explicit_2017}.
}
\nr{The transient case, which is a little less studied, is also the subject of several articles which study its escape probability along the axes \cite{fomichov_franceschi_ivanovs_2022}, its absorption probability at the vertex \cite{franceschi_raschel_dual_2022,ernst_franceschi_escape_2021} or its Green's functions also called occupation density
\cite{ernst_franceschi_asymptotic_2021,franceschi_green_2021}.}\nb{The transient case, which is less studied, is also considered by several articles which study the escape probability along the axes \cite{fomichov_franceschi_ivanovs_2022}, the absorption probability at the vertex \cite{franceschi_raschel_dual_2022,ernst_franceschi_escape_2021}, or the corresponding Green’s functions \cite{ernst_franceschi_asymptotic_2021,franceschi_green_2021}}.

In this article, we consider a transient obliquely reflected Brownian motion
in a cone of angle $\beta \in (0, \pi)$ with two different reflection laws from two boundary rays 
of the cone. We denote by $\widetilde g(\rho\cos (\omega), \rho \sin(\omega)  )$ the Green's function of this process in polar coordinates\nb{; Green’s functions are used to study the distribution of time that the process spends at a point on the cone}.\nr{ Green's function is the average time density that the process spends at a point on the cone.} The article determines the asymptotics of $\widetilde g(\rho\cos(\omega),\rho\sin(\omega))$
as $\rho \to \infty$ and $\omega \to \omega_0$ for any given angle $\omega_0 \in [0, \beta]$. See Theorem~\ref{thm1} when $\omega_0\in(0,\beta)$ and Theorem~\ref{thm2} when $\omega_0=0$ or $\beta$. 
\nr{It}\nb{This} extends results of~\cite{ernst_franceschi_asymptotic_2021} in two aspects. \nr{First}\nb{Firstly}, asymptotic results are obtained 
in any convex two-dimensional cone  with two different reflection laws from its boundaries.
While in~\cite{ernst_franceschi_asymptotic_2021} the authors are able to easily calculate an explicit Laplace transform of the Green's function for the half plane, the same is certainly not true for RBM in the cone. Laplace transforms of Green functions in this case 
are expressed in \cite{franceschi_green_2021} in terms of integrals as solutions of Riemann boundary problems\nr{, which hardly suit for further analysis}. \nr{Second}\nb{Secondly}, Theorem~\ref{thm1} provides Green function's \nr{asymptotic}\nb{asymptotics} for any direction of the cone and not only along straight rays as in~\cite{ernst_franceschi_asymptotic_2021}, namely when the angle $\omega$ above 
tends to a given angle $\omega_0$\nr{and not just equals it}. The \nr{asymptotic depends}\nb{asymptotics depend} on the rate of convergence of $\omega \to \omega_0$ and \nr{allows to determine the full Martin boundary of the process}\nb{enables us to determine the Martin boundary of the process}.  

   In~\cite{franceschi_asymptotic_2016} the \nr{asymptotic}\nb{asymptotics} of the stationary distribution for recurrent Brownian motion in a cone is found along all regular directions $\omega_0 \in (0, \beta)$,
while some special directions $\omega_0$ were left open for future work. The asymptotics \nr{is}\nb{are} obtained by studying the singularities and applying the saddle point method to the inverse Laplace transform of the stationary distribution.
This article applies the approach of~\cite{franceschi_asymptotic_2016} to Green's functions and provides new techniques which \nr{allow to}\nb{enable us to} treat all special directions where the \nr{asymptotic}\nb{asymptotics} depend of the \nr{rate of convergence of $\omega$ to $\omega_0$}\nb{convergence rate of $\omega$ to $\omega_0$ rather to that of $r$ tending towards infinity.}
This is the case when $\omega_0=0$ or $\beta$ (see Theorem~\ref{thm2}), and also when the \emph{saddle point \nr{meet}\nb{meets} a pole} of the boundary Laplace transform 
(see Theorem~\ref{thm3}).

The tools used in this paper are inspired by methods introduced by Malyshev \cite{malyshev_asymptotic_1973}, \nr{which}\nb{who} studied the asymptotic of the stationary distribution for random walks in the quarter plane. \nr{Articles studying asymptotics in line with Malshev's approach pursued in that direction, such as \cite{kurkova_martin_1998}}\nb{Articles studying asymptotics in line with Malshev’s approach include \cite{kurkova_martin_1998}}, which studies the Martin boundary of random walks in the quadrant; \cite{kurkova_malyshevs_2003}, which extends these methods to the join-the-shorter-queue issue; and \cite{kourkova_random_2011}, which studies the asymptotics of the occupation measure for random walks in the quarter plane with drift absorbed at the axes. 
Fayolle and Iasnogorodski \cite{fayolle_two_1979} also developed a method to determine explicit \nr{expression}\nb{expressions} for generating functions \nb{using the}\nr{thanks to} Riemann and Carleman boundary value problems. 
Then\nb{, in the seminal book \cite{FIM17},} Fayolle, Iansogorodski and Malyshev \nr{deepened and }merged their analytic approach for random walks in the quadrant\nr{ in the famous book \cite{FIM17}}.
The work \cite{franceschi_asymptotic_2016} was the first to extend their approach to continuous stochastic processes in the quadrant to compute asymptotics of stationary distributions, and \cite{ernst_franceschi_asymptotic_2021} was the first one \nr{studying}\nb{to study} the asymptotics of Green's functions using this analytic approach.








\begin{figure}[hbtp]
\centering
\includegraphics[scale=0.7]{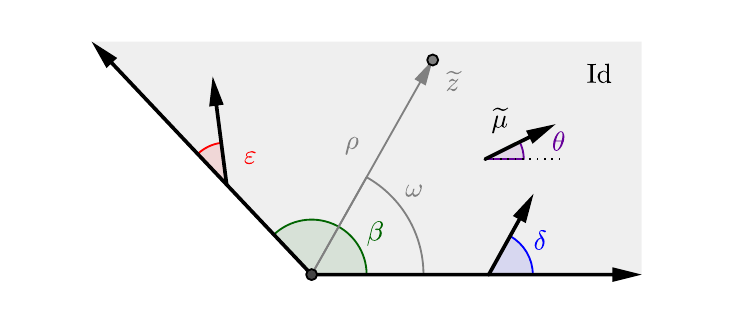}
\vspace{-0.5cm}
\caption{The cone of angle $\beta$, the reflection angles $\delta$ and $\varepsilon$ and the drift $\widetilde{\mu}$ with its direction $\theta$. \nr{In grey the point $\widetilde z$ of polar coordinates $\rho$ and $\omega$}\nb{The point $\tilde z$ with polar coordinates $\rho$ and $\omega$ is displayed.}}
\label{fig:cone}
\end{figure}

\subsection*{Main results}
\label{sec:mainresults}
 
We consider an obliquely reflected standard Brownian motion in a cone of angle $\beta\in(0,\pi)$ starting from $\widetilde{z_0}$, with reflection angles $\delta\in(0,\pi)$ and $\varepsilon\in(0,\pi)$ and of drift $\widetilde{\mu}$ of angle $\theta\in(0,\beta)$ with the horizontal axis, see Figure~\ref{fig:cone}. 
We assume that $$\delta+\varepsilon<\beta+\pi.$$ This well known condition ensures that the process is a semi-martingale reflected Brownian motion \cite{Williams-85,Williams-95}.
The reflected Brownian motion will be properly defined in the next section. 
The process is transient since we have assumed that $\theta\in(0,\beta)$ which means that the drift belongs to the cone. If we assume that \nb{$\widetilde{p_t}$}\nr{$p_t$} is the transition probability of this process, the Green's function is defined for $\widetilde z$ inside the cone by
\begin{equation}\label{densite:green}
\widetilde{g}(\widetilde z)=\int_0^\infty \widetilde{p_t}(\widetilde{z_0},\widetilde{z})\mathrm{d}t.
\end{equation}

For $\omega\in(0,\beta )$ and $\rho>0$ we will denote $\widetilde z = (\rho\cos \omega,\rho \sin \omega)$ the polar coordinates in the cone. Note that the tilde symbol $\widetilde{\,}$ stands for quantities linked to the standard reflected Brownian motion in the $\beta$-cone. The same \nr{notations}\nb{notation} without the tilde symbol will stand for the corresponding process in the quadrant $\R_+^2$\nr{, see Remark~\ref{rem:conequadrant} \nr{bellow}\nb{below}}.

\nr{We are now going to}\nb{In the next remark we} explain how to go from a standard Brownian motion reflected in a convex cone to a reflected Brownian motion reflected in a quadrant by adjusting the covariance matrix. This will be useful because our strategy of proof is to first establish our results in the quadrant for a general covariance matrix, and then to extend the results to all convex cones. \nb{The proof of the main Theorems \ref{thm1}, \ref{thm2} and \ref{thm3} stated below in the case of a cone can be found at the very end of Section~\ref{sec:asymptcone} and are based on Theorems \ref{thm4}, \ref{thm5} and \ref{thm6}, which determine the asymptotics in the case of a quadrant.}
\begin{rem}[Equivalence between cones and quadrant]
There is a bijective equivalence between the following two families of models:
\begin{itemize}
\item \emph{Standard} reflected Brownian motions (i.e. identity covariance matrix) in \emph{any convex cone} of angle $\beta\in(0,\pi)$,
\item Reflected Brownian motions in a \emph{quadrant} of \emph{any covariance matrix} of the form 
$$\left(
  \begin{array}{cc}
    \displaystyle   1 & -\cos\beta \\
    -\cos\beta & \displaystyle  1
  \end{array}\right).$$
\end{itemize} 
In Section~\ref{sec:asymptcone} this equivalence is established by means of a simple linear transformation defined in~\eqref{eq:lineartransform}.
Therefore, all the results established for one of these two families can be \nr{transposed directly to the other family}\nb{applied directly to the other family}. 

Furthermore, any reflected Brownian motion in a general convex cone and with a general covariance matrix can always be reduced via a simple linear transformation to a Brownian motion of one of the two families of models mentioned above\nb{ (see Remark~\ref{rem:gencone} below)}. 

\label{rem:conequadrant}
\end{rem}

Before presenting our results in more detail, we pause to make the following remark.  
\begin{rem}[Notation]
\nb{Throughout}\nr{Along all} this article, we will use the symbol~$\sim$ to express an asymptotic expansion of a function. If for some functions $f$ and $g_k$ we state that $f(x)\sim \sum_{k=1}^n g_k(x)$ when $x\to x_0$, then $g_k(x)=o(g_{k-1}(x))$ and $f(x)- \sum_{k=1}^n g_k(x)=o(g_n(x))$ when $x\to x_0$.
\end{rem}
We now state the main result of the article. We define the angles 
\begin{equation*}
\omega^*:=\theta-2\delta\quad\text{and}\quad 
\omega^{**}:=\theta+2\epsilon .
\label{eq:omega***}
\end{equation*}
\nb{Note}\nr{We can remark} that $\omega^{*}<\theta<\omega^{**}$. 

\begin{theorem}[Asymptotics in the general case]
We consider a standard reflected Brownian motion in a wedge of opening $\beta$, with reflection angles $\delta$ and $\varepsilon$ and a drift $\widetilde{\mu}$ of angle $\theta$ (see Figure~\ref{fig:cone}). Then, the Green's function $\widetilde{g}(\rho \cos \omega,\rho \sin \omega)$ of this process has the following asymptotics when $\omega\to\omega_0\in(0,\beta)$ and $\rho\to\infty$, for all $n\in\mathbb{N}$: 
\begin{itemize}
\item If $\omega^{*}<\omega_0<\omega^{**}$ then
\begin{equation}
\widetilde{g}
(\rho \cos \omega,\rho \sin \omega)
\underset{\rho\to\infty \atop\omega\to\omega_0}{\sim} 
e^{-2\rho|\widetilde{\mu}| \sin^2 \left( \frac{\omega-\theta}{2} \right)} \frac{1}{\sqrt{\rho}}
  \sum_{k=0}^n \frac{\widetilde{c_k}(\omega)}{ \rho^{k}}.
  \label{eq:asymptsaddlepoint}
  \end{equation}
  \item If $\omega_0<\omega^*$ then
\begin{equation}
\widetilde{g}
(\rho \cos \omega,\rho \sin \omega)
\underset{\rho\to\infty \atop\omega\to\omega_0}{\sim} 
c^{*} e^{-2\rho|\widetilde{\mu}| \sin^2 \left( {\omega+\delta-\theta} \right)}
+
e^{-2\rho|\widetilde{\mu}| \sin^2 \left( \frac{\omega-\theta}{2} \right)} \frac{1}{\sqrt{\rho}}
  \sum_{k=0}^n \frac{\widetilde{c_k}(\omega)}{ \rho^{k}}.
  \label{eq:asymptpole1}
  \end{equation}
  \item If $\omega^{**}<\omega_0$ then
\begin{equation}
\widetilde{g}
(\rho \cos \omega,\rho \sin \omega)
\underset{\rho\to\infty \atop\omega\to\omega_0}{\sim}
c^{**}e^{-2\rho|\widetilde{\mu}| \sin^2 \left( {\omega-\epsilon-\theta} \right)}
+
e^{-2\rho|\widetilde{\mu}| \sin^2 \left( \frac{\omega-\theta}{2} \right)} \frac{1}{\sqrt{\rho}}
  \sum_{k=0}^n \frac{\widetilde{c_k}(\omega)}{ \rho^{k}}.
  \label{eq:asymptpole2}
  \end{equation}
\end{itemize}
where $c^*$ and $c^{**}$ are {positive} constants and $c_k(\omega)$ are constants depending on $\omega$ such that $\widetilde{c_k}(\omega)\underset{\omega\to\omega_0}{\longrightarrow} \widetilde{c_k}(\omega_0)$.
\label{thm1}
\end{theorem}
There are four cases which are illustrated by Figure~\ref{fig:coneasympt}. 

\begin{figure}[hbtp]
\centering
     \begin{subfigure}[b]{0.4\textwidth}
\includegraphics[scale=0.72]{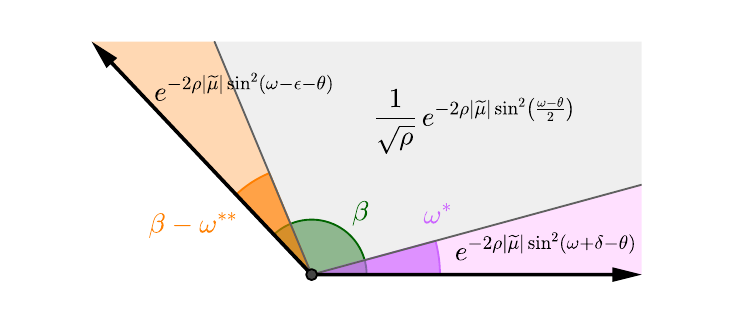}
\caption{$0<\omega^{*}<\omega^{**}<\beta$}
     \end{subfigure}
     \hspace{1cm}
     \centering
     \begin{subfigure}[b]{0.4\textwidth}
\includegraphics[scale=0.72]{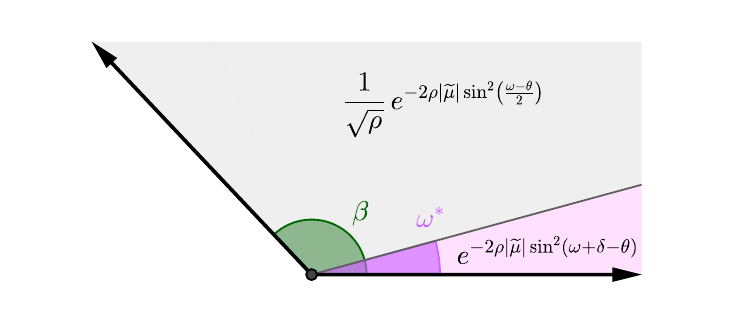}
\caption{$0<\omega^{*}<\beta<\omega^{**}$}
     \end{subfigure}
     \centering
     \begin{subfigure}[b]{0.4\textwidth}
\includegraphics[scale=0.72]{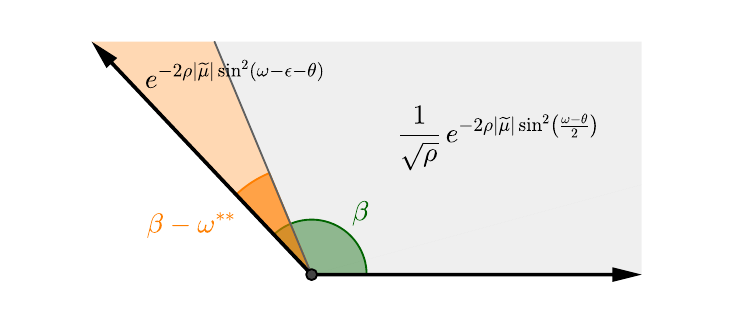}
\caption{$\omega^{*}<0<\omega^{**}<\beta$}
     \end{subfigure}
     \hspace{1cm}
     \centering
     \begin{subfigure}[b]{0.4\textwidth}
\includegraphics[scale=0.72]{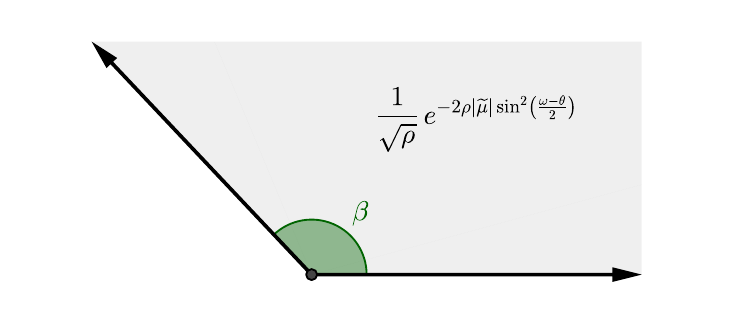}
\caption{$\omega^{*}<0<\beta<\omega^{**}$}
     \end{subfigure}
\caption{Asymptotics of the Green's function determined in Theorem~\ref{thm1} according to the direction $\omega_0$: four different cases according to the value of angles $\omega^{*}=\theta-2\delta$ and $\omega^{**}=\theta+2\epsilon$.\nr{When $\omega_0$ is in the grey region the asymptotics is given by \eqref{eq:asymptsaddlepoint}, in the purple region by \eqref{eq:asymptpole1}, in the orange region by
\eqref{eq:asymptpole2}.
}
\nb{When $\omega_0$ belongs to the gray region, the asymptotics are given by \eqref{eq:asymptsaddlepoint}; in the purple region, they are given by \eqref{eq:asymptpole1}; in the orange region, they are given by \eqref{eq:asymptpole2}.
}}
\label{fig:coneasympt}
\end{figure}

Our second result states the asymptotics \nr{near}\nb{along} the edges when $\omega\to0$ or $\omega\to\beta$.
\begin{theorem}[Asymptotics along the edges]
We now assume that $\omega_0=0$ and let $\rho\to\infty$ and $\omega\to \omega_0=0$.
In \nr{these}\nb{this} case, we have $\widetilde{c_0}(\omega)\underset{\omega\to 0}{\sim} c' \omega$ and $\widetilde{c_1}(\omega)\underset{\omega\to 0}{\sim} c''$ for some non-negative constants $c'$ and $c''$ which are non-null when $\omega^* <0$. Then, the Green's function $\widetilde{g}(\rho \cos \omega,\rho \sin \omega)$ has the following asymptotics:
\begin{itemize}
\item When $\omega^{*}<0$ the asymptotics are still given by \eqref{eq:asymptsaddlepoint}. In particular, we have
$$
\widetilde{g}
(\rho \cos \omega,\rho \sin \omega)
\underset{\rho\to\infty \atop\omega\to 0}{\sim} 
e^{-2\rho|\widetilde{\mu}| \sin^2 \left( \frac{\omega-\theta}{2} \right)} \frac{1}{\sqrt{\rho}} 
\left( c' \omega + \frac{c''}{\rho} \right).
$$
\item When $\omega^{*}>0$ the asymptotics given by \eqref{eq:asymptpole1} \nr{remains}\nb{remain} valid. In particular, we have
$$
\widetilde{g}
(\rho \cos \omega,\rho \sin \omega)
\underset{\rho\to\infty \atop\omega\to 0}{\sim} 
c^{*} e^{-2\rho|\widetilde{\mu}| \sin^2 \left( {\omega+\delta-\theta} \right)}.
$$
\nb{where $c^*$ is the same constant as in Theorem \ref{thm1}.}
\end{itemize}
Therefore, when $\omega^* < 0$, there is a competition between the two first terms of the sum $\sum_{k=0}^n \frac{\widetilde{c_k}(\omega)}{ \rho^{k}}$ to know which one is dominant between $c'\omega$ and $\frac{c''}{\rho}$. More precisely:
\begin{itemize}
\item If $\rho \sin \omega \underset{\rho\to\infty \atop \omega\to 0}{\longrightarrow} \infty$ then the first term is dominant.
\item If $\rho \sin \omega \underset{\rho\to\infty \atop \omega\to 0}{\longrightarrow} c>0$ \nr{then both terms contribute, they have the same order of magnitude}\nb{then both terms contribute and have the same order of magnitude}.
\item If $\rho \sin \omega \underset{\rho\to\infty \atop \omega\to 0}{\longrightarrow} 0$ then the second term is dominant.
\end{itemize}

A symmetric result holds when we take $\omega_0=\beta$. The asymptotics are given by \eqref{eq:asymptsaddlepoint} when $\beta<\omega^{**}$ and  by \eqref{eq:asymptpole2} when $\omega^{**}<\beta$. The first two terms of the sum compete to be dominant, and this depends on the limit of $\rho \sin(\beta-\omega)$.
\label{thm2}  
\end{theorem}

We will explain later in Propositions~\ref{prop:saddlepolar} and \ref{prop:polepolar} that $\omega^*$ and $\omega^{**}$ correspond in some sense to the poles of the Laplace transforms of the Green's functions and that $\omega$  \nr{correspond}\nb{corresponds}  to the saddle point obtained when we will take the inverse of the Laplace transform. Our third result states the asymptotics when the saddle point \nr{\emph{meet}}\nb{\emph{meets}} the poles, which \nr{means}\nb{occurs} when $\omega\to\omega^*$ or $\omega\to\omega^{**}$.

\nb{Throughout, we let $\Phi(z):= \frac{2}{\sqrt{\pi}} \int_0^z \exp(-t^2)dt$.}
\begin{theorem}[Asymptotics when the saddle point \nr{\emph{meet}}\nb{\emph{meets}} a pole]
We now assume that $\omega_0=\omega^{*}=\theta-2\delta$ and let $\omega\to\omega^{*}$ and $\rho\to\infty$. Then, the Green's function $\widetilde{g}(\rho \cos \omega,\rho \sin \omega)$ has the following asymptotics:
\begin{itemize}
\item When $\rho(\omega-\omega^{*})^2\to 0$ the asymptotics \nr{is}\nb{are} given by \eqref{eq:asymptpole1} with the constant $c^{*}$ of the first term has to be replaced by $\frac{1}{2}c^{*}$.
\item When $\rho(\omega-\omega^{*})^2\to c>0$ for some constant $c$ then:
\begin{itemize}
\item If $\omega<\omega^{*}$ the asymptotics \nr{is}\nb{are} still given by \eqref{eq:asymptpole1} with the constant $c^{*}$ of the first term has to be replaced by $\frac{1}{2}c^{*}(1+\Phi(\sqrt{c}A))$ for some constant $A$.
\item If $\omega>\omega^{*}$ the asymptotics \nr{is}\nb{are} still given by \eqref{eq:asymptpole1} with the constant $c^{*}$ of the first term has to be replaced by $\frac{1}{2}c^{*}(1-\Phi(\sqrt{c}A))$ for some constant $A$.
\end{itemize} \nr{In the previous items, we denoted $\Phi(z):= \frac{2}{\sqrt{\pi}} \int_0^z \exp(-t^2)dt$.}
\item When $\rho(\omega-\omega^{*})^2\to \infty$ then:
\begin{itemize}
\item If $\omega<\omega^{*}$ the asymptotics \nr{is}\nb{are} given by \eqref{eq:asymptpole1}
\item If $\omega>\omega^{*}$ the asymptotics \nr{is}\nb{are} given by \eqref{eq:asymptsaddlepoint} and we have $\widetilde{c_0}(\omega)\underset{\omega\to\omega^{*}}{\sim}\frac{c}{\omega-\omega^{*}}$ for some constant $c$.
\end{itemize}
\end{itemize}

A symmetric result holds when we assume that $\omega_0=\omega^{**}=\theta+2\epsilon$.
\label{thm3}
\end{theorem}



\nb{These main asymptotic results are very similar to those obtained in the article \cite{franceschi_asymptotic_2016} on the stationary distribution in the recurrent case when the drift points towards the apex of the cone. This makes sense given that the Green's functions and the stationary distribution measure the time or proportion of time that the process spends at a point. However, the analysis of Green's functions is more complex because of their dependence on the initial state of the process.
}

In the three previous theorems, we considered a Brownian motion which is \emph{standard}, i.e. of covariance matrix identity. But all the results stated above may easily be extended to all covariance matrices \nb{by the}\nr{thanks to} the simple linear transformation mentioned in the previous remark. The next remark explains how to proceed, in line with what is stated in Section~\ref{sec:asymptcone}. 

\begin{rem}[Generalisation to any covariance matrix in any convex cone]
Consider $\widehat Z_t$ an obliquely reflected Brownian motion in a cone of angle $\widehat \beta_0\in(0,\pi)$ starting from $\widehat{z_0}$, with reflection angles $\widehat\delta$ and $\widehat\varepsilon$, of drift $\widehat{\mu}$ of angle $\widehat\theta$ and of covariance matrix $\widehat \Sigma$.  
We introduce the angle $\widehat\beta_1:=\arccos \left( -\frac{\widehat\sigma_{12}}{\sqrt{\widehat\sigma_{11}\widehat\sigma_{22}}} \right)
\in(0,\pi)$ and the linear transformation
$$
\widehat T:=
\left(
  \begin{array}{cc}
    \displaystyle     \frac 1 {\sin \widehat\beta_1} & \cot \widehat\beta_1\\
    0 & 1
  \end{array}
\right)
\left(
  \begin{array}{cc}
    \displaystyle    \frac 1 {\sqrt{\widehat\sigma_{11}}} & 0\\
    0 & \displaystyle  \frac 1 {\sqrt{\widehat\sigma_{22}}}.
  \end{array}\right)
$$
Then, the process $\widetilde Z_t:=\widehat T \widehat Z_t$ is an obliquely reflected \emph{standard} Brownian motion in a cone of angle $\beta\in(0,\pi)$ starting from $\widetilde{z_0}:=\widehat T \widehat z_0$, with reflection angles $\delta$ and $\varepsilon$ and of drift $\widetilde{\mu}:=\widehat T \widehat \mu$ of angle $\theta$. The angle parameters are in $(0,\pi)$ and are determined by
$$
\tan \beta=\frac{\sin \widehat \beta_1}{\frac{1}{\tan \widehat \beta_0}\sqrt{\frac{\widehat\sigma_{22}}{\widehat\sigma_{11}}} +\cos \widehat\beta_1}
\ ,
\qquad
\tan \theta=\frac{\sin \widehat \beta_1}{\frac{1}{\tan \widehat \theta}\sqrt{\frac{\widehat\sigma_{22}}{\widehat\sigma_{11}}} +\cos \widehat\beta_1} \ ,
$$
$$
\tan\delta
=\frac{\sin \widehat \beta_1}{\frac{1}{\tan \widehat \delta}\sqrt{\frac{\widehat\sigma_{22}}{\widehat\sigma_{11}}} +\cos \widehat\beta_1}
\ ,
\qquad
\tan(\beta-\varepsilon)
=\frac{\sin \widehat \beta_1}{\frac{1}{\tan (\widehat \beta_0 -\widehat \varepsilon)}\sqrt{\frac{\widehat\sigma_{22}}{\widehat\sigma_{11}}} +\cos \widehat\beta_1}.
$$

\nr{Thanks to this linear transformation, we obtain}\nb{The linear transformation $\widehat{T}$ gives} the following relation between the Green's function of $\widehat Z_t$ denoted by $\widehat g( \widehat z)$ for $\widehat z$ inside the cone of angle $\widehat \beta_0$ and the Green's function of $\widetilde Z_t$ denoted by $\widetilde g( \widetilde z)$ for $\widetilde z$ inside the cone of angle $\beta$:
$$
\widehat g( \widehat z)= \frac{1}{\sqrt{\det \widehat\Sigma}} \widetilde g (\widehat T \widehat z).
$$ 
Therefore, the previous formula allows us to extend our results from $\widetilde g$ to $\widehat g$.
\label{rem:gencone}
\end{rem}

The following remark concerns the Martin boundary.
\begin{rem}[Martin boundary]
The Martin boundary associated to this process can be computed from the asymptotics of the Green's function obtained in the previous theorems. The
corresponding harmonic functions can also be obtained \nb{utilizing the}\nr{thanks to} the constants of the dominant terms of the asymptotics. See Section~6 of \cite{ernst_franceschi_asymptotic_2021} which briefly reviews some elements of this theory in a similar context.
\end{rem}

\subsection*{Plan and strategy of proof}
In this article, the results will be first established in a quadrant for any covariance matrix and then will be extended to a cone in the last section. 

The first step in solving our problem is to determine a functional equation relating the Laplace transforms of Green's functions in the quadrant and on the edges (see~Section~\ref{sec:functionalequation}). \nr{Next, we continue these Laplace transforms and study their singularities (see Section~\ref{sec:continuation})}\nb{In Section~\ref{sec:continuation}, we continue to study these Laplace transforms, in particular their singularities}.
Then, we use the inversion Laplace transform formula combined with the functional equation to express the Green's functions as a sum of simple integrals (see Section~\ref{sec:sumofsimpleintegrals}).\nr{ Doetsch's book \cite{doetsch_introduction_1974} is one of the leading references on Laplace transforms.} To determine the asymptotics, we first use
complex analysis to obtain Tauberian results, which links the poles of the Laplace transforms to the asymptotics of the Green's functions. Then, we use a double refinement of the classical saddle-point method: the uniform method of the steepest descent. One of the reference books on this classical approach is \nr{those}\nb{that} of Fedoryuk \cite{fedoryuk_asymptotic_1989}. Appendix~\ref{sec:morse}, which gives a generalized version of the classical Morse Lemma by introducing a parameter dependency, will be useful \nr{to refine}\nb{in understanding the refinement of} the saddle-point method. \nr{Section~\ref{sec:saddlepoint} studies the saddle point, Section~\ref{sec:shift} explains how we shift the integration contour and thus determines}\nb{Section ~\ref{sec:saddlepoint} studies the saddle point and Section~\ref{sec:shift} explains how we shift the integration contour, thus determining} the contribution of the encountered poles to the asymptotics. Section~\ref{sec:neglibible} \nr{shows that some part}\nb{identifies which parts} of the new integration contour are negligible. Section~\ref{sec:asymptoticshiftedcontour}
establishes the contribution of the saddle point to the asymptotics and states the main result.
Section~\ref{sec:asymptoticaxes} studies the asymptotics along axes and Section~\ref{sec:polemeetsaddlepoint} \nb{studies} the asymptotics in the\nr{technical} case where the saddle point \nr{meet}\nb{meets} a pole. Appendix~\ref{app:tech} states a technical result useful to this section.
Finally, Section~\ref{sec:asymptcone} explains how to transfer \nr{to any convex cone the asymptotic results obtained in the quadrant}\nb{the asymptotic results obtained in the quadrant to any convex cone} and thus concludes the proof of Theorems~\ref{thm1}, \ref{thm2} and \ref{thm3}.


\section{Convergence of Laplace transforms and functional equation}
\label{sec:functionalequation}

\subsection*{Transient reflected Brownian motion in a cone}

Let $(Z_t)_{t\geq 0} = (z_0 + \mu t + B_t + RL_t)_{t\geq 0}$ be a (continuous) semimartingale reflected Brownian motion (SRBM) in $\R^2_+$ on a filtered probability space where $\mu =(\mu_1,\mu_2)^\top \in \R^2$ is the drift, $\Sigma$ \nb{is} the covariance matrix associated to the Brownian motion $B$, $R=(r_{ij})_{1\leqslant i,j\leqslant 2}\in \mathbb{R}^{2\times 2}$ \nb{is} the reflection matrix, and $(L_t)_{t\geq 0} = ((L^1_t, L^2_t)^\top)_{t\geq 0}$ \nr{the local times}\nb{is the bivariate local time} on the edges associated to the process. We will assume that $\det(\Sigma) > 0$, i.e. that $\Sigma$ is positive-definite. See Figure~\ref{fig:paramquadrant} to visualize the parameters of this process. We recall the following classical result concerning the existence of such a process, see for example \cite{taylor_existence_1993,Williams-85}.
\begin{prop}[Existence and uniqueness of SRBM] \label{R_gentil}
There exists an SRBM with parameters $(\mu,\Sigma, R)$ if and only if $\Sigma$ is a covariance matrix and $R$ is
completely-$\mathcal{S}$, i.e. 
\begin{equation}
r_{11} > 0,\; r_{22} > 0, \;\,\textnormal{and}\, \;\textnormal{[}\det(R) > 0 \quad\textnormal{or}\quad r_{21},\; r_{12 } > 0\,\textnormal{]}. 
\label{eq:condexist}
\end{equation}
In this case, the SRBM is unique in law and defines a Feller continuous
strong Markov process.
\end{prop}

Condition~\eqref{eq:condexist} will therefore be required throughout the article. The recurrence and transience conditions of those processes are well known, see \cite{hobson_recurrence_1993, williams_recurrence_1985}. In our case, the SRBM will be \nr{systematically transient}\nb{transient} because of the following assumption of positive drift, which we assume to hold throughout the sequel.

\begin{assumption}[Positivity of the drift] \label{drift_positif}
We assume that $\mu_1 > 0$ and $\mu_2 > 0$.
\label{assumptiton:drift}
\end{assumption}

Note that this assumption is equivalent to that made in the introduction: $\theta\in(0,\beta)$. \nb{Under Assumption \ref{assumptiton:drift}, the reflected Brownian motion is transient by \cite{hobson_recurrence_1993}.}

\begin{figure}[hbtp]
\centering
\includegraphics[scale=0.6]{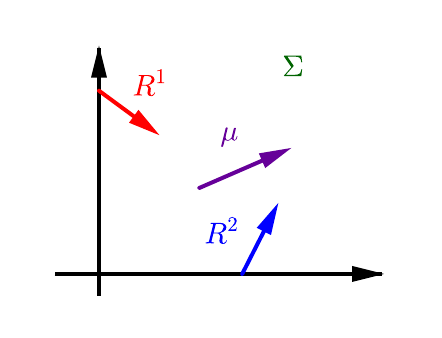}
\caption{SRBM parameters in the quadrant: drift $\mu$, reflection vectors $R^1$ and $R^2$ and covariance matrix $\Sigma$.}
\label{fig:paramquadrant}
\end{figure}

\subsection*{Green's function}

\nr{We are working in the transient case, and we will focus on the Green's functions.
\begin{defi*}[Green's measures and densities]
The Green's measure $G$ inside the quadrant is defined by 
$$G(z_0,A) := \E_{z_0}\left[\int_0^\infty\fc_A(Z_t)dt\right]=\int_A g(z) dz $$
for $z_0 \in \R_+^2$ and $A\subset \mathbb{R}^2$ and admits a density $g$ with respect to the Lebesgue measure. The density $g$ is called the Green's function.
For $i \in \{1,2\}$, we define $H_i$ the Green's measures on the edges of the quadrant which also has densities $h_i$ with respect to the Lebesgue measure, namely
$$H_i(z_0,A) := \E_{z_0}\left[\int_0^\infty\fc_A(Z_t)dL^i_t
\right]
=\int_A h_i(z) dz.$$
The measure $H_1$ has its support on the vertical axis and $H_2$ has its support on the horizontal axis.
\end{defi*}
}
\nb{
As in \eqref{densite:green}, recall that the Green's measure $G$ inside the quadrant is defined by 
$$G(z_0,A) := \E_{z_0}\left[\int_0^\infty\fc_A(Z_t)dt\right]$$
for $z_0 \in \R_+^2$ and $A\subset \mathbb{R}^2$.
For $i \in \{1,2\}$, we define $H_i$ the Green's measures on the edges of the quadrant by
$$H_i(z_0,A) := \E_{z_0}\left[\int_0^\infty\fc_A(Z_t)dL^i_t
\right].$$
The measure $H_1$ has its support on the vertical axis and $H_2$ has its support on the horizontal axis.
}
\nb{
\begin{prop}
Green measures $G$ (resp. $H_1, H_2$) have densities $g$ (resp. $h_1, h_2$) with respect to the two dimensional (resp. one dimensional) Lebesgue measure. 
\end{prop}
We then have $G(z_0,A) =\int_A g(z) dz $ for $A\subset \mathbb{R}^2$, $H_1(z_0,B\times\{0\})=\int_B h_1(z) dz$ for $B\subset \mathbb{R}$ and $H_2(z_0,\{0\}\times C)=\int_C h_2(z) dz$ for $C\subset \mathbb{R}$.
}

In the sequel \nr{one should be kept in mind that in the notations}\nb{it should be kept in mind that in the notations} $g$ and $h_i$ we have omitted the dependence on the starting point $z_0$. 

\begin{proof}
In the recurrent case, Harrison and Williams proved in \cite{harrison_brownian_1987} 
that the invariant measure 
has a density with respect to the Lebesgue measure. The proof in that article extends to the transient case and \nr{justify}\nb{justifies} the existence of a density with respect to the Lebesgue measure for the Green's measures. Indeed, the proof of Lemma $9$ of section $7$ in \cite{harrison_brownian_1987} 
shows that for a Borel set $A$ of Lebesgue measure $0$, we have 
$$\E\left[\int_0^{+\infty}\mathds{1}_{A}(Z_t)dt\right] = 0.$$ 
This is even an equivalence, although we will not need it in the present article. Since the proof does not \nr{requires}\nb{require} the recurrence property, this gives the desired result by the Radon Nikodym theorem. The same argument applies to the densities of $H_i$ for $i=1,2$, see theorem $1$, section $8$ in \cite{harrison_brownian_1987}. 
\end{proof}

\nb{In the following, we denote $\mathbb{R}_+=[0,\infty)$ and $\mathbb{R}^*_+=(0,\infty)$.}

\begin{rem}[Partial differential equation]
Let us denote $\mathcal{L} = \frac{1}{2}\nabla\cdot\Sigma\nabla + \mu\cdot\nabla$ the generator of the SRBM inside the quadrant and $\mathcal{L}^* = \frac{1}{2}\nabla\cdot\Sigma\nabla - \mu\cdot\nabla$ its dual operator. Then, the Green's function $g$ satisfies
$$
\mathcal{L}^*g = -\delta_{z_0} 
$$ 
in the sense of distributions $\mathcal{D'}((\R^*_+)^2)$. 

Let us define the matrix $R^*=2\Sigma-R  \ \textnormal{diag}(R)^{-1}\textnormal{diag}(\Sigma)$. We denote $R_1^*$ and $R_2^*$ the two columns of $R^*$. Then, the following boundary conditions hold
$$
\begin{cases}
\partial_{R_1^*}g(z)-2\mu_1 g(z)=0 \text{ for } z\in\{0\}\times \mathbb{R_+}
\\
\partial_{R_2^*}g(z)-2\mu_2 g(z)=0 \text{ for } z\in\mathbb{R_+}\times \{0\}
\end{cases}
$$
where $\partial_{R_i^*}=R_i^* \cdot\nabla$.
\end{rem}
\begin{proof}[Sketch of proof of the remark]
The partial differential equation of the Green's function and its boundary conditions are derived from the forward equation of the transition kernel established in \cite{HaRe-81}, see Equation~(8.3). However, we \nr{provides}\nb{provide} here a direct elementary proof of the fact that $\mathcal{L}^*g = -\delta_{z_0} $.
Let $\phi \in C^\infty_c((\R_+^*)^2)$. \nr{We apply Ito's formula and we take the expectation}\nb{Applying Ito’s formula and taking expectations}, we obtain
$$\E[\phi(Z_t)] = \phi(z_0) + \E\left[\int_0^{t}\mathcal{L}\phi(Z_s)ds\right].$$
One may remark that there are no boundary terms \nr{since $\phi$ cancel}\nb{since the functions $\phi$ will cancel} on {a neighborhood of} the boundaries.
Since we are in the transient case and since $\phi$ is bounded, the left term converges to $0$ as $t$ tends towards infinity by the dominated convergence theorem. Since successive derivatives of $\phi$ are bounded, $\mathcal{L}\varphi(a, b)$ is bounded by an exponential function up to a multiplication constant. \nr{Thanks to convergence domain of the Laplace transform} \nb{Due to the convergence domain of the Laplace transform} (see Proposition~\ref{propconvergence} \nr{bellow}\nb{below}), we obtain by dominated convergence that $
 \phi(z_0) = - \E\left[\int_0^{+\infty}\mathcal{L}\phi(Z_s)ds\right] = - \int_{\R_+^2}\mathcal{L}\phi(z)g(z) dz$ which implies that $\mathcal{L}^*g = -\delta_{z_0} $.
\end{proof}
Furthermore, it is preferable to have continuity of the Green's function \nr{to talk about their asymptotic}\nb{when investing their asymptotic} behaviour. This is the content of the following comment.
\begin{rem}[Smoothness of Green's functions]
By the strictly elliptic regularity theorem \nb{(see for instance the Hypoelliptic theorem 5.1 in \cite{HormanderHypoelliptic})}, we may deduce from $\mathcal{L}^*g = -\delta_{z_0}$ that the density $g$ 
has a $\mathcal{C}^\infty$ version on \nr{$(\R^*_+)^2\backslash\{z_0\}$ }\nb{$(0,+\infty)^2\backslash\{ z_0\}$}. We will not go into more detail here about the proof of this result. In the remainder of this article, we will assume that \nr{this property is true and that} $g$ is continuous on \nr{$(\mathbb{R}_+^2)^*\setminus \{z_0\}$}\nb{$[0, +\infty)^2\setminus \{0, z_0\}$}.
\end{rem}

\subsection*{Laplace transform and functional equation}

\begin{defi}[Laplace transform of Green's functions]
For $(x, y) \in \C^2$ we define the Laplace transforms of the Green's measures by
$$\varphi(x, y) := \E_{z_0}\left[\int_0^\infty e^{(x, y)\cdot Z_t}dt\right]
=\int_{\mathbb{R}_+^2} e^{(x,y)\cdot z}g(z) dz
$$ 
and
$$
\varphi_1(y):= \E_{z_0}\left[\int_0^\infty e^{(x,y)\cdot Z_t}dL^1_t\right]
=\int_{\mathbb{R}_+} e^{y b}h_1(b) db,
\quad 
\varphi_2(x):= \E_{z_0}\left[\int_0^\infty e^{(x,y)\cdot Z_t}dL^2_t\right]
=\int_{\mathbb{R}_+} e^{x a}h_2(a) da
.$$ 
\end{defi}
Let us remark that $\varphi_1$ does not depend on $x$ and  $\varphi_2$ does not depend on $y$. 
\nr{One has to remember}\nb{Recall} the dependence on the starting point $z_0$ even though it is omitted in the notation. 

Since Green's measures are not probability measures, the convergence of their Laplace transforms are not guaranteed. For example, $\varphi(0)$ is not finite. \nr{Convergence domains had already been studied in}\nb{Convergence domains for Laplace transforms of Green’s functions have been studied in}~\cite{franceschi_green_2021} but we need stronger results. The following proposition establishes the convergence when the real parts of $x$ and $y$ are negative.
\begin{prop}[Convergence of the Laplace transform] \label{propconvergence}
Assuming that $\mu_1>0$ and $\mu_2 > 0,$ 
\begin{itemize}
    \item $\varphi_1(y)$ converges (at least) on $y \in \{y \in \C, \Re(y) < 0\}$
    \item $\varphi_2(x)$ converges (at least) on $x \in \{x \in \C, \Re(x) < 0\}$
    \item $\varphi(x, y)$  converges (at least) on $(x, y) \in \{(x,y) \in \C^2, \Re(x) < 0\; \textnormal{and}\; \Re(y) < 0\}. $
\end{itemize}\end{prop}
Before proving this proposition, we state the functional equation that will be central in this article.
First, we need to define for $(x, y) \in \C^2$ the following polynomials 
$$\begin{cases}
        \gamma(x, y) = \frac{1}{2}(x, y)\cdot\Sigma(x, y) + (x, y)\cdot\mu = \frac{1}{2}(\sigma_{11}x^2 + 2\sigma_{12}xy + \sigma_{22}y^2) + \mu_1 x + \mu_2 y\\
        \gamma_1(x, y) = R^1\cdot (x, y) = r_{11}x + r_{21}y \\ 
        \gamma_2(x, y) = R^2\cdot (x, y) = r_{12}x + r_{22}y
    \end{cases}$$
    where $R^1, R^2$ are the two columns of the reflection matrix $R$. The polynomial $\gamma$ is called the kernel.
\begin{prop}[Functional equation]
If $\Re(x) < 0$ and $\Re(y) < 0$, then
\begin{equation}\label{Equation fonctionnelle}
    - \gamma(x, y)\varphi(x, y) = \gamma_1(x, y)\varphi_1(y) + \gamma_2(x, y)\varphi_2(x) + e^{(x, y)\cdot z_0}.
\end{equation}
\label{prop:equationfonct}
\end{prop}
The proofs of these two proposition \nr{are deeply linked. So we'll be gathering their proofs.}\nb{are directly related, so we will prove both together.}
\begin{proof}[Proof of Propositions~\ref{propconvergence} and \ref{prop:equationfonct}]
The main idea of the proof is to take the expectation of Itô's formula applied to the SRBM and to use a sign argument to justify the limit when $t \rightarrow +\infty$. The beginning of the proof is inspired \nr{of}\nb{by} the Proposition $5$ of \cite{franceschi_green_2021}. 

Letting $(x, y) \in (\R_-^*)^2$, Itô's formula applied to $f(z) := e^{(x, y)\cdot z}$ gives 
\begin{align}
        f(Z_t) - f(z_0) &= \int_0^t\nabla f(Z_s).dB_s + \int_0^t\mathcal{L}f(Z_s)ds + \sum_{i=1}^2\int_0^t R_i\cdot\nabla f(Z_s)dL^i_s\label{Itogenerateur}\\
        &= \int_0^t \nabla f(Z_s).dB_s + \gamma(x, y)\int_0^te^{(x, y)\cdot Z_s}ds + \sum_{i=1}^2\gamma_i(x, y)\int_0^t  e^{(x, y)\cdot Z_s}dL^i_s \label{AH}
\end{align}
where $\mathcal{L} = \frac{1}{2}\nabla\cdot\Sigma\nabla + \mu\cdot\nabla$ is the generator of the Brownian motion.
Since $(x,y) \in (\R_-^*)^2$, the integral $\int_0^t\nabla f(Z_s).dB_s$ 
is a martingale (its quadratic variation is bounded by $C.t$ for a constant $C > 0$) and its expectation cancels out. 
Therefore, 
\begin{multline}
    \E_{z_0}\left[e^{(x, y)\cdot Z_t}\right] - e^{(x, y)\cdot z_0} - \gamma(x, y)\E_{z_0}\left[\int_0^t e^{(x, y)\cdot Z_s}ds\right]
    \\ = \E_{z_0}\left[\gamma_1(x, y)\int_0^t e^{(x, y)\cdot Z_s}dL^1_s +\gamma_2(x, y)\int_0^t e^{(x, y)\cdot Z_s}dL^2_s\right].
    \label{youpi}
\end{multline}
The expectations in the left-hand side of the previous equation are finite because for $(x, y) \in (\R_-^*)^2$, the first expectation is bounded by $1$ and the second expectation is bounded by $t$. This implies that the expectation of the right-hand side is also finite.

The aim now is to take the limit of \eqref{youpi} when $t$ goes to infinity to show the finiteness of the Laplace transforms and the functional equation.
First, since
$(x, y)\in(\R^*_-)^2$ and $\nor{Z_t} \underset{t \to \infty}{\longrightarrow}+\infty$ a.s., the expectation $\E_x\left[e^{(x, y)\cdot Z_t}\right]$ converges toward $0$ when $t \to \infty$ by the dominated convergence theorem. 
Secondly, by the monotone convergence theorem, the expectation $\E_{z_0}\left[\int_0^t e^{(x, y)\cdot Z_s}ds\right]$ converges in $[0,\infty]$ to $\varphi (x,y)=\E_{z_0}\left[\int_0^\infty e^{(x, y)\cdot Z_s}ds\right] $. 

\nr{Let us assume for a moment that it is possible to choose $(x_0,y_0)\in(\R^*_-)^2$ such that $\gamma(x_0,y_0)<0$, $\gamma_1(x_0,y_0)<0$ and $\gamma_2(x_0,y_0)<0$. We use a proof by contradiction assuming that we have $\E_{z_0}\left[\int_0^\infty e^{(x_0, y_0)\cdot Z_s}ds\right]  =  + \infty$}\nb{We now prove by contradiction that $\phi(x_0, y_0)$ is finite. For the sake of contradiction, let us assume that it is possible to choose $(x_0,y_0)\in(\R^*_-)^2$ such that $\gamma(x_0,y_0)<0$, $\gamma_1(x_0,y_0)<0$ and $\gamma_2(x_0,y_0)<0$ and $\E_{z_0}\left[\int_0^\infty e^{(x_0, y_0)\cdot Z_s}ds\right]  =  + \infty$}.
Since $\gamma(x_0,y_0)<0$, the left-hand side of \eqref{youpi} will be positive for \nr{$t$ }large enough \nb{$t$}. But, since $\gamma_1(x_0,y_0)<0$ and $\gamma_2(x_0,y_0)<0$, the right-hand side of \eqref{youpi} is always negative. We have thus obtained a contradiction, allowing us to conclude that $\varphi (x_0,y_0)=\E_{z_0}\left[\int_0^\infty e^{(x_0, y_0)\cdot Z_s}ds\right] $ is finite. Hence, the limit of the right-hand side of \eqref{youpi} is also finite and converges by the monotone convergence theorem to $\gamma_1(x_0,y_0)\phi_1(y_0)+\gamma_2(x_0,y_0)\phi_2(x_0)$. We deduce that $\phi_1(y_0)$ and $\phi_2(x_0)$ are also finite and that the functional equation \eqref{Equation fonctionnelle} is satisfied for $(x_0,y_0)$.
This implies that for all $x$ and $y$ in $\mathbb{C}$ such that $\Re x <x_0$ and $\Re y <y_0$ the Laplace transforms $\varphi (x,y)$, $\phi_1(y)$ and $\phi_2(x)$ are finite and the functional equation \eqref{Equation fonctionnelle} is satisfied by taking the limit of \eqref{youpi} when $t\to\infty$. 

All that remains is to show that we can always choose $x_0$ and $y_0$ \emph{as close to $0$ as we like}, such that $(x_0, y_0) \in (\R_-^*)^2$, $\gamma(x_0,y_0)<0$, $\gamma_1(x_0,y_0)<0$ and $\gamma_2(x_0,y_0)<0$ and the proof of Propositions~\ref{propconvergence} and~\ref{prop:equationfonct} will be complete. Let us denote \nb{by} $\mathcal{E}$ the ellipse of equation $\gamma(x, y)=0$. One may observe that the interior of the ellipse $\mathcal{E}$ defined by $\gamma(x, y)< 0$ contains a neighbourhood of $0$ intersecting $(\R^*_-)^2$ by Assumption~\ref{assumptiton:drift} on the positivity of the drift. Indeed, the drift is an external normal to the ellipse at $(0,0)$. 
We consider two cases coming from the existence condition of the process~\eqref{eq:condexist}.
\nr{First case,}\nb{The first case is given by} $r_{11}>0$, $r_{22}>0$, $r_{12} > 0$ and $r_{21} > 0$ (see Figure~\ref{r_positifs}). 
In this case, one may see directly see that $\gamma_1(x, y) < 0$ and $\gamma_2(x, y) <0$ on $(\R^*_-)^2$. It is therefore easy to pick $(x_0,y_0)$ close enough to $(0,0)$ which satisfies the required conditions. \nr{Second case}\nb{The second case is given by} $r_{11}>0$, $r_{22}>0$ and $\det(R) > 0$ (see Figure~\ref{r_negatifs}). 
In this case, the cone defined by $\gamma_1 < 0$ and 
$\gamma_2 < 0$ has a non-empty intersection with $(\R^*_-)^2$. Hence, \nr{one}\nb{we} can still choose $(x_0, y_0)$ as close as we want to $(0,0)$ inside the desired cone and the ellipse $\mathcal{E}$.
\end{proof}
\begin{figure}[hbtp]
\centering
     \begin{subfigure}[b]{0.45\textwidth} 
\includegraphics[scale=3]{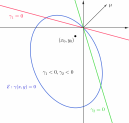}
\caption{Case $r_{11}>0$, $r_{22}>0$, $r_{12}>0$ and $r_{21} > 0$.}
\label{r_positifs}
     \end{subfigure}
     \hfill
     \centering
     \begin{subfigure}[b]{0.45\textwidth}
\includegraphics[scale=3]{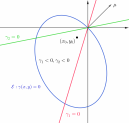}
\caption{Case $r_{11}>0$, $r_{22}>0$ and $\det R >0$.}
\label{r_negatifs}
     \end{subfigure}
\caption{\nr{For $(x, y)\in\mathbb{R}^2$,} Illustration of the domain where $\gamma_1 < 0$ and $\gamma_2 < 0$.}
\label{fig:rrr}
\end{figure}

\nb{\begin{rem}[Dependency on the initial state]
The main difference compared to the recurrent case~\cite{franceschi_asymptotic_2016} comes from the additional term $e^{(x,y)\cdot z_0}$ in the functional equation.
With the exception of this one term, it is coherent that Green's functions in the transient case have similar asymptotic behaviors that those of the stationary densities in the recurrent case.
\end{rem}}

The following \nr{lemma}\nb{proposition} follows from the functional equation and states that the boundary Green's densities $h_1$ and $h_2$ are equal, up to some constant, to the bivariate Green's function $g$ on the axes. 
\begin{prop}[Green's densities on the boundaries]
The Green's density $g$ is related to the boundary Green's densities $h_i$ by the formulas
$$
r_{11}h_1(b)=\frac{\sigma_{11}}{2} g(0,b)
\quad\text{and}\quad
r_{22}h_2(a)=\frac{\sigma_{22}}{2} g(a,0).
$$
\label{prop:linkhg}
\end{prop}
\begin{proof}
The initial value formula of a Laplace transform gives
$$
x\phi(x,y) \underset{x\to -\infty}{\longrightarrow}
-\int_0^\infty e^{yb} g(0,b)db.
$$
Therefore, by dividing the functional equation~\eqref{Equation fonctionnelle} by $x$ and taking the limit when $x$ tends to infinity, we obtain
$$
\frac{1}{2}\sigma_{11} \int_0^\infty e^{yb} g(0,b)db = r_{11}\phi_1(y)=r_{11} \int_0^\infty e^{yb} h_1(b)db
$$
which implies the result.
\end{proof}

\section{Continuation and properties of \texorpdfstring{$\phi_1(x)$}{phi1(x)} and \texorpdfstring{$\phi_2(y)$}{phi2(y)}} 
\label{sec:continuation}

The first step of \nr{this analytic}\nb{the analytical} approach \nb{\cite{FIM17,ernst_franceschi_asymptotic_2021}} is to study the kernel.
\begin{lemma}[Kernel study]
\label{alleq} 
\begin{itemize} 
\item[(i)] Equation $\gamma(x,y)=0$ determines an algebraic function
$Y(x)$ [resp. $X(y)$] with two branches 
$$ Y^{\pm } (x)= \frac{1}{\sigma_{22 }}\Big(-\sigma_{12} x -\mu_2 \pm \sqrt{(\sigma_{12}^2 -\sigma_{11} \sigma_{22}) x^2 + 2  (\mu_2 \sigma_{12}- \mu_1 \sigma_{22})x + \mu_2^2} \Big).$$
The function $Y(x)$ [resp. $X(y)$] has two branching points 
$x_{min}$ and $x_{max}$ [resp. $y_{min}$ and $y_{max}$] given by \begin{equation*}
x_{min} = \frac{\mu_2\sigma_{12} - \mu_1\sigma_{22} - \sqrt{D_1}}{\det(\Sigma)}, \quad x_{max} = \frac{\mu_2\sigma_{12} - \mu_1\sigma_{22} + \sqrt{D_1}}{\det(\Sigma)},
\end{equation*}
\begin{equation*}
y_{min} = \frac{\mu_1\sigma_{12} - \mu_2\sigma_{11} - \sqrt{D_2}}{\det(\Sigma)}, \quad y_{max} = \frac{\mu_1\sigma_{12} + \mu_2\sigma_{11} - \sqrt{D_2}}{\det(\Sigma)},
\end{equation*}
where $D_1 = (\mu_2\sigma_{12} - \mu_1\sigma_{22})^2 + \mu_2^2\det(\Sigma) $ and $D_2 = (\mu_1\sigma_{12} - \mu_2\sigma_{11})^2 + \mu_1^2\det(\Sigma)$.
  Both of them are real and 
 $x_{min}<0<x_{max}$ [resp. $y_{min}<y_{max}$].  
    The branches of  $Y(x)$ [resp. $X(y)$]  take real values 
    if and only if $ x\in [x_{min}, x_{max}]$
 [resp. $y \in [y_{min}, y_{max}]$]. 
         Furthermore $Y^{-}(0)= -\frac{2\mu_2}{\sigma_{22}}<0$, 
$Y^-(x_{max})<0$,  $Y^{+}(0)=0$, $Y^+(x_{max})<0$. See Figure~\ref{fig:pole}.
 \item[(ii)]  For any $u \in {\mathbb{R}}$  
 $${\rm Re} Y^{\pm}(u+iv)=
 \frac{1}{\sigma_{22 }}\Big(-\sigma_{12} u -\mu_2 \pm \frac{1}{\sqrt{2}} 
   \sqrt{  (u-x_{min}) (x_{max}-u)+v^2 +| (u+iv - x_{min})(x_{max}-u-iv ) |      } \Big).$$  
 \item[(iii)]    Let  $\delta=\infty$ if $\sigma_{12}\geq 0$
           and $\delta= -\mu_2/\sigma_{12} -x_{max}>0$   if 
     $\sigma_{12}<0$. Then for some $\epsilon>0$ small enough 
     $${\rm Re} Y^{-}(u+iv)<0 \ \ \ \text{for}\ u \in ] -\epsilon, \   x_{max}+\delta [,\ \  v \in {\mathbb{R}}.$$ 
\end{itemize} 
\end{lemma} 

\begin{proof} Points (i) and (ii) follow from elementary considerations.
   The fact that $Y^{+}(x_{max})<0$ implies the inequality $-\sigma_{12}x_{max}- \mu_2<0$, so that $\delta>0$. Furthermore, by (ii) 
  ${\rm Re}\, Y^{-}(u+iv) \leq {\rm Re} Y^{-}(u)$ which is strictly negative 
    for $u \in ]-\epsilon, x_{max} +\delta[$
 by the analysis \nr{made }in (i). 
 \end{proof}

\begin{lemma}[Continuation of the Laplace transform] 
\label{conti}
     Function $\phi_2(x)$  can be meromorphically 
continued to the (cut) domain 
\begin{equation}
\{x = u+iv \mid  u < x_{max}+\delta, v \in {\mathbb{R}} \}
\setminus [x_{max}, x_{max}+\delta]
\label{eq:domain}
\end{equation}
by the formula : 
\begin{equation}
    \label{zio}
\phi_2(x)=  \frac{ -\gamma_1(x,Y^-(x)) \phi_1(Y^-(x)) -\exp \big(a_0 x + b_0 Y^{-}(x) \big)}{\gamma_2(x, Y^{-}(x))}
\end{equation} 
\end{lemma} 
\nb{where $z_0 = (a_0, b_0)$}. A symmetric continuation formula holds for $\phi_1$.
\begin{proof}
  By Lemma \ref{alleq} (iii)  
  for any $x=u+iv$ with $u  \in  ] -\epsilon, 0[$  
     the following equation holds 
$$ - \gamma(x, Y^{-}(x)) \phi(x, Y^{-}(x)) =  \gamma_1(x,Y^-(x)) \phi_1(Y^-(x))+ \gamma_2(x,Y^-(x)) \phi_2(x) + \exp(a_0 x + b_0 Y^{-}(x)).$$
Since $\gamma(x, Y^{-}(x))=0$, the statement follows.
\end{proof} 

\nb{
We now define 
\begin{equation}  \label{eq:defx*y**}
x^* = 2\frac{\mu_2\frac{r_{12}}{r_{22}} - \mu_1}{\sigma_{11} - 2\sigma_{12}\frac{r_{12}}{r_{22}} + \sigma_{22}\left(\frac{r_{12}}{r_{22}}\right)^2} \quad \textnormal{and} \quad y^{**} = 2\frac{\mu_1\frac{r_{21}}{r_{11}} - \mu_2}{\sigma_{11}\left(\frac{r_{21}}{r_{11}}\right)^2 - 2\sigma_{12}\frac{r_{21}}{r_{11}} + \sigma_{22}}.
\end{equation}
}
\begin{prop}[Poles of the Laplace transform, necessary condition] 
\begin{itemize} 
\item[(i)] $x=0$ is not a pole of $\phi_2(x)$, and $\phi_2(0) = \E[L^2_\infty] < +\infty$. The local time spent by the process on the horizontal axis is finite.


\item[(ii)]\nr{If $x^{*}$}\nb{If $x$} is a pole of $\phi_2(x)$ in the domain~\eqref{eq:domain},
then \nb{$x = x^*$ and} $(x^*, Y^{-}(x^*))$  is a unique non-zero solution of the system of two equations 
\begin{equation}
    \label{za}
 \gamma(x,y)=0, \  \  \   \ \gamma_2(x,y)= r_{12}x+ r_{22}y=0.
\end{equation}  
\nr{   
Moreover, $x^*$ is real and belongs to $]0, x_{max}[$.
Furthermore, this solution exists only if 
   $$x_{max}r_{12} + Y^{\pm}(x_{max})r_{22} > 0.$$
}
\nb{In this case,  $x_{max}r_{12} + Y^{\pm}(x_{max})r_{22} > 0,$ $x^*$ is real and belongs to $(0, x_{max})$.}

\item[(iii)]   \nr{If $y^{**}$}\nb{If $y$} is a pole of $\phi_1(y)$, 
then \nb{
$y = y^{**}$ and $(X^{-}(y^{**}), y^{**})$  is a unique non-zero solution of the system of two equations 
}
\begin{equation}
    \label{za2}
 \gamma(x,y)=0, \  \  \   \  \gamma_1(x,y)=r_{11}x+ r_{21}y=0.
\end{equation}  
   \nr{Moreover, $y^{**}$ is real and belongs to $]0, y_{max}[$.
Furthermore, this solution exists only if  
   $$y_{max}r_{21} + X^\pm(y_{max})r_{11} > 0.$$}
\nb{
In this case, $y_{max}r_{21} + X^\pm(y_{max})r_{11} > 0$, $y^{**}$ is real and belongs to $(0, y_{max})$.
}
\end{itemize} 
\nr{
When these solutions exist we have
\begin{equation}
x^* = 2\frac{\mu_2\frac{r_{12}}{r_{22}} - \mu_1}{\sigma_{11} - 2\sigma_{12}\frac{r_{12}}{r_{22}} + \sigma_{22}\left(\frac{r_{12}}{r_{22}}\right)^2} 
 \quad \text{and} \quad
  y^{**} = 2\frac{\mu_1\frac{r_{21}}{r_{11}} - \mu_2}{\sigma_{11}\left(\frac{r_{21}}{r_{11}}\right)^2 - 2\sigma_{12}\frac{r_{21}}{r_{11}} + \sigma_{22}}.
  \label{eq:defx*y**}
\end{equation}
}
\label{pole} 
\end{prop} 
Finally, we define 
$$y^*:=Y^+(x^{*})
\quad\text{and}\quad
x^{**}:=X^+(y^*).$$ 
\nr{See Figure~\ref{fig:pole} to visualize all these points}\nb{See Figure~\ref{fig:pole} below, which depicts the poles $x^*$ and $y^{**}$ when they are both poles}.

\begin{figure}[ht]
     \includegraphics[scale=2]{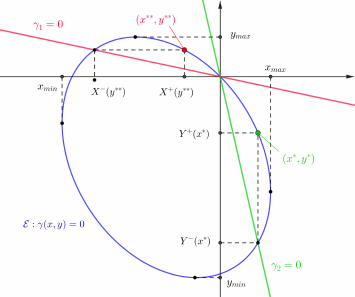} 
     \caption{In the real plane $(x, y)$, graphical representation of poles $x^{*}$ and $y^{**}$ when both exist.}
\label{fig:pole}
\end{figure}

\begin{proof} 
(i) The observation that $\gamma_2(0, Y^{-}(0))=r_{22} \times Y^{-}(0) \ne 0$
   implies the first statement. 

(ii) If \nr{$x^*$}\nb{$x$} is a pole of \nr{$\phi_2(x)$}\nb{$\phi_2$},  then \nr{$(x^*, Y^{-}(x^*))$}\nb{$(x, Y^{-}(x))$}
 should be a solution of the system (\ref{za}) above by the continuation formula~\eqref{zio} and the continuity of $\varphi_{1}$ [resp. $\varphi_2$] on $\{\Re y \leq 0\}$ [resp. $\{\Re x \leq 0\}$]. This system has one solution 
  $(0,0)$ and the second one $(x^\circ,y^\circ)$, which is necessarily real. Then  
    $x^\circ \in  [x_{min}, x_{max}]$ and $y^\circ$ is either $Y^{-}(x^\circ)$ or $Y^{+}(x^\circ)$. But $x^\circ$ can be a pole of $\phi_2(x)$, if only it is  
     within  $]0, x_{max}]$ and $y^\circ=Y^{-}(x^\circ)$. This last condition implies 
      $\frac{r_{12}}{r_{22}} >\frac{ - Y^{\pm } (x_{\max}) }{x_{max}}$.    
       \end{proof} 

\begin{prop}[Poles of the Laplace transforms, sufficient condition]\label{poles_expr}
\nr{The pole $x^*$ (resp. $y^{**}$) of $\varphi_2$ (resp. $\varphi_1$) exists}\nb{The point $x^*$ (resp. $y^{**}$) is a pole of $\varphi_2$ (resp. $\varphi_1$)} if (and only if) $x_{max}r_{12} + Y^{\pm}(x_{max})r_{22} > 0$ (resp. $y_{max}r_{21} + X^\pm(y_{max})r_{11} > 0$). 
\end{prop}
\begin{proof}
\nr{The conditions of the previous proposition are necessary.}\nb{The inequalities above are necessary by the previous proposition.} The next two lemmas prove \nr{the }sufficiency. In those, we denote the dependence of Laplace transforms with the initial condition $z_0$ by $\varphi_1^{z_0}, \varphi_2^{z_0}$ instead of $\varphi_1, \varphi_2$. The proof is done for $x^*$, but is of course symmetrical for $y^{**}$.
\end{proof}

\begin{lemma}[Existence of the pole for a starting point]\label{existencepole}
If $x_{max}r_{12} + Y^{\pm}(x_{max})r_{22} > 0$, there exists $z_0 \in \R_+^2$ such that $x^*$ is a pole of $\varphi_2^{z_0}$.
\end{lemma}

\begin{proof}
The denominator of the continuation formula \eqref{zio} vanishes since we assume that $x_{max}r_{12} + Y^{\pm}(x_{max})r_{22} > 0$. We are looking for a $z_0$ such that the numerator doesn't vanish at $x^*$, which will imply that $z_0$ is a pole of $\phi_2$. If $\gamma_1(x^*, Y^-(x^*)) \geq 0$, this is obvious \nr{thanks}\nb{due} to the exponential term and \nr{a sign argue}\nb{since $\gamma_1(x^*, Y^-(x^*)\phi_1(Y^-(x)) \geq 0$}. We suppose now that $-C := \gamma_1(x^*, Y^-(x^*)) < 0$. 
\nr{We make a proof by contradiction assuming that}\nb{We proceed with a proof by contradiction. For the sake of contradiction, assume that}
\begin{equation}\label{absurde}
\forall z_0=(a_0, b_0) \in \R_+^2, \quad -C\varphi_1^{(a_0, b_0)}(Y^-(x^*)) + e^{a_0x^* + b_0Y^-(x^*)} = 0.
\end{equation}
Let $T$ be the stopping time defined by the first hitting time of the axis $\{x = 0\}$, i.e. $T = \inf\{t \geq 0, Z^{1}_t = 0\}$ with $Z = (Z^1, Z^2)$. (It is possible that $T = +\infty$).
Firstly, since the Stieltjes measure $dL^1$ is supported by $\{Z^1 = 0\}$ and since $Z$ is a strong Markov process, for a starting point $z_0=(a_0,b_0)$ we have:

\begin{align}
\varphi_1^{(a_0, b_0)}(Y^-(x^*)) &= \E_{(a_0, b_0)}\left[\int_T^{+\infty}e^{Z^2_t.Y^-(x^*)}dL^1_t\fc_{T<+\infty}\right] \\
&= \E_{(a_0, b_0)}\left[\E_{Z_T}\left[\int_0^{+\infty}e^{Z^2_t.Y^-(x^*)}dL^1_t\right]\fc_{T<+\infty}\right] \\&= \E_{(a_0, b_0)}\left[\varphi_1^{(0, Z^2_T)}(Y^-(x^*))\fc_{T<+\infty}\right].
\end{align}

Conditioning by the value of $Z^2_T$, using \eqref{absurde} and $Y^-(x^*) \leq 0$, we get :

\begin{align}
\varphi_1^{(a_0, b_0)}(Y^-(x^*)) &= \int_{0}^{+\infty}\varphi_1^{(0,b)}(Y^-(x^*))\P_{(a_0, b_0)}({T<+\infty}, Z^2_T = db)\\
&= \int_0^{+\infty} \frac{1}{C}e^{0.x^* + bY^-(x^*)}\P_{(a_0, b_0)}(T < +\infty, Z^2_T = db) \leq  \frac{1}{C}\P_{(a_0, b_0)}(T < +\infty) \leq \frac{1}{C}.
\end{align}
But, $(a_0, b_0)$ can be chosen such that $e^{a_0x^* + b_0Y^-(x^*)}$ is as \nr{huge}\nb{large} as desired because $x^* > 0$. This is in contradiction with (\ref{absurde}).
\end{proof}

\begin{lemma}[Existence of a pole for all starting points]
If $x^*$ is a pole of $\varphi_2^{z_0}$ for some $z_0 \in \R^2_+$, then $x^*$ is a pole of $\varphi_2^{z_0'}$ for every $z_0' \in \R^2_+$.
\label{lem:poleallz0}
\end{lemma}
\nr{The proof of this lemma is postponed bellow  Proposition~\ref{prop:asympth} since it needs this proposition to be established.}\nb{The proof of Lemma \ref{lem:poleallz0} requires Proposition~\ref{prop:asympth} to be established and is therefore postponed until after Proposition~\ref{prop:asympth}.}

\begin{lemma}[Nature of the branching point of $\phi_2$]
\nr{Let}\nb{Letting} $x \to x_{max}$ with $x < x_{max}$, we have 
    \begin{itemize}
        \item If $\gamma_2(x_{max}, Y^-(x_{max})) = 0$, i.e. $x^*=x_{max}$, then
        $$\varphi_2(x) = \frac{C}{\sqrt{x_{max} - x} }+ O(1)$$ for a constant $C > 0.$ 
        \item If $\gamma_2(x_{max}, Y^-(x_{max})) \neq 0$, then 
        $$\varphi_2(x) = C_1 + C_2\sqrt{x_{max} - x } + O(x_{max} - x)$$ for constants $C_1 \in\mathbb{R}$ and $C_2 >0$.
    \end{itemize}
    \label{lem:branchementphi2}
\end{lemma}

\begin{proof}
    \nr{Thanks to}\nb{By} Lemma~\ref{alleq}, $Y^-$ can be written as $Y^-(x) = Y^-(x_{max}) - c\sqrt{x_{max} - x} + O(x_{max} - x)$ where $c > 0$. \nr{Let's carry out}\nb{We proceed to calculate} an elementary asymptotic expansion of the quotient of the continuation formula \eqref{conti}. \nr{First of all}\nb{Firstly}, \begin{align*}
        \frac{1}{\gamma_2(x, Y^-(x))} &= \frac{1}{\gamma_2(x_{max}, Y^-(x_{max})) - r_{22}c\sqrt{x_{max} - x} + O(x_{max} - x)} \\
        &= 
    \begin{cases}
        \frac{-1}{r_{22}c\sqrt{x_{max} - x}} &\textnormal{if} \; \gamma_2(x_{max}, Y^-(x_{max}) = 0,
        \\
        \frac{1}{\gamma_2(x_{max}, Y^-(x_{max}))}\left(1 + \frac{r_{22}c \sqrt{x_{max} - x} }{\gamma_2(x_{max}, Y^-(x_{max}))}+ O(x_{max} - x)\right) &\textnormal{if} \; \gamma_2(x_{max}, Y^-(x_{max})) \neq 0.
    \end{cases}
    \end{align*}
Secondly, for the numerator, 
\begin{multline}
\gamma_1(x, Y^-(x))\varphi_1(Y^-(x)) + e^{a_0x + b_0Y^-(x)} = 
\\ \left(\gamma_1(x_{max}, Y^-(x_{max})) - r_{21}c\sqrt{x_{max} - x} + O(x_{max} - x)\right)
\\ \times
\left(\varphi_1(Y^-(x_{max})) - c\varphi_1'(Y^-(x_{max}))\sqrt{x_{max} - x} + O(x_{max} - x) \right)
\\
 + e^{a_0x_{max} + b_0Y^-(x_{max})}(1 - cb_0\sqrt{x_{max} - x} + O(x_{max} - x))
\end{multline}
Combining the two asymptotic expansions, we obtain the desired formula with 
$$
C = \frac{\gamma_1(x_{max}, Y^-(x_{max}))\varphi_1(Y^-(x_{max})) + e^{a_0x_{max} + b_0Y^-(x_{max})}}{r_{22}c}$$

and 
\begin{multline*}
C_2 = \frac{1}{\gamma_2(x_{max}, Y^-(x_{max}))} \Big[r_{21}c\varphi_1(Y^-(x_{max})) + c\gamma_1(x_{max}, Y^{-}(x_{max}))\varphi'_1(Y^-(x_{max})) + cb_0e^{a_0x_{max} + b_0Y^-(x_{max})}\\
- \frac{r_{22}c}{\gamma_2(x_{max}, Y^-(x_{max}))}\left(\gamma_1(x_{max}, Y^-(x_{max}))\varphi_1(Y^-(x_{max})) + e^{a_0x_{max} + b_0Y^-(x_{max})} \right)\Big].
\end{multline*}
\end{proof}

The following proposition states the asymptotics of the Green's functions $h_1$ and $h_2$ on the boundaries. We note that we obtain the same asymptotics as in Theorem~\ref{thm2} and~\ref{thm5} with $\alpha \to 0$, which is consistent with the link made between $h_1$, $h_2$ and $g$ in Proposition~\ref{prop:linkhg}. 

\begin{prop}[Asymptotics of the Green's functions on the boundary $h_1$ and $h_2$]
In this \nb{proposition} \nr{lemma} we denote by $c$ a constant which is allowed to vary from one line to the next.
\begin{enumerate}
\item Suppose that we have a pole $x^* \in ]0, x_{max}[$ for $\varphi_2$. Then, the Green's function $h_2$ has the following asymptotics 
$$h_2(u) \underset{u\to\infty}{\sim} c e^{-x^* u}. $$ \label{111}

\item Suppose that $x^* = x_{max}$, then $$h_2(u) \underset{u\to\infty}{\sim} c u^{-1/2}e^{-x_{max} u} .$$

\item Suppose that there is no pole in $ ]0, x_{max}[$ and that $x^* \neq x_{max}$, then, $$h_2(u) \underset{u\to\infty}{\sim} c u^{-3/2}e^{-x_{max} u} . $$

\end{enumerate}
A symmetric result holds for $h_1$.
\label{prop:asympth}
\end{prop}

\begin{proof}
The result directly follows from classical Tauberian inversion lemmas which link the asymptotic of a function at infinity to the first singularity of its Laplace transform (which is here given in Lemma~\ref{lem:branchementphi2}). We refer here to Theorem 37.1 of Doetsch's book~\cite{doetsch_introduction_1974} and more precisely we apply the special case stated in Lemma C.2 of \cite{dai_reflecting_2011}.
To apply this lemma, we have to verify the analyticity and the convergence to $0$ at infinity of $\phi_2$ in a domain $\mathcal{G}_\delta(x_{max}) :=\{z\in\mathbb{C}:z\neq x_{max}, |\arg (z-x_{max})|>\delta \}$ for some $\delta\in(0,\pi/2)$. But this follows directly from the continuation procedure of Lemma~\ref{conti} :the exponential part of the continuation formula~\eqref{zio} tends to $0$ in a domain $\mathcal{G}_\delta(x_{max})$ for some $\delta\in(0,\pi/2)$ by using (ii) of lemma~\ref{alleq}. 
 Note that the convergence to $0$ also follows from Lemma~\ref{lem:C1}. Then, Lemma~\ref{lem:branchementphi2} gives the nature at the branching point $x_{max}$ which is the smallest singularity except in the case where there is a pole in $]0,x_{max}[$, where the pole $x^*$ is the smallest singularity.
\end{proof}

{\begin{rem}\label{uniformez0}
We remark in the proof of Lemma~\ref{lem:branchementphi2} that $O(1)$ and $O(x_{max} - x)$ of this lemma are locally uniform according to $z_0$. \nr{Which}\nb{This} means that $\sup_{z'_0 \in V}\left|\varphi_2^{(z'_0)}(x) - \frac{C^{(z'_0)}}{\sqrt{x_{max} - x} }\right| = O(1)$ as $x \to x^*$ when $\gamma_2(x_{max}, Y^-(x_{max})) = 0$ for a sufficiently small neighborhood $V$ of $z_0$ (and the same holds for $O(x_{max} - x)$ in the other case).
This implies that the results of Proposition~\ref{prop:asympth} hold locally uniformly in $z_0$. Indeed, it is enough to adapt the Tauberian lemmas of \cite{doetsch_introduction_1974} used in the proof of Proposition~\ref{prop:asympth}\nr{ in a slightly more technical but quite similar way}. Note that the constants $c$ of this proposition depend continuously on $z_0$.
\end{rem}}

\begin{proof}[Proof of Lemma~\ref{lem:poleallz0}]
Let $z_0=(a_0,b_0)$ be a starting point such that $x^*$ is a pole of $\phi_2^{z_0}$. Then, the continuation formula~\eqref{zio} implies that $ -\gamma_1(x^*,Y^-(x^*)) \phi_1^{z_0}(Y^-(x^*)) -\exp \big(a_0 x^* + b_0 Y^{-}(x^*) \big) \neq 0$. By continuity with respect to the starting point (which follows from the integral formula given in~\cite{franceschi_green_2021} or from \cite{lipshutz_ramanan_pathwise_19}), there exists a neighbourhood $V$ of $z_0$ such that $ -\gamma_1(x^*,Y^-(x^*)) \phi_1^{z_0'}(Y^-(x^*)) -\exp \big(a_0' x^* + b_0' Y^{-}(x^*) \big) \neq 0$ for all $z_0'=(a_0',b_0') \in V$. Therefore, by the continuation formula, $x^*$ is a pole of $\phi_2^{z_0'}$ for all $z_0' \in V$. From Proposition~\ref{prop:asympth} and by continuity of the constant of this proposition according to $z'_0$ we conclude the following. If $x^*$ is a pole of $\phi_2^{z_0'}$, there exists a constant $c$ such that for all $z_0' \in V$ {we have $h_2^{(z_0')}(u) = c e^{-x^*u} (1+o(1))$ (notice that $o(1)$ is uniform in $z_0'$ in the sense of Remark~\ref{uniformez0} and that $c$ is continuous in $z_0'$).}
 For $z_0''\in \mathbb{R}^2$ we introduce the stopping time
$$
T_V:=\inf \{ t >0 : Z_t^{z_0''} \in V \}
$$
where $Z_t^{z_0''}$ denotes the process starting from $z_0''$.
By the strong Markov property applied to $T_V$ we have {for some constant $C$ and when $u\to\infty$,
$$
h_2^{z_0''}(u) \geqslant \mathbb{P}_{z_0''}(T_V < \infty) \inf_{z_0'\in V} h_2^{z_0'}(u) = C e^{-x^*u} (1+o(1)).
$$
We deduce by Proposition~\ref{prop:asympth} that $z_0''$ is necessarily a pole.}
\end{proof}

We conclude this section with the following lemma which will be needed in Section~\ref{sec:shift}.
\begin{lemma}[Boundedness of the Laplace transform] 
\label{phibound} 
  Let $\eta \in ]0, \delta[$, we have
$$ \sup_{ u \in [X^{\pm}(y_{\max})-\eta,  x_{max}+\eta] \atop 
  |v|>\epsilon} |\phi_2(u+iv)| <\infty. $$
\end{lemma} 


\begin{proof} 
   Clearly, for any $x=u+iv$ with $u<0$,
$|\phi_2(u+iv)|\leq \phi_2(u)$.  Then for any $\epsilon>0$,
\begin{equation}
    \label{opzz} 
\sup_{ u \in [X^{\pm}(y_{\max})-\eta,  -\epsilon]} 
 |\phi_2(u+iv)|  <\infty. 
 \end{equation} 
    For any $x=u+iv$ with $u \in [-\epsilon, x_{max}+\eta]$  
Lemma \ref{conti} applies and gives the representation 
 \eqref{zio}. Let us consider all its terms.  
  By Lemma \ref{alleq} (ii), for any fixed 
$u \in {\mathbb{R}}$,
    \nb{the} function ${\rm Re} Y^{-}(u+iv)$
  is strictly decreasing as $|v|$  goes from $0$ to infinity.
Moreover, for any $u \in [-\epsilon, x_{max} +\delta]$
 $$ {\rm Re} Y^{-} (u+iv) \leq -\frac{1}{\sqrt{2} \sigma_{22}} |v|.$$ 
   Then,
   \begin{equation}
       \label{prom}
   |\phi_1(Y^{-}(u+iv))| \leq  \phi_1 ({\rm Re}\,Y^{-}(u+iv))  \leq 
      \phi_1 \left( \frac{-1}{\sqrt{2} \sigma_{22}} |v|\right) \leq  \phi_1(0). 
      \end{equation} 
  By \nr{Lemma}\nb{Proposition} \ref{pole} (i) $\phi_1(0)<\infty$. It follows that 
 \begin{equation}
     \label{opzzz}
 \sup_{ u \in [-\epsilon, x_{max} +\delta] }\phi_1(Y^{-}(u+iv)) <\infty.   
 \end{equation} 
  By Lemma \ref{alleq} (i)  there exists a constant $d_1>0$
such that 
\begin{equation}
    \label{d1}
|\gamma_1(u+iv, Y^{-}(u+iv))| \leq d_1 |v|,  \  \    
    \forall  u \in [-\epsilon, x_{max}+\eta], \ |v| \geq \epsilon. 
    \end{equation} 
   Note that  $|\gamma_2(u+iv, Y^{-}(u+iv))| \geq  |r_{12} u + r_{22} {\rm Re} Y^{-}(u+iv)|$. Then by Lemma \ref{alleq} (ii) and also by \nr{Lemma}\nb{Proposition} \ref{pole} (ii) 
  there exists a constant $d_2>0$ such that 
  \begin{equation}
      \label{d2}
  |\gamma_2(u+iv, Y^{-}(u+iv))|    
     \geq d_2  |v|,  \  \    
    \forall u \in [-\epsilon, x_{max}+\eta], \  |v|\geq \epsilon.
    \end{equation} 
  Finally by Lemma \ref{alleq} (ii)
  \begin{equation}
      \label{exp}
  |\exp(a_0 (u+iv)+ b_0 Y^{-}(u+iv))| = 
    \exp (a_0 u + b_0 {\rm Re}\, Y^{-} (u+iv) ) \leq \exp \left(\left(a_0 {- b_0\frac{\sigma_{12}}{\sigma_{22}}} \right)u  -  \frac{b_0 }{ \sqrt{2} \sigma_{22} } |v| \right)
    \end{equation} 
      for any $u \in [-\epsilon, x_{max}+ \eta]$ and $v$ with $|v|>\epsilon$.   
      Then the estimate (\ref{opzz}), 
      the representation (\ref{zio}) combined with the estimates (\ref{opzzz}), (\ref{d1}), (\ref{d2}) and (\ref{exp}) 
lead to the statement 
of the lemma.
\end{proof} 

\section{Inverse Laplace transform: from a double integral to simple integrals} 
\label{sec:sumofsimpleintegrals}

By the Laplace transform inversion formula (\cite[Theorem 24.3 and 24.4]{doetsch_introduction_1974} and \cite{brychkov_multidimensional_1992}), for any $\epsilon>0$ small enough,
$$ g(a,b)= \frac{1}{(2\pi i)^2 } \int_{-\epsilon-i \infty}^{-\epsilon+ i \infty} \int_{-\epsilon - i \infty }^{-\epsilon+ i \infty} \phi(x,y) \exp(-a x -by) dxdy,$$
in the sense of principal value convergence.
\begin{lemma}[Inverse Laplace transforms as a sum of simple integrals]\label{3_integ}
Let $z_0=(a_0,b_0)$ be the starting point of the process. For any $(a,b)\in \mathbb{R}_+^2$ where either $a>a_0, b>0$ or $b>b_0, a>0$
  the following representation holds : 
$$g(a,b)= I_1(a,b)+ I_2(a,b) + I_3(a,b)$$ where 
$$ I_1(a,b)= \frac{1}{2 \pi i } \int_{-\epsilon - i \infty}^{- \epsilon + i \infty }
 \phi_2 (x) \gamma_2 (x, Y^{+} (x))
 \exp( -a x - b Y^{+}(x)) \frac{dx}{\gamma'_y(x, Y^{+}(x)) }, $$
$$ I_2(a,b)= \frac{1}{2 \pi i } \int_{-\epsilon - i \infty}^{- \epsilon + i \infty }
 \phi_1 (y) \gamma_1 ( X^{+} (y), y)  \exp( -a X^{+}(y)-by) \frac{dy}{\gamma'_x( X^{+}(y), y)}, $$
$$ I_3(a,b)= \frac{1}{2\pi i }    \int_{-\epsilon - i \infty}^{- \epsilon + i \infty }
 \exp(a_0 x + b_0 Y^{+}(x)) 
 \exp( -a x - b Y^{+}(x)) \frac{dx}{\gamma'_y(x, Y^{+}(x)) } \ \  \hbox{ if }b>b_0, $$
 $$ 
 I_3(a,b)= \frac{1}{2\pi i }  \int_{-\epsilon - i \infty}^{- \epsilon + i \infty }
 \exp(a_0 X^{+}(y)  + b_0 y ) 
 \exp( -a X^{+}(y) - b y ) \frac{dy}{\gamma'_x(X^{+}(y), y) } \ \  \hbox{ if }a>a_0.
 $$   
 \label{lem:I123}
\end{lemma} 

The two different formulas for $I_3$ will be useful in Section~\ref{sec:asymptoticaxes} in studying the asymptotics along the axes.

\begin{proof} 

For any $\epsilon>0$ small enough $\gamma(-\epsilon, -\epsilon)<0$. 
Then
\begin{equation}
\label{zui}
{\rm Re} \gamma(-\epsilon+iv_1, -\epsilon+ iv_2) <0 \  \  \forall v_1, v_2 \in {\mathbb{R}} 
\end{equation} 
since $\Sigma$ is a covariance matrix. 
Then, by \eqref{Equation fonctionnelle} 
$$g(a,b)= \frac{-1}{(2 \pi i)^2 } \int_{-\epsilon -i \infty}^{-\epsilon + i \infty} 
  \int_{-\epsilon -i\infty}^{-\epsilon + i\infty}
  \frac{\gamma_1(x,y) \phi_1(y)+ \gamma_2(x,y) \phi_2(x)+ \exp (a_0 x + b_0 y)}{\gamma(x,y)} \exp(-a x -by)dxdy. $$
Now, let us consider for example the second term. It can be written as 
$$ \frac{-1}{(2 \pi i)^2} \int_{-\epsilon -i \infty}^{-\epsilon+i \infty}
\phi_2(x) \exp(-ax) 
   \Big( \int_{-\epsilon -i \infty}^{-\epsilon + i\infty} 
   \frac{\gamma_2(x,y)}{\gamma(x,y)} \exp(-by) dy \Big) dx.$$
   Note that the convergence in the sense of the principal value of this integral can be guaranteed by integration by parts. 
Now, it just remains to show that 
    \begin{equation}
        \label{sv}
     \frac{-1}{2 \pi i } \int_{-\epsilon -i\infty}^{-\epsilon + i\infty} 
       \frac{\gamma_2(x, y)}{\gamma(x,y)} \exp(-by)dy = 
       \frac{\gamma_2(x, Y^{+}(x))}{\gamma'_y (x, Y^{+}(x))}\exp\big(-b Y^{+}(x)\big).
       \end{equation} 
  Let $x=-\epsilon$. The equation $\gamma(-\epsilon, y)=0$
     has two solutions, $Y^{+}(-\epsilon)>0$ and $Y^{-}(-\epsilon)<0$.
      (In fact, for $\epsilon>0$  small enough $Y^{+}(-\epsilon)$ is close to $Y^{+}(0)=0$
        staying positive and $Y^{-}(-\epsilon)$  is close to $Y^{-}(0)= - 2\mu_2/\sigma_{22}<0$). 
  Let $x=-\epsilon + iv$. 
The functions  $Y^{+}(-\epsilon+iv)$ and
       $Y^{-}(-\epsilon+iv)$ are continuous in $v$.  By \eqref{zui}
    their real parts do not equal $-\epsilon$ for \nr{no}\nb{any} $v \in {\mathbb{R}}$. 
       Thus ${\rm Re} Y^{+}(-\epsilon + iv)>-\epsilon$ and ${\rm Re} Y^{-}(-\epsilon + iv)<-\epsilon$  for all $v \in {\mathbb{R}}$.      
       Let us construct the contour $ [-\epsilon -i R, -\epsilon + i R] \cup \{ t + i R, \mid t \in [-\epsilon, 0] \} \cup  \{R e^{it} \mid t \in ]-\pi/2 + \pi/2[ \} 
    \cup   \{ t - i R,  \mid t \in [-\epsilon, 0] \}$, see Figure~\ref{fig:inversioncontour}.

\begin{figure}[ht]
     \includegraphics[scale=1.8]{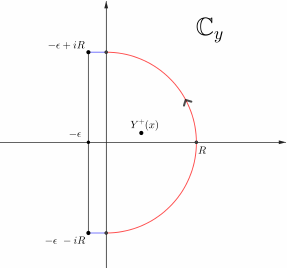} 
     \caption{Integral contour in the complex plane $\C_y$, with the pole $Y^+(x)$.}
    \label{fig:inversioncontour}
\end{figure}

       For any fixed $x=-\epsilon+iv$,   
   the integral over this contour taken in the counter-clockwise direction 
      of the function $\frac{\gamma_2(x,y)}{\gamma(x,y)} \exp(-by)$  equals 
     the residue of this function multiplied by $2\pi i$, which is 
        exactly the result in \eqref{sv}. 
     It suffices to show that the integral over  $\{ t + i R \mid t \in [-\epsilon, 0] \} \cup  \{R e^{it} \mid t \in ]-\pi/2 + \pi/2[ \} 
    \cup   \{ t - i R \mid t \in [-\epsilon, 0] \}$  converges to zero as $R \to \infty$.    
       The integral over the half of the circle
           $\{R e^{it} \mid t \in ]-\pi/2 + \pi/2[ \}$ equals 
    $$ \int_{-\pi/2}^{\pi/2} \frac{\gamma_2(x, Re^{it})}{ \gamma(x, Re^{it})} \exp (-b R e^{it}) iRe^{it} dt.$$
     We have $\sup_{R>R_0} \sup_{t \in ]-\pi/2, \pi/2[} \Big| \frac{\gamma_2(x, R e^{it})  }{\gamma(x, R e^{it})  } i R e^{it}   \Big|<\infty $ for $R_0 = R_0(x) > 0$ large enough, while 
       $|\exp (-b R e^{it}) | = \exp (-b R \cos t) \to 0$ as $ R \to \infty$
          for any $t \in ]-\pi/2, \pi/2[$ since $b>0$.
             Hence, the integral over the half of the circle 
     converges to zero as $R \to \infty$ by the dominated convergence theorem.
    Let us look at the integral over segment $\{ t + i R \mid  t \in [-\epsilon, 0]\}$. 
       For any fixed $x=-\epsilon+iv$,
           there exists a constant $C(x)>0$ such that for any  $R$ large enough 
      $$ \sup_{u \in [-\epsilon, 0]}\Big| \frac{\gamma_2(x, u+ i R)}{\gamma(x, u+ iR)}   \Big| \leq \frac{C(x)}{R}.$$
     Therefore 
       $$\Big|\int_{-\epsilon}^0 \frac{\gamma_2(x, u+ iR)   }{ \gamma(x, u+ iR)  }  \exp (- b (u+ i R)) du \Big| \leq \epsilon \exp ( b \epsilon ) \frac{C(x)}{R} \underset{R\to\infty}{\longrightarrow} 0.$$
   The representation of $I_1(a,b)$ follows.

        The reasoning is the same for the third term. 
         The integral over the half of the circle equals  
    $$ \int_{-\pi/2}^{\pi/2} \frac{\exp (-(b-b_0) R e^{it}) }{ \gamma(x, Re^{it})} iRe^{it} dt.$$
          We have $\sup_{R>R_0} \sup_{t \in ]-\pi/2, \pi/2[} \Big| \frac{1 }{\gamma(x, R e^{it})  } i R e^{it}   \Big|<\infty $ while 
       $|\exp (-(b-b_0) R e^{it}) | = \exp (-(b-b_0) R \cos t) \to 0$ as $ R \to \infty$
          for any $t \in ]-\pi/2, \pi/2[$ since $b-b_0>0$. 
                The integral over the half of the circle 
     converges to zero  as $R \to \infty$ by the dominated convergence theorem once again.
          For any fixed $x=-\epsilon+iv$,
           there exists a constant $C(x)>0$ such that for any  $R$ large enough 
      $$ \sup_{u \in [-\epsilon, 0]}\Big| \frac{1}{\gamma(x, u+ iR)}   \Big| \leq \frac{C(x)}{R^2}.$$
     Therefore 
       $$\Big|\int_{-\epsilon}^0 \frac{\exp(- (b-b_0)(u+ i R))  }{ \gamma(x, u+ iR)  }  du \Big| \leq \epsilon \exp ( (b-b_0) \epsilon ) \frac{C(x)}{R^2} \to 0,  \  \  R \to \infty.$$

       The representations for $I_2(a,b)$ and $I_3(a,b)$ with $a>a_0$  are obtained in the same way. 
\end{proof} 

\noindent{\bf Remark.} Let us introduce \nr{some notations}\nb{the notation} $a$, $b$, $c$, $\widetilde a$, $\widetilde b$, $\widetilde c$ by
\begin{equation}
    \label{abc}
\gamma(x,y)= a(x)y^2 +b(x)y+c(x) = \widetilde a(y)x^2+ \widetilde b(y) x + \widetilde c(y).
\end{equation} 
 Then functions in the integrand can be represented as 
 \begin{equation}
 \label{ddd}
  \gamma'_y(x, Y^+(x))= a(x) (Y^+(x)- Y^-(x)) = 2 a(x) Y^+(x) + b(x)= \sqrt{b^2(x)-4a(x) c(x)} 
  \end{equation}
  \begin{equation}
      \label{dddd}
  \gamma'_x(X^+(y), y)= \widetilde a(y) (X^+(y)- X^-(y)) = 2 \widetilde a(y) X^+(y) + \widetilde b(y)= \sqrt{ \widetilde b^2(y)-4\widetilde a(y) \widetilde c(y)}.
  \end{equation} 
\nr{All these forms will be used in the following.}
    
\section{Saddle point and contour of the steepest descent} 
\label{sec:saddlepoint}

Our aim is to study the integrals $I_1$, $I_2$ and $I_3$ of Lemma~\ref{3_integ} using the saddle point method \nb{(see, for example, Fedoryuk \cite{fedoryuk_asymptotic_1989})}.\nr{ One of the reference books about this approach is those of Fedoryuk \cite{fedoryuk_asymptotic_1989}.} 

\subsection*{Saddle point}
For $\alpha\in[0,2\pi[$ we define
\begin{equation}
(x(\alpha), y(\alpha)):= {\rm argmax}_{(x,y): \gamma(x,y)=0}(x \cos \alpha + y \sin \alpha).
\label{eq:defsaddlepoint}
\end{equation}
We will see that this point turns out to be the saddle point of the functions inside the exponentials of the integrals $I_1$, $I_2$ and $I_3$. See Figure~\ref{fig:saddlepoint} for a geometric interpretation of this point.

\begin{figure}[ht]
     \includegraphics[scale=4]{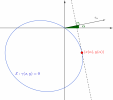} 
     \caption{Graphical representation of the saddle point. We denote $e_{\alpha} = (\cos(\alpha), \sin(\alpha))$.}
    \label{fig:saddlepoint}
\end{figure}
The map $\alpha : [0, 2\pi[ \to  \{(x,y) : \gamma(x,y)=0\}$ is a diffeomorphism. 
  The functions $x(\alpha), y(\alpha)$ are in the class $C^\infty([0, 2\pi])$.
For any $\alpha \in [0, \pi/2]$ the function  
$\cos (\alpha)  x + \sin (\alpha) Y^{+}(x)$ reaches its maximum at the unique point on $[X^{\pm} (y_{max}), x_{max}]$ called  $x(\alpha)$. 
This function is strictly increasing on $[X^{\pm}(y_{max}), x(\alpha)]$ and  
  strictly decreasing on $[x(\alpha), x_{max}]$. 
  \nr{Function}\nb{The function} 
$\cos (\alpha) X^{+}(y) + \sin (\alpha) y$
reaches its maximum on $[Y^{\pm}(x_{max}), y_{max}]$ 
  at the unique point $y(\alpha)$.  It is strictly increasing on $[Y^{\pm}(x_{max}), y(\alpha)]$ and strictly decreasing on $[y(\alpha), y_{max}]$.  
     
      Thus $x(0)=x_{max}$, $y(0)= Y^{\pm} (x_{max})$, 
      $x(\pi/2)= X^{\pm}(y_{max})$, $y(\pi/2)= y_{max}$.
Finally, $x(\alpha)=0$ and $y(\alpha)=0$  if 
$(\cos (\alpha), \sin (\alpha))= \left(\frac{\mu_1}{\sqrt{\mu_1^2 + \mu_2^2}},\frac{\mu_2}{\sqrt{\mu_1^2 + \mu_2^2}}\right)$. We denote the direction corresponding to the drift by  $\alpha_\mu$. 

 \nr{Let's}\nb{Let us} define the functions
 \begin{equation}
     \label{fff}
 F(x,\alpha)=  - \cos (\alpha) x - \sin(\alpha) Y^{+}(x) + \cos (\alpha) x(\alpha) + \sin (\alpha)  y(\alpha)  ,
 \end{equation}
     $$G(y, \alpha) =   -\cos(\alpha) X^{+}(y)- \sin(\alpha) y 
       + \cos (\alpha) x(\alpha) + \sin (\alpha) y(\alpha) .$$
       \nr{The function $F$ appears to be (up to a constant)}\nb{We see that the function $F$ is (up to a constant)} the function inside exponential of the integral $I_1$, and the function $G$ \nr{appears to be}\nb{is} (up to a constant) the function inside the exponential of the integral $I_2$, see Lemma~\ref{3_integ}.
  We have 
  $$F(x(\alpha), \alpha)=0 \ \  \forall \alpha \in [0, \pi/2]$$
and 
$$F'_x(x(\alpha), \alpha)=0 \ \ \forall \alpha \in ]0, \pi/2], \hbox{ 
  but not at }\alpha=0.$$
   In the same way $G(y(\alpha), \alpha)=0$ for any $\alpha \in [0, \pi/2]$
and $G'_y(y(\alpha), \alpha)=0$ for any $\alpha \in [0, \pi/2[$
  but not at $\alpha=\pi/2$.
  Then $(Y^{+}(x(\alpha)))' = -{\rm ctan }(\alpha)$ and 
        $(X^{+}(y(\alpha)))' = -{\rm tan }(\alpha)$. 
      
  Using the identities  $\gamma(x, Y^{+}(x)) \equiv 0$ and 
  $\gamma(X^+(y), y) \equiv 0$, we get : 
\begin{equation} \label{all}
(Y^{+}(x))'\Bigm|_{x=x(\alpha)} = - \frac{ \gamma'_x(x(\alpha), y(\alpha))  }{ \gamma'_y(x(\alpha), y(\alpha))} = -\frac{\cos (\alpha)}{\sin (\alpha)}, \ \  \ \alpha \in ]0, \pi/2]
\end{equation} 
$$(X^{+}(y))'\Bigm|_{y=y(\alpha)} = - \frac{ \gamma'_y(x(\alpha), y(\alpha))  }{ \gamma'_x(x(\alpha), y(\alpha))} = -\frac{\sin (\alpha)}{\cos (\alpha)}, \ \  \alpha \in [0, \pi/2[$$
$$(Y^{+}(x))''\Bigm|_{x=x(\alpha)} = -\frac{ \sigma_{11} + 2\sigma_{12} (-{\rm ctan}\,(\alpha)) + \sigma_{22} (-{\rm ctan}\,(\alpha))^2 }{ \gamma'_y(x(\alpha), y(\alpha) )}    $$
$$(X^{+}(y))''\Bigm|_{y=y(\alpha)} = -\frac{ \sigma_{11} (-{\rm tan } (\alpha))^2  + 2\sigma_{12} (-{\rm tan}\,(\alpha)) + \sigma_{22} }{ \gamma'_x(x(\alpha), y(\alpha) )}$$
\nr{
\begin{equation}
\label{fzfz}
F''_{x}(x(\alpha), \alpha) = \frac{ \sigma_{11}\sin^2(\alpha) + 2\sigma_{12} \sin (\alpha) \cos (\alpha) + \sigma_{22} \cos^2(\alpha) }{ \gamma'_y(x(\alpha), y(\alpha) ) \sin \alpha }>0  \  \ \alpha \in ]0, \pi/2],
\end{equation} 
}
\nb{
\begin{equation}
\label{fzfz}
F''_{x}(x(\alpha), \alpha) = \frac{ \sigma_{11}\sin^2(\alpha) - 2\sigma_{12} \sin (\alpha) \cos (\alpha) + \sigma_{22} \cos^2(\alpha) }{ \gamma'_y(x(\alpha), y(\alpha) ) \sin \alpha }>0  \  \ \alpha \in ]0, \pi/2],
\end{equation} 
}
\nr{
 $$G''_{y}(y(\alpha), \alpha)=\frac{ \sigma_{11}\sin^2(\alpha) + 2\sigma_{12} \sin (\alpha) \cos (\alpha) + \sigma_{22} \cos^2(\alpha) }{ \gamma'_x(x(\alpha), y(\alpha) ) \cos (\alpha) }>0  \  \ \alpha \in [0, \pi/2[,$$  
}
\nb{
 $$G''_{y}(y(\alpha), \alpha)=\frac{ \sigma_{11}\sin^2(\alpha) - 2\sigma_{12} \sin (\alpha) \cos (\alpha) + \sigma_{22} \cos^2(\alpha) }{ \gamma'_x(x(\alpha), y(\alpha) ) \cos (\alpha) }>0  \  \ \alpha \in [0, \pi/2[,$$  
}
where the strict inequality arises from \eqref{ddd}, \eqref{dddd} and the positive-definite form of $\Sigma$.

\begin{figure}[hbtp]
\centering
\includegraphics[scale=2,clip=true,trim=0.5cm 0.5cm 0.5cm 0.5cm]{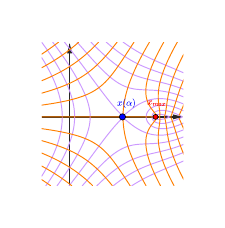}
\caption{Level sets of $\Re (F)$ in purple and of $\Im (F)$ in orange. The saddle point $x(\alpha)$ is represented in blue and the branch point $x_{max}$ is in red.}
\label{fig:levelset}
\end{figure}

The values of $x(\alpha)$ and $y(\alpha)$ are given by the following formulas.

\begin{equation}\label{degueu1}
    x(\alpha) = \frac{(\mu_2 \sigma_{12} - \mu_1 \sigma_{22})}{\det(\Sigma)} + \frac{{1}}{\det(\Sigma)} (\sigma_{22} - \tan(\alpha) \sigma_{12}) \sqrt{{\frac{{\mu_2^2 \sigma_{11} - 2 \mu_1 \mu_2 \sigma_{12} + \mu_1^2 \sigma_{22}}}{{\sigma_{11} \tan^2(\alpha) - 2 \sigma_{12} \tan(\alpha) + \sigma_{22}}}}}
\end{equation}
\begin{equation}\label{degueu2}
    y(\alpha) = \frac{{(\mu_1 \sigma_{12} - \mu_2 \sigma_{11})}}{\det(\Sigma)} + \frac{{1}}{\det(\Sigma)} \left(\sigma_{11} - \frac{1}{\tan(\alpha)} \sigma_{12}\right) \sqrt{{\frac{{\mu_1^2 \sigma_{22} - 2 \mu_1 \mu_2 \sigma_{12} + \mu_2^2 \sigma_{11}}}{{\frac{{\sigma_{22}}}{\tan^2(\alpha)} - 2 \frac{\sigma_{12}}{\tan(\alpha)} + \sigma_{11}}}}}.
\end{equation}
Indeed, using the same calculations as in section $4.2$ of \cite{franceschi_asymptotic_2016}, the equation $0=\frac{d}{dx}\left[\gamma(x, Y^+(x))\right]|_{x = x(\alpha)}$ combined with the first equation of \eqref{all} gives a linear relationship between $x(\alpha)$ and $y(\alpha)$. Injecting this condition in the polynomial equation $\gamma(x(\alpha), y(\alpha))=0$, we get two possible values for $x(\alpha)$ and $y(\alpha)$. \nr{The choice of the sign depends then of $\alpha$ and we get this expression}\nb{The choice of sign then depends on $\alpha$}.

\subsection*{Contour of the steepest descent}   
Before continuing, the reader should read Appendix~\ref{sec:morse} which states a parameter dependent Morse lemma. \nb{The usual Morse Lemma enables one to find steepest descent contours for a function at a critical point. The parameter dependent Morse lemma treats the case of a family of functions $(f_\alpha)_{\alpha}$ which have critical points $x(\alpha)$ (with smooth dependency in $\alpha$). This lemma tells us that the contours of steepest descents of $f_\alpha$ at $x(\alpha)$ are also smooth in $\alpha$. This property is necessary to obtain the asymptotic behaviour where $r\to+\infty$ and $\alpha \to \alpha_0$.} 
Let $\alpha_0 \in ]0, \pi/2]$. We apply Lemma~\ref{Morse} to $F$ defined in~\eqref{fff}.
Let us fix any $\epsilon \in ]0, K[$  and consider any $\alpha$ such  that $|\alpha -\alpha_0 |<\eta$, where 
     constants $K$ and $\eta$  are taken from the definition of 
  $\Omega(0, \alpha_0)$ in Lemma~\ref{Morse}. 
Then, for any $\alpha$ we can construct the contour of the steepest descent 
  $$\Gamma_{x,\alpha} =\{x(it, \alpha) \mid  t \in [-\epsilon, \epsilon] \}.$$
Clearly, 
 $$F(x(it, \alpha), \alpha)= -t^2.$$
  We denote by $x^+_\alpha= x(i\epsilon, \alpha)$ and $x^{-}_\alpha = x(-i\epsilon, \alpha)$\nr{ its ends}.
Then 
  \begin{equation}
      \label{fe}
   F(x^+_\alpha, \alpha)=-\epsilon^2, \  \  F(x^-_\alpha, \alpha)=-\epsilon^2.
   \end{equation}
Since $F''_x(x(\alpha), \alpha) \ne 0$, the contours in a neighborhood of $x(\alpha)$ where the function $F$ is real are orthogonal, see Figure~\ref{fig:levelset}.  One of them is the real axis. 
    The other is the contour of 
  the steepest descent, which is the  orthogonal to the real axis. 
      It follows that ${\rm Im}x_{\alpha_0}^+>0$  and 
   ${\rm Im} x_{\alpha_0}^-<0$. By continuity of $x(i \epsilon,  \alpha)$ on $\alpha$ for any $\eta>0$ small enough, there exists $\nu>0$ such that 
 \begin{equation} 
\label{nu} 
{\rm Im}\,x_{\alpha}^+>\nu, \  \  {\rm Im}\,x_{\alpha}^-<-\nu \ \  \forall \alpha : |\alpha-\alpha_0|<\eta. 
\end{equation} 

  In the same way, for any $\alpha \in [0, \pi/2[$, 
    we may define by the generalized Morse lemma the function $y(\omega, \alpha)$ w.r.t. $G(y, \alpha)$. Let $\alpha_0 \in [0, \pi/2[$. We can 
    construct the contour of the steepest descent 
$$\Gamma_{y,\alpha} =\{y(it, \alpha) \mid t \in [-\epsilon, \epsilon] \}$$ 
   with end points $y^+_\alpha= y(i\epsilon, \alpha)$ and $y^{-}_\alpha = y(-i\epsilon, \alpha)$ and the property 
   analogous to (\ref{nu}).

We note that for any $\alpha=]0,\pi/2[$ 
\begin{equation}
\Gamma_{x, \alpha}= \underrightarrow{\overleftarrow{X^+ (\Gamma_{y, \alpha})}},\ \
 \Gamma_{y, \alpha} = \underrightarrow{\overleftarrow{Y^+ (\Gamma_{x, \alpha})}}.
 \label{eq:gammaXarrow}
\end{equation}
   The arrows mean that the direction has to be changed because of the facts that $(X^+(y))'\Bigm|_{y=y(\alpha)}<0$ and 
    $(Y^+(x))'\Bigm|_{x=x(\alpha)}<0$. \nb{This notation comes from \cite{FIM17} (chapter 5.3, p 137).}
    
    \begin{figure}[hbtp]
    \centering
    \includegraphics[scale=1.5]{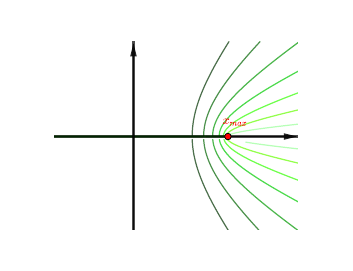}
    \caption{Steepest descent contour for $\Re(F)$ according to $\alpha$. As $\alpha$ gets closer to zero, the corresponding contours appear in lighter shades of green. When $\alpha\to 0$ this contour tends to the half line $[x_{max},\infty)$.}
    \label{fig:steepestalpha0}
    \end{figure}
    
 \subsubsection*{Case where $\alpha_0=0$}
In this case $\Gamma_{y, 0}$ is now well-defined, but not 
  $\Gamma_{x, 0}$ (since $F_x''(x(0),0)=\infty$), see Figure~\ref{fig:steepestalpha0}. Thus we define
  $$\Gamma_{x,0}=\underrightarrow{\overleftarrow{X^{+}(\Gamma_{y,0})}}$$  with end points 
     $x^+_0 = X^+(y^{+}_{0})= x_{max}+ \epsilon^2$  and 
   $x^{-}_0=X^+(y^{-}_{0})=x_{max}+\epsilon^2$.
In fact, for $\alpha=0$, we have $G(y, 0)=-X^+(y)+x_{max}$
 and $G(y(i\epsilon, 0), 0)=-\epsilon^2$.
  Thus  $\Gamma_{x,0}$ runs the real segment 
 from $x_{max}+\epsilon^2$  to $x_{max}$  and back to
 $x_{\max}+ \epsilon^2$. Figure~\ref{fig:steepestalpha0} illustrates why this phenomenon happens when $\alpha=0$. Again by continuity on $\alpha$
we may find $\eta>0$ and $\nu>0$ small enough, such that 
\begin{equation}
\label{zoi} 
  Re x_\alpha^+-x_{max}>\nu, \ \  Re x_\alpha^- 
   - x_{max}>\nu,  \  \  \forall \alpha \in [0,  \eta].  
\end{equation}  
        
If $\alpha_0=\pi/2$, 
$\Gamma_{x, \pi/2}$ is well-defined, 
but not $\Gamma_{y, \pi/2}$. 
We \nr{put then}\nb{then let}
  $$\Gamma_{y,\pi/2}=\underrightarrow{\overleftarrow{ Y^{+}(\Gamma_{x,\pi/2}) }}$$
    with endpoints $y^+_\alpha= Y^+(x^{+}_{\alpha})$  and  
       $y^-_\alpha=Y^+(x^{-}_{\alpha})$.

\section{Shift of the integration contours and contribution of the poles} 
\label{sec:shift}

We will now define the integration contours of $I_1$, $I_2$ and $I_3$ \nr{thanks to}\nb{using} the contours of the steepest descent studied in the previous section.
First, let $$S^+_{x, \alpha}=\{x^+_{\alpha} +it \mid t \geq 0 \ \},\ \ 
 S^-_{x, \alpha}=\{x^-_{\alpha} - it \mid t \geq 0 \},$$ 
        $$S^+_{y, \alpha}=\{y^+_{\alpha} +it \mid t \geq 0 \}, \ \  
 S^-_{y, \alpha}=\{y^-_{\alpha} - it \mid t \geq 0 \}.$$

Now, let us construct the integration contours 
  $T_{x, \alpha}= S^-_{x, \alpha}+ \Gamma_{x, \alpha} + S^+_{x, \alpha}$  and 
     $T_{y, \alpha}= S^-_{y, \alpha}+ \Gamma_{y, \alpha} + S^+_{y, \alpha}$ 
for any $\alpha \in [0, \pi/2]$.
See Figure~\ref{fig:contoursaddle} which illustrates these integration contours.

\begin{figure}[ht]
     \includegraphics[scale=3.5]{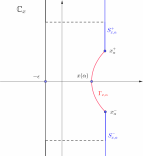} 
     \caption{In the complex plane $\C_x$, shift of the integration contour passing through the saddle point along the steepest line. }
    \label{fig:contoursaddle}
\end{figure}

\subsubsection*{Case where the saddle point \nr{meet}\nb{meets} the pole}
The only exception \nr{to define these contours}\nb{in defining these contours} will be for $\alpha \in [0, \pi/2]$ 
such that $x(\alpha )=x^* \in ]0, x_{max}[$  is a pole 
of $\phi_2(x)$ and $y(\alpha)=y^{**} \in ]0, y_{max}[$ is a pole of 
 $\phi_1(y)$. We call these directions $\alpha^*$ and $\alpha^{**}$, 
 so that $x(\alpha^*)=x^*$, $y(\alpha^*)=Y^{+}(x^*)=y^*$, 
    $y(\alpha^{**})=y^{**}$,  $x(\alpha^{**})= X^+(y^{**})=x^{**}$.
\nr{When the poles $x^*$ and $y^{**}$ exist}\nb{When $x^*$ and $y^{**}$ are poles}, we recall that by the Lemma \ref{pole} :
\begin{equation}
x^* = 2\frac{\mu_2\frac{r_{12}}{r_{22}} - \mu_1}{\sigma_{11} - 2\sigma_{12}\frac{r_{12}}{r_{22}} + \sigma_{22}\left(\frac{r_{12}}{r_{22}}\right)^2} 
 \quad \text{and} \quad
  y^{**} = 2\frac{\mu_1\frac{r_{21}}{r_{11}} - \mu_2}{\sigma_{11}\left(\frac{r_{21}}{r_{11}}\right)^2 - 2\sigma_{12}\frac{r_{21}}{r_{11}} + \sigma_{22}}.
  \label{eq:x*y**}
\end{equation}
We also recall that \nr{by definition}
\begin{equation}\label{eq:x**y*}
    y^*:=Y^{+}(x^*)
    \quad\text{and}\quad
    x^{**}:= X^+(y^{**}).
\end{equation}
We remark that we have $y^* = -\frac{r_{12}}{r_{22}}x^*$ (resp. $x^{**} = -\frac{r_{21}}{r_{11}}y^{**}$) if and only if $x^*$ (resp $y^{**}$) is not a pole of $\varphi_2$ (resp. $\varphi_1$) because of the condition on $x^*$ and $y^{**}$ to be poles.


\nr{
 If the pole $x^*$ exists, then $\alpha^* \in ]0, \alpha_\mu[$, and 
   if $y^*$ exists, then $\alpha^{**} \in ]\alpha_\mu, \pi/2[$. 
 We denote for convenience 
   $\alpha^*=-\infty$ if $x^*$ does not exist and $\alpha^{**}=+\infty$ if $y^*$  does not exist.
}
\nb{
\begin{rem}
 If $x^*$ is a pole, then $\alpha^* \in ]0, \alpha_\mu[$, and 
   if $y^{**}$ is a pole, then $\alpha^{**} \in ]\alpha_\mu, \pi/2[$. 
 We denote
   $\alpha^*=-\infty$ if $x^*$ is not a pole and $\alpha^{**}=+\infty$ if $y^{**}$ is not a pole.
\end{rem}
}
   
 If $\alpha=\alpha^* \in ]0, \alpha_\mu[$, we modify in the definition of $T_{x,\alpha}$  the contour $\Gamma_{x, \alpha}$
 by $\widetilde \Gamma_{x, \alpha}$, which is  the half of the circle 
   centered at $x(\alpha^*)$ going from $x^{+}_{\alpha^*}$ to $x^{-}_{\alpha^*}$ in the counter-clockwise direction. 
The same modification is made for $\alpha=\alpha^{**} \in ]\alpha_\mu, \pi/2[$.

The next lemma \nr{perform}\nb{performs} the shift of the integration contour and \nr{take}\nb{takes} into account the contribution of the crossed poles. Recall that $I_1$, $I_2$ and $I_3$ are defined in Lemma~\ref{3_integ}.
\begin{lemma}[Contribution of the poles to the asymptotics] \label{residus}
Let $\alpha \in [0, \pi/2]$. Then for any $a,b>0$
$$I_1(a,b)=\frac{ \big({\rm -res}_{x=x^*} \phi_2(x) \big ) \gamma_2(x^*, y^*)}{ \gamma'_y(x^*, y^*)} \exp(-ax^* -b y^* ) \times {\bf 1}_{ \alpha <\alpha^*} $$
$$ + \frac{1}{ 2 \pi i } \int\limits_{T_{x, \alpha}} 
\frac{\phi_2(x) \gamma_2(x, Y^+(x))}{\gamma'_y(x, Y^+(x))} \exp(-ax -b Y^+(x))dx,$$
$$ I_2(a, b) = \frac{ \big({\rm -res}_{y=y^{**}} \phi_1(y) \big) \gamma_1(x^{**}, y^{**})}{ \gamma'_x(x^{**},y^{**})} \exp(-ax^{**} -by^{**}) \times {\bf 1}_{ \alpha >\alpha^{**}} $$ 
$$ + \frac{1}{2\pi i } \int\limits_{T_{y, \alpha}} 
\frac{\phi_1(y) \gamma_1(X^+(y), y)}{\gamma'_x(X^+(y), y)} \exp(-aX^+(y) -b y)dy,$$
$$ I_3(a,b)= \frac{1}{2\pi i }  \int\limits_{T_{x, \alpha}}   \exp((a_0-a) x + (b_0-b) Y^{+}(x)) 
  \frac{dx}{\gamma'_y(x, Y^{+}(x)) } \ \  \hbox{ if }b>b_0 $$
 $$ 
 I_3(a,b)= \frac{1}{2\pi i }\int\limits_{ T_{x, \alpha}}
 \exp((a_0-a) X^{+}(y)  + (b_0-b) y ) 
 \frac{dy}{\gamma'_x(X^{+}(y), y) } \ \  \hbox{ if }a>a_0.
 $$   
\end{lemma} 

We remark that we have $\gamma_2(x^*, y^*){\rm res}_{x^*}\phi_2<0$ and $\gamma_1(x^{**}, y^{**}){\rm res}_{y^{**}}\phi_1<0$.

\begin{proof} We start from the result of Lemma~\ref{lem:I123}, and
we use Cauchy theorem to shift the integration contour. We take into account the poles by the residue theorem noting that $x^*<x(\alpha)$ if and only if $\alpha<\alpha^*$ and that $y^{**}<y(\alpha)$ if and only if $\alpha^{**}<\alpha$. In order to get the representation of $I_1$ by shifting the contour, we want to show that the integrals on the dotted lines of Figure~\ref{fig:contoursaddle} \nr{tends}\nb{tend} to $0$ when these lines go to infinity. To do so, it suffices to show that
  for any $\eta>0$ small enough ,
 $$ \sup_{ u \in [X^+(y_{max})-\eta, x_{max}+ \eta]}
   \Big| \frac{ \phi_2(u+iv) \gamma_2(u+iv, Y^+(u+iv))}{ \gamma_y'(u+iv, Y^+(u+iv))} \exp(-a(u+iv) - b Y^+(u+iv)) \Big| \to 0, \ \  \hbox{ as } v\to \infty.$$
\nr{In reality, it would be sufficient to study}\nb{It suffices to study} the supremum on $[- \epsilon, x_{max}+ \eta]$.
By Lemma \ref{phibound} for any $\epsilon>0$,
  $$ \sup_{ u \in [X^+(y_{max})-\eta, x_{max}+ \eta], |v| \geq \epsilon } |\phi_2(u+iv)|<\infty.$$
Let us observe that by (\ref{ddd}) 
\begin{equation}
\label{dd}
\gamma'_y(x, Y^{+}(x))=\sqrt{b^2(x)-4 a(x) c(x)} = \sqrt{ 
  (\sigma_{12}^2 -\sigma_{11} \sigma_{22})x^2 + 2 
     (\mu_2 \sigma_{12} - \mu_1 \sigma_{22}) x + \mu_2^2}.
  \end{equation}    
This function equals zero only at real points $x_{min}$ and 
  $x_{max}$ and grows linearly in\nr{ the} absolute value as $|{\Im} x| \to \infty$.  
 By Lemma \ref{alleq} (i) \nb{the} function $|\gamma_2(x, Y^{+}(x))|$
     grows linearly as $|{\Im} x | \to \infty$.  Then for any $\epsilon>0$
$$ \sup_{ u \in [X^+(y_{max})-\eta, x_{max}+ \eta],\atop  |v| \geq \epsilon }
   \Big|\frac{ \gamma_2(u+iv, Y^+(u+iv))}{ \gamma_y'(u+iv, Y^+(u+iv))} \exp(-a (u+iv)) \Big| <\infty.\ \  $$
Finally,
$$  \sup_{ u \in [X^+(y_{max})-\eta, x_{max}+ \eta]} |\exp(-b Y^+(u+iv))| = 
\sup_{ u \in [X^+(y_{max})-\eta, x_{max}+ \eta]} \exp(-b {\rm Re } Y^+(u+iv))  
\to 0,$$
  as $|v| \to \infty $   
 due to Lemma \ref{alleq} (ii) and the fact that $b>0$. 
The other representations are obtained in the same way. 
In the representations of $I_3(a,b)$ we have used the facts that $a-a_0>0$ and $b-b_0>0$.     
\end{proof}

\section{Exponentially negligible part of the asymptotic} 
\label{sec:neglibible}

Let us recall the integration contours $T_{x, \alpha}= S^-_{x, \alpha}+ \Gamma_{x, \alpha} + S^+_{x, \alpha}$  and 
     $T_{y, \alpha}= S^-_{y, \alpha}+ \Gamma_{y, \alpha} + S^+_{y, \alpha}$ 
for any $\alpha \in [0, \pi/2].$
This section establishes a domination of the integrals on the contours $S_{x,\alpha}^\pm$ and $S_{y,\alpha}^\pm$. This domination will be useful in the following sections to show that these integrals are negligible. We will see that the asymptotics of integrals $I_1$, $I_2$ and $I_3$ of contour $T_{x, \alpha}$ and $T_{y, \alpha}$ are given by the integrals on the lines of steepest descent $\Gamma_{x, \alpha}$ and $\Gamma_{y, \alpha}$.

\begin{lemma}[Negligibility of the integrals on $S_{x,\alpha}^\pm$ and $S_{y,\alpha}^\pm$] 
\label{pp}
For any couple $(a,b) \in {\mathbb{R}}_+^2$  we may define 
$\alpha(a,b)$ as the angle in $[0,\pi/2]$ such that $\cos (\alpha) = \frac{ a }{ \sqrt{a^2+b^2}}$  and  $\sin(\alpha) = \frac{ b }{\sqrt{a^2+ b^2}}$.
\begin{itemize}
\item Let $\alpha_0 \in ]0, \pi/2]$. Then for any $\eta$  small enough and any $r_0>0$ there exists a constant $D>0$
  such that for any couple $(a,b)$ where $\sqrt{a^2+b^2}>r_0$  and $|\alpha(a,b) -\alpha_0|<\eta$ we have
\begin{equation}
    \label{truc}
\Big|\int\limits_{S^{+}_{x, \alpha}}
\frac{\phi_2(x) \gamma_2(x, Y^+(x))}{\gamma'_y(x, Y^+(x))} \exp\big(-ax -b Y^+(x)\big)dx \Big| \leq \frac{D}{b} \exp\Big(-a x(\alpha) -b y(\alpha) - \epsilon^2 \sqrt{a^2+b^2}  \Big).
\end{equation}
Furthermore, if $b>b_0$ we have
\begin{equation}
\label{truct}
\Big| \int\limits_{S^{+} _{x, \alpha}}  
\exp((a_0-a) x + (b_0-b) Y^{+}(x))
  \frac{dx}{\gamma'_y(x, Y^{+}(x)) } \Big| \leq  
  \frac{D}{b-b_0} \exp\Big(-a x(\alpha) -b y(\alpha) -\epsilon^2\sqrt{a^2+(b-b_0)^2}\Big).
\end{equation}  

\item
Let $\alpha_0 \in [0, \pi/2[$.  Then for any $\eta$  small enough and any $r_0>0$ there exists a constant $D>0$
  such that for any couple $(a,b)$ such that $\sqrt{a^2+b^2}>r_0$,  $|\alpha(a,b) -\alpha_0| \leq \eta$ we have
\begin{equation}
    \label{truc1}
\Big|\int\limits_{S^{+}_{y, \alpha}}
\frac{\phi_1(y) \gamma_1(X^+(y), y)}{\gamma'_x(X^+(y), y)} \exp\big(-aX^+(y) -b y \big)dy \Big| \leq \frac{D}{a} \exp\Big(-a x(\alpha) -b y(\alpha) - \epsilon^2\sqrt{a^2+b^2}\Big).
\end{equation}
Furthermore, if $a>a_0$ we have
\begin{equation}
\label{tructt}
\Big| \int\limits_{S^{+} _{y, \alpha}}  
\exp((a_0-a) X^+(y) + (b_0-b) y )
  \frac{dy}{\gamma'_x(X^+(y), y) } \Big| \leq  
 \frac{D}{a-a_0}\exp\Big(-a x(\alpha) -b y(\alpha) -\epsilon^2\sqrt{(a-a_0)^2+b^2}\Big).
\end{equation}  
 \end{itemize}
 The same estimations hold for $S^{-}_{x, \alpha}$ and $S^{-}_{y, \alpha}$.
\end{lemma}
\begin{proof}
First, with definition (\ref{fff}) and the notation in (\ref{fe}), the estimate (\ref{truc}) can be written as 
\begin{equation}
    \label{lefths}
\Big|\int\limits_{v>0}   
\frac{\phi_2(x_\alpha^+ +iv) \gamma_2(x_\alpha^+ + iv, Y^+(x_\alpha^+ + iv))}{\gamma'_y(x_\alpha^++iv, Y^+(x_\alpha^+ + iv))} \exp\big(-\sqrt{a^2+b^2}\big(F(x_\alpha^+ +iv,\alpha) -F(x_\alpha^+, \alpha)\big)dx\Big|\leq \frac{D}{b}
\end{equation}
with $\alpha=\alpha(a,b)$.

Let be $\alpha_0 \in ]0, \pi/2]$.
   If $\alpha_0 \ne \pi/2$, let us fix $\eta>0$ \nr{so}\nb{sufficiently} small such that
  $\alpha_0 - \eta  >0$,  and $\alpha_0+\eta \leq \pi/2$.
   If $\alpha_0 =\pi/2$, let us fix any small $\eta>0$ and consider only $\alpha \in [\pi/2-\eta, \pi/2]$.

By Lemma \ref{phibound} 
  and \nr{remark}\nb{equation}~(\ref{nu})
\begin{equation}
\label{zlm}
 \sup_{v\geq  0, |\alpha-\alpha_0| \leq \eta } |\phi_2(x_\alpha^+ + iv)| < \infty.
 \end{equation}
 By the observation
  (\ref{ddd}) $\gamma'_y(x, Y^+(x))=0$  only if $x=x_{min}, x_{max}$. Then by (\ref{nu}) we have
  \begin{equation}
  \label{zlm1}
  \inf_{v \geq 0, |\alpha-\alpha_0|\leq \eta}  |\gamma'_y(x_\alpha^+ + iv, Y^+(x_\alpha^+ + iv))|>0.
  \end{equation}
Again by (\ref{dd}) and Lemma \ref{alleq} (ii) we have
  \begin{equation}
      \label{zlm2}
   \sup_{  v\geq  0, |\alpha-\alpha_0|\leq \eta } \Big| \frac{\gamma_2(x_\alpha^+ + iv, Y^+(x_\alpha^+ + iv))}{\gamma'_y(x_\alpha^++iv, Y^+(x_\alpha^+ + iv))}    \Big| <\infty .
   \end{equation}
  Finally
\begin{equation}
    \label{bb}
 |\exp\big(-\sqrt{a^2+b^2}(F(x_\alpha^+ +iv,\alpha) -F(x_\alpha^+, \alpha)) \big)|=\exp \big(- b ({\rm Re} Y^{+}(x_\alpha^++iv)- {\rm Re} Y^{+} (x_\alpha^+) ) \big).
 \end{equation}
By Lemma \ref{alleq} (ii)\nb{, the}
   function ${\rm Re} Y^{+}(x_\alpha^++iv)- {\rm Re} Y^{+} (x_\alpha^+)$ equals $0$ at $v=0$ is strictly
increasing as $v$ goes from zero to infinity. Moreover,
it grows linearly as $v\to \infty$ :  there exists a constant $c>0$ such that
  for any $\alpha$ such that  $|\alpha-\alpha_0|\leq \eta$
and any $v$ large enough
\begin{equation}
    \label{cv}
{\rm Re} Y^{+}(x_\alpha^++iv)- {\rm Re} Y^{+} (x_\alpha^+) \geq c v.
\end{equation}
  It follows from (\ref{zlm}), (\ref{zlm2}), (\ref{bb}) and (\ref{cv})
    that the left-hand side of (\ref{lefths})
    is bounded by
    $$C \int_0^{\infty}  \exp (- bcv ) dv = C \times (cb)^{-1}$$  
       with some constant $C>0$ for
   all couples $(a, b)$  with $|\alpha(a,b)-\alpha_0| \leq \eta$.

As for the  integral (\ref{truct}), we make  the change of variables
$B=b-b_0>0$. Next, we proceed exactly as we did in (\ref{truc}).
The only different detail is the elementary estimation
$\sup_{|\alpha-\alpha_0|\leq \eta, v>0} |\exp(a_0 (x_\alpha^++iv))| <\infty$.  
  We then obtain the bound
  $\frac{D'}{B} \exp (-a x(\alpha) - B y(\alpha) -\epsilon \sqrt{a^2 + B^2} )$ with some $D'>0$. Then with
   $D=D' \exp(b_0 y(\alpha))$ the estimation (\ref{truct}) follows.

The proofs for (\ref{truc1}) and (\ref{tructt}) are symmetric.  
\end{proof}

The previous lemma will be useful in Section~\ref{sec:asymptoticshiftedcontour} in establishing the \nr{asymptotic}\nb{asymptotics} when $\alpha_0\in]0,\pi/2[$.
In the next lemma we will show the negligibility of the integrals in the two\nr{ missing} cases where $\alpha_0=0$ or $\pi/2$. This will be useful in Section~\ref{sec:asymptoticaxes}.

\begin{rem}[Pole and branching point]
\label{excl}
In the next lemma and in Section~\ref{sec:asymptoticaxes} and~\ref{sec:polemeetsaddlepoint}, we exclude the case 
$\gamma_2(x_{max}, Y^{\pm}(x_{max})) = 0$
[resp. $\gamma_1(X^{\pm}(y_{max}),y_{max}) = 0$]
such that the branching point and the pole of $\phi_2(x)$ coincides. 
This case \nr{correspond}\nb{corresponds} to $x^*=x_{max}$ [resp. $y^*=y_{max}$], i.e. $\alpha^{*}=0$ [resp. $\alpha^{**}=\pi/2$].
Note that we already obtained the asymptotics of $h_1$ and  $h_2$ in these specific cases in Proposition~\ref{prop:asympth}.
\end{rem} 
   
\begin{lemma}[Negligibility of the integrals on $S_{x,\alpha}^\pm$ and $S_{y,\alpha}^\pm$, case where $\alpha_0=0$ or $\pi/2$]    
\label{negl2}
For any $\eta>0$  small enough and any $r_0>0$ there exists a constant $D>0$ such that for any couple $(a,b)$ where $\sqrt{a^2+b^2}>r_0$  and $0<\alpha(a,b)<\eta$ we have
\begin{equation}
    \label{truc0}
\Big|\int\limits_{S^{+}_{x, \alpha}}
\frac{\phi_2(x) \gamma_2(x, Y^+(x))}{\gamma'_y(x, Y^+(x))} \exp\big(-ax -b Y^+(x)\big)dx \Big| \leq D \exp\Big(-a x(\alpha) -b y(\alpha) - \epsilon^2 \sqrt{a^2+b^2}  \Big).
\end{equation}
Furthermore, if $b>b_0$ we have
\begin{equation}
\label{trucpi2}
\Big| \int\limits_{S^{+} _{x, \alpha}}  
\exp((a_0-a) x + (b_0-b) Y^{+}(x))
  \frac{dx}{\gamma'_y(x, Y^{+}(x)) } \Big| \leq  
   D \exp\Big(-a x(\alpha) -b y(\alpha) -\epsilon^2\sqrt{a^2+(b-b_0)^2}\Big).
\end{equation}  
The same estimations hold for $S^{-}_{x, \alpha}$.
For any couple $(a,b)$ such that $\sqrt{a^2+b^2}>r_0$  and $0< \pi/2-\alpha(a,b)<\eta$, a symmetric result holds for the integrals on $S^{+}_{y, \alpha}$ and $S^{-}_{y, \alpha}$.
\end{lemma}
 
   \begin{proof}
Let $\alpha_0=0$ so that $x(\alpha_0)=x_{max}$.
  Our aim is to prove (\ref{truc0}), which is then reduced to the estimate  
\begin{equation}
    \label{integral}   
 \Big|\int_{v>0} \frac{ \phi_2(x_\alpha^++ iv) \gamma_2(x_\alpha^+ +iv, Y^+(x_\alpha^++iv))  }{\gamma'_y (x_\alpha^+ +iv, Y^+(x_\alpha^++iv))}
 \exp\big(-aiv - b (Y^+(x_\alpha^+ +iv) -Y^+(x_\alpha^+)) \big) dv \Big| \leq D.
\end{equation}
 Let us fix any $\eta>0$ small enough and consider
  $\alpha \in ]0, \eta ]$.  By (\ref{ddd}) the denominator
  $\gamma'_y(x, Y^{+}(x))$ \nr{has}\nb{is} zero at $x=x_{max}$ but not at other points in a neighborhood of $x_{max}$.
    Then by (\ref{zoi}) we have
 \begin{equation}
     \label{cxc}
 \inf_{0\leq \alpha \leq \eta}|\gamma'_y(x_\alpha^+, Y^{+} (x_\alpha^+))|>0.
 \end{equation}
 The function $\phi_2(x)$ has a branching
point at $x_{\max}$. But it follows from
   the representation (\ref{zio}) that it is bounded
 in a neighborhood of $x_{max}$ cut along the real segment due to Remark \ref{excl}.  
       Hence, this integral has no singularity at $v=0$  for \nr{none}\nb{any} $\alpha \in ]0, \eta]$ so that
\begin{equation}
\label{cxcx}
\sup_{ 0 \leq \alpha \leq \eta}
\frac{\phi_2(x_\alpha^+) \gamma_2(x_\alpha^+, Y^+(x_\alpha^+ ))  }{\gamma'_y (x_\alpha^+, Y^+(x_\alpha^+))}  <\infty.
\end{equation}

Let us consider the \nr{asymptotic}\nb{asymptotics} of the integrand  in (\ref{integral}) as $v \to \infty$.
   It is clear that $Y^+(x_\alpha^+ + iv)$  grows linearly as $v \to \infty$ and so do
functions $\gamma_2$ and $\gamma'_y$ of this argument. \nb{The function}\nr{Function}
$\phi_2(x_\alpha +iv)$ is defined by the formula of the analytic continuation
\begin{equation}
\label{repres}
 \phi_2(x_\alpha + iv)= -\frac{ \gamma_1  (x_\alpha^++iv, Y^-(x_\alpha^+ + iv)) \phi_1 (Y^- (x_\alpha^++iv)) + \exp \big( a_0(x_\alpha^+ +iv) +
   b_0  Y^-(x_\alpha^+ +iv) \big) }{ \gamma_2 (x_\alpha^+ +iv, Y^-(x_\alpha^+ + iv)) }.
  \end{equation}  
 We know that  $Y^-(x_\alpha^+ +iv)$  varies linearly as $v \to \infty$,
 and  moreover $Re Y^- (x_\alpha+iv) \leq -c_1-c_2 v$ for all $v \geq 0$ and $\alpha \in ]0, \eta] $ with some $c_1,c_2>0$.
Then by Lemma~\ref{lem:C1} in Appendix \ref{appendixC}
\begin{equation}
\label{es11}
|\phi_2(x_\alpha^+ +iv)| \leq  C v^{\lambda-1}     
\end{equation}
  for any $\alpha \in ]0, \eta]$ and $v>V_0$  with some $C>0$, $V_0>0$ and $\lambda<1$.
 Hence, the integrand
 $$\frac{\phi_2(x_\alpha^+ + iv) \gamma_2(x_\alpha^+ +iv, Y^+(x_\alpha^++iv))  }{\gamma'_y (x_\alpha^+ +iv, Y^+(x_\alpha^++iv))}$$  
   is\nr{about} $O(v^{\lambda-1})$  as $v \to \infty$.
The positivity of ${\rm Re} Y^+(x_\alpha^+ +iv) -{\rm Re} Y^+(x_\alpha^+)$ for any $v\geq 0$
   and the inequality (\ref{cv}) in the exponent stay valid for any $\alpha \in ]0, \eta]$,
so that the exponential term is bounded in \nr{the }absolute value by $\exp(-c b v )$  with some $c>0$.
   But for $\eta$ small enough, the assumption $\alpha(a,b) \in ]0, \eta]$ implies the arbitrary smallness of $b$.
In the limiting case $b=0$ the integral in the l.h.s of (\ref{integral}) is not absolutely convergent.
In order to prove the required estimate (\ref{integral}), we proceed by integration by parts. This integral equals
 \begin{equation}
\label{ipp}
\frac{\phi_2(x_\alpha^+ + iv) \gamma_2(x_\alpha^+ +iv, Y^+(x_\alpha^++iv))  }{\gamma'_y (x_\alpha^+ +iv, Y^+(x_\alpha^++iv)) (-ai-b
(Y^+(x_\alpha^++iv))'_v) }\exp\Big( -aiv - b (Y^+(x_\alpha^+ +iv) -Y^+ (x_\alpha)) \Big) \Bigm|_{v=0}^{v=\infty}
 \end{equation}
 \begin{equation}
 \label{deriv}
 - \int_0^\infty \Big(
\frac{\phi_2(x_\alpha^+ + iv) \gamma_2(x_\alpha^+ +iv, Y^+(x_\alpha^++iv))  }{\gamma'_y (x_\alpha^+ +iv, Y^+(x_\alpha^++iv)) (-ai-b
(Y^+(x_\alpha^++iv))'_v)}  \Big)'_v  \exp\big(  -aiv - b (Y^+(x_\alpha^+ +iv) - Y^+ (x_\alpha^+))  \big) dv.        
\end{equation}    
Note that although in this case \nr{$x_{\alpha_0}= x_{max}$}\nb{$x^+_{\alpha_0}= x_{max}$} which is a branching point for $Y^+(x)$, the first and second derivatives are bounded
\begin{equation}
\label{bsb}
\sup_{\alpha \in [0, \eta]} \Big|Y(x_\alpha^+  +iv)'\Bigm|_{v=0} {\Big|} <\infty, \ \ \sup_{\alpha \in [0, \eta]} \Big|Y(x_\alpha^+  +iv)''\Bigm|_{v=0} {\Big|} <\infty
\end{equation}
by remark (\ref{zoi}). Furthermore, $Y^\pm (x_\alpha^+ +iv)'$  is of the constant order
and $Y^\pm (x_\alpha^+ +iv)''$  is not greater than $O(1/v)$ as $v \to \infty$.  
 
 The term (\ref{ipp}) at $v=0$ is bounded in\nr{ the} absolute value by some constant due to (\ref{cxcx}) and (\ref{bsb}).
 It converges to zero
 as $v \to \infty$ by the statements above for any $\alpha \in [0, \infty]$, $a,b\geq 0$.
To evaluate (\ref{deriv}), we compute the derivative in its integrand and show that it is of order
  $O(v^{\lambda-2})$ as $v \to \infty$. We skip the technical details of this computation but outline the fact that
 $\phi_2(x_\alpha^++iv)'_v$ is computed via the representation (\ref{repres}) and  $|\phi_1( Y^-(x_\alpha^++iv))'_v|$  is evaluated
again by Lemma~\ref{lem:C1}. Namely, it is of order not greater than $O(v^{\lambda-2})$ as $v \to \infty$.  
     Thus, the integral (\ref{deriv}) is absolutely convergent for any $a, b\geq 0$ and can be
bounded by some constant as well.  This finishes the proof of (\ref{truc0}).
  The proof of  (\ref{trucpi2}) is symmetric.
\end{proof}
 
Note that the proof of Lemma~\ref{pp} essentially uses the result of Lemma~\ref{phibound} which bounds the Laplace transforms.
The proof of Lemma \ref{negl2} uses a stronger result stated in Appendix~\ref{appendixC} which gives a more precise estimate of the Laplace transform near infinity. 

Following the lines of the proof we could establish a better estimate, namely the one that the integral is bounded by some
  universal constant divided by $a$, but we do not need it for our purposes.

\begin{rem}[Negligibility]
When $\alpha(a,b) \to \alpha_0 \in ]0, \pi/2[$, Equations~\eqref{truc}, \eqref{truct}, \eqref{truc1}, \eqref{tructt} of Lemma~\ref{pp} give quite satisfactory estimates which prove the negligibility of the integrals on the contours $S_{x,\alpha}^\pm$ and $S_{y,\alpha}^\pm$ with respect to integrals on contours $\Gamma_{x,\alpha}$ and $\Gamma_{y,\alpha}$, see
Lemma~\ref{nn} below.  
   In fact 
 $$ \frac{ \exp(-a x(\alpha) -b y (\alpha)  -\epsilon^2 \sqrt{a^2+b^2} ) }{b} 
   = o\Big(\frac{\exp(-a x(\alpha)-b y(\alpha)) }{\sqrt[4]{a^2+b^2}} \Big),$$ 
   $$\frac{ \exp(-a x(\alpha) -b y (\alpha)  -\epsilon^2 \sqrt{a^2+b^2} ) }{a} 
   = o\Big(\frac{\exp(-a x(\alpha)-b y(\alpha)) }{\sqrt[4]{a^2+b^2}} \Big) .  $$  
When $\alpha(a,b) \to 0 $ or $ \pi/2$, Equations~\eqref{truc0} and~\eqref{trucpi2} of Lemma~\ref{negl2} give satisfactory estimates which prove the negligibility which will be useful in Section~\ref{sec:asymptoticaxes} when computing the asymptotics along the axes.
\end{rem}

\section{Essential part of the asymptotic and main theorem} 
\label{sec:asymptoticshiftedcontour}

This section is dedicated to the asymptotics of $g(a,b)=I_1+I_2+I_3$ when $\alpha(a,b) \to \alpha_0 \in ]0, \pi/2[$.
The next lemma determines the asymptotics of the integrals on the lines of steepest descent $\Gamma_{x,\alpha}$ and $\Gamma_{y,\alpha}$ of the shifted contours.

For any couple $(a,b) \in {\mathbb{R}}_+^2$  we define 
$\alpha(a,b)$ as the angle in $[0,\pi/2]$ such that $\cos (\alpha) = \frac{ a }{ \sqrt{a^2+b^2}}$  and  $\sin(\alpha) = \frac{ b }{\sqrt{a^2+ b^2}}$ and we define $r\in\mathbb{R}_+$  such that $r=\sqrt{a^2+b^2}$.
\begin{lemma}[Contribution of the saddle point to the asymptotics]
\label{nn} 
  Let $\alpha_0 \in ]0, \pi/2[$. Let $\alpha(a,b) \to \alpha_0$ and $r=\sqrt{a^2+b^2} \to \infty$.  
Then for any $n \geq 0$ we have
$$
 \frac{1}{ 2 \pi i } \int\limits_{\Gamma_{x, \alpha}} 
\frac{\phi_2(x) \gamma_2(x, Y^+(x))}{\gamma'_y(x, Y^+(x))} \exp(-ax -b Y^+(x))dx +
 \frac{1}{2\pi i } \int\limits_{\Gamma_{y, \alpha}} 
\frac{\phi_1(y) \gamma_1(X^+(y), y)}{\gamma'_x(X^+(y), y)} \exp(-aX^+(y) -b y)dy $$
$$ +
\frac{1}{2\pi i } \int\limits_{\Gamma_{x, \alpha}}   \exp((a_0-a) X^+(y) + (b_0-b) y) 
  \frac{dy}{\gamma'_x(X^+(y), y) } \  $$
\begin{equation}
\label{sss} 
\sim  \exp(-ax(\alpha(a,b)) -b y (\alpha(a,b))) 
  \sum_{k=0}^n \frac{c_k(\alpha(a,b))}{ \sqrt[4]{a^2+b^2}(a^2+b^2)^{k/2}}
 \end{equation} 
with some constants $c_0(\alpha), c_1(\alpha), \ldots, c_n(\alpha)$ continuous at $\alpha_0$. 
Namely 
\begin{equation}
c_0(\alpha) =  \frac{ \gamma_1(x(\alpha), y(\alpha)) \phi_1(y(\alpha))+ \gamma_2(x(\alpha), y(\alpha)) \phi_2(x(\alpha))+ \exp(a_0 x(\alpha)+ b_0 y(\alpha))}
{ \sqrt{2\pi  ( \sigma_{11}\sin^2(\alpha) + 2\sigma_{12} \sin (\alpha) \cos (\alpha) + \sigma_{22} \cos^2(\alpha))} } \times C(\alpha),
\label{eq:c0col}
\end{equation}
where $$C(\alpha)= \sqrt{ \frac{ \sin(\alpha)}{ \gamma'_y(x(\alpha), y(\alpha))   }} = \sqrt{ \frac{ \cos(\alpha) }{ \gamma'_x( x(\alpha), y(\alpha) } }. $$
\end{lemma}

\begin{proof}
Consider the first integral. We make the change of variables $x=x(it, \alpha)$, see Section~\ref{sec:saddlepoint} and Appendix~\ref{sec:morse}. Then \nr{it}\nb{the sum of integrals} becomes 
$$\frac{\exp(-ax(\alpha) - by(\alpha))}{2\pi}\int_{-\epsilon}^\epsilon 
 f(it, \alpha) \exp(-\sqrt{a^2+b^2} t^2)dt$$
where 
  $$f(it, \alpha)= \frac{ \phi_2(x(it, \alpha)) \gamma_2(x(it, \alpha), Y^+(x(it, \alpha)) )  }{ \gamma'_y(x(it, \alpha), Y^+(x(it, \alpha))) } x'_\omega(it, \alpha).$$ 
We take $\Omega(\alpha_0)$ from Lemma \ref{Morse} where $K$ and $\eta$ are defined in this lemma. 
 For any 
$\alpha \in [\alpha_0-\eta,\alpha_0+ \eta]$ and $t \in [-\epsilon, \epsilon]$
we have 
$$\Big| f(it, \alpha) -  \sum_{l=0}^{2n} f^{(l)} (0, \alpha)\frac{(it)^l}{l!}\Big| \leq C |t|^{2n+1}$$
  with the constant
 $$C= \sup_{|\omega|=K, \atop |\alpha-\alpha_0|\leq \eta} 
     \Big| \frac{ f(\omega, \alpha) -  \sum_{l=0}^{2n} f^{(l)} (0, \alpha)\frac{\omega^l}{l!} } 
        {\omega^{2n+1} } \Big|  $$ 
  by the maximum modulus principle and the fact that 
  $f(\omega, \alpha)$  is in class $C^{\infty}$ 
  in $\Omega(\alpha_0)$.
The integral 
  $$\int_{-\epsilon}^{\epsilon} t^{l} \exp (-\sqrt{a^2+b^2} t^2)dt$$
equals $0$ if $l$ is odd. 
By the change of variables $s=\sqrt[4]{a^2+b^2}t$ it equals 
$$
\frac{(l-1)(l-3)...(1)}{2^{l/2}}
\frac{\sqrt{\pi} }{ (\sqrt[4]{a^2+b^2})^{l+1} }  + 
  O\Big( \frac{ \exp(-\sqrt{a^2+b^2} \epsilon) }{  (\sqrt[4]{a^2+b^2})^{l+1}}\Big),  \  \   \sqrt{a^2+b^2} \to \infty$$ 
if $l$ is even. 
The constant comes from the fact that $\int_{-\infty}^{+\infty}t^le^{-s^2}ds = \frac{(l-1)(l-3)...(1)}{2^{l/2}}\sqrt{\pi}$. 
By the same reason 
 $$\int_{-\epsilon}^{\epsilon} |t|^{2n+1} \exp (-\sqrt{a^2+b^2} t^2)dt =  O \Big( \frac{ 1 }{ (\sqrt[4]{a^2+ b^2})^{2n+2} } 
   \Big),  \  \  \sqrt{a^2+b^2} \to \infty.$$ 
The representation (\ref{sss}) for the first integral follows with the constants 
 $$c_l^1(\alpha) = \frac{(l-1)(l-3)...(1)}{2^{l/2}} \frac{\sqrt{\pi}}{2\pi}  \frac{(-1)^l f^{(2l)}(0, \alpha)}{ (2l)!}.$$  In particular 
$$c^1_0(\alpha) = \frac{1}{2 \sqrt{\pi} } 
\times \frac{ \gamma_2(x(\alpha), y(\alpha)) \phi_2(x(\alpha))}{ \gamma'_y(x(\alpha), y(\alpha))} \times x'_\omega(0, \alpha). $$ 
 Using the expressions (\ref{zop}) and (\ref{fzfz}), we get 
$$c^1_0(\alpha)= \frac{\gamma_2(x(\alpha), y(\alpha)) \phi_2(x(\alpha))} 
{ \sqrt{2\pi  ( \sigma_{11}\sin^2(\alpha) + 2\sigma_{12} \sin (\alpha) \cos (\alpha) + \sigma_{22} \cos^2(\alpha)) }} \times 
\sqrt{ \frac{ \sin(\alpha)}{ \gamma'_y(x(\alpha), y(\alpha)) }}.
$$ 
In the same way, using the variable $y$ instead of $x$, 
 we get the asymptotic expansions of the second and the third integral with constants $c^2_0(\alpha), \ldots, c_{n}^2(\alpha), c^3_0(\alpha), \ldots, c^3_n(\alpha)$. 
 Namely, 
$$c^2_0(\alpha) + c^3_0(\alpha)= \frac{\gamma_1(x(\alpha), y(\alpha)) \phi_1(y(\alpha)) + \exp(a_0 x(\alpha) + b_0 y(\alpha)) } 
{ \sqrt{2\pi  ( \sigma_{11}\sin^2(\alpha) + 2\sigma_{12} \sin (\alpha) \cos (\alpha) + \sigma_{22} \cos^2(\alpha))} } \times 
\sqrt{ \frac{ \cos(\alpha)}{ \gamma'_x(x(\alpha), y(\alpha))   }} .$$ 
By (\ref{all})  
$\sin(\alpha) \gamma'_x(x(\alpha), y(\alpha)) = \cos (\alpha) \gamma'_y (x(\alpha), y(\alpha))$.  This implies the representation (\ref{sss}) and concludes the proof with $c_k(\alpha)=\sum_{i=1}^3 c_k^i(\alpha)$. 
  \end{proof}

We will justify later that the constants $c_0(\alpha)$ are not zero. We now turn to the main result of the paper.

\begin{theorem}[Asymptotics in the quadrant, general case]
We consider a reflected Brownian motion in the quadrant of parameters $(\Sigma, \mu,  R)$ satisfying conditions of Proposition \ref{R_gentil} and Assumption~\ref{drift_positif}.  Then, the Green's density function $g(r\cos(\alpha), r\sin(\alpha))$ of this process has the following asymptotics for all $n \in \mathbb{N}$ when $\alpha\to\alpha_0\in(0,\pi/2)$ and $r\to\infty$:
\begin{itemize}
\item If $\alpha^*<\alpha_0<\alpha^{**}$ then
\begin{equation}\label{cas_1}
g(r\cos(\alpha), r\sin(\alpha))
\underset{r\to\infty \atop\alpha\to\alpha_0}{\sim} 
e^{-r(\cos(\alpha)x(\alpha) + \sin(\alpha)y(\alpha))} \frac{1}{\sqrt{r}}
  \sum_{k=0}^n \frac{c_k(\alpha)}{ r^{k}}.
  \end{equation}

  \item If $\alpha_0<\alpha^{*}$ then
\begin{equation}\label{cas_2}
g(r\cos(\alpha), r\sin(\alpha))
\underset{r\to\infty \atop\alpha\to\alpha_0}{\sim} 
c^{*}e^{-r(\cos(\alpha)x^* + \sin(\alpha)y^*)}
+
e^{-r(\cos(\alpha)x(\alpha) + \sin(\alpha)y(\alpha))}\frac{1}{\sqrt{r}}
  \sum_{k=0}^n \frac{c_k(\alpha)}{ r^{k}}.
  \end{equation}
  \item If $\alpha^{**}<\alpha_0$ then
\begin{equation}\label{cas_3}
g(r\cos(\alpha), r\sin(\alpha))
\underset{r\to\infty \atop\alpha\to\alpha_0}{\sim}
c^{**} e^{-r(\cos(\alpha)x^{**} + \sin(\alpha)y^{**})}
+
e^{-r(\cos(\alpha)x(\alpha) + \sin(\alpha)y(\alpha))} \frac{1}{\sqrt{r}}
  \sum_{k=0}^n \frac{c_k(\alpha)}{r^{k}}
  \end{equation}
\end{itemize}
where explicit expressions of the saddle point coordinates $x(\alpha)$ and $y(\alpha)$ are given by (\ref{degueu1}) and (\ref{degueu2}), the coordinates of the poles $x^*$, $y^*$, $y^{**}$, $x^{**}$ are given by \eqref{eq:x*y**} and \eqref{eq:x**y*}, and the constants are given by 
$$c^* = \frac{(-res_{x = x^*}\phi_2(x))\gamma_2(x^*,y^*)}{\gamma'_y(x^*, y^*)} >0
\quad\text{and}\quad
c^{**} = \frac{(-res_{y = y^{**}}\phi_1(y))\gamma_1(x^{**},y^{**})}{\gamma'_y(x^{**}, y^{**})} >0
$$
where the $c_k$ are constants depending on $\alpha$ and such that $c_k(\alpha)\underset{\alpha\to\alpha_0}{\longrightarrow} c_k(\alpha_0)$ where $c_0(\alpha)$ is given by~\eqref{eq:c0col}. We have $c_0(\alpha) > 0$ at least when $\alpha^*<\alpha_0<\alpha^{**}$ where it gives the dominant term of the asymptotics in~\eqref{cas_1}.
\label{thm4}
\end{theorem}

\begin{proof}
The theorem follows directly from combining several lemmas.
By Lemma \ref{3_integ} the inverse Laplace transform $g(a,b)$ can be expressed as of the sum of three simple integrals $I_1 + I_2 + I_3$. Those integrals have been rewritten in Lemma~\ref{residus}\nr{ ,thanks to} \nb{by} the residue theorem\nr{,} as the sum of residues and integrals whose contour locally follows  the steepest descent line through the saddle point. This has been done in Section~\ref{sec:shift} using Morse's Lemma, see Appendix \ref{sec:morse}. Residues are present if $0< x^* < x(\alpha)$ or $0< y^{**} < y(\alpha)$. In addition, we proved in Lemma~\ref{pp} the negligibility of the integrals of the lines $S^{\pm}_{x,\alpha}$ and $S^{\pm}_{y,\alpha}$ compared to the integrals on the steepest descent lines. The main asymptotics are then given by the poles plus the asymptotics of the steepest descent integrals. A disjunction of cases concerning the pole's contributions gives the three cases of the theorem (recall that $\alpha^* < \alpha^{**}$). In the second case, when $\alpha_0 < \alpha^*$, $\phi_2$ has a pole and then $c^* \neq 0$ because we have $\frac{r_{12}}{r_{22}} > \frac{-Y^{\pm}(x_{max})}{x_{max}}$ which implies $\gamma_2(x^*, y^*) \neq 0$. The same holds for $c^{**}$. 
Finally, Lemma \ref{nn} gives the desired asymptotic expansion of the integrals on the lines of the steepest descent. The fact that $c_0(\alpha_0) \neq 0$ when $\alpha^*<\alpha_0<\alpha^{**}$ is postponed to Lemma~\ref{nonnul} and Lemma~\ref{lem:c0nonnul}.
\end{proof}

The constants $c_0(\alpha)$ shall not be zero at least when $\alpha^*<\alpha_0<\alpha^{**}$, that is when the poles are not involved in the \nr{asymptotic}\nb{asymptotics}. We divide the proof into two lemmas. 

Most of the quantities studied so far depend on the starting point of the process, even if this dependence is not explicit in the notation. In the following, we add a power $z_0$ (or $(a_0, b_0)$) in the notation of the objects which correspond to a process whose starting point is $z_0=(a_0, b_0)$. For example, we will note $h_1^{z_0}$ or $\phi_1^{z_0}$ when we want to emphasise the dependency on the starting point. 

\begin{lemma}[Non nullity of the constant \nr{$c(\alpha)$}\nb{$c_0(\alpha)$} for at least a starting point]\label{nonnul}
    If $\alpha \in \left(0, \frac{\pi}{2}\right)\setminus{\{\alpha^*,\alpha^{**}\}}$, there exists some starting point $z_0 \in \R_+^2$ such that $c_0^{z_0}(\alpha) \neq 0$.
\end{lemma}
\begin{proof}
Let $z_0=(a_0, b_0)$ the starting point of the process. We proceed by contradiction assuming that $c_0^{(a_0, b_0)}(\alpha) = 0$ for all $a_0, b_0 \geq 0.$ Since $x(\alpha) \leq 0$ or $y(\alpha) \leq 0$, we suppose without loss of generality that $y(\alpha) \leq 0.$  We have then, by \eqref{eq:c0col} and the continuation formula:
    \begin{equation}\label{azerty}
        c_1\varphi_1^{(a_0,b_0)}(y(\alpha)) - c_2\varphi_1^{(a_0,b_0)}(Y^-(x(\alpha))) = \gamma_2(x(\alpha), y(\alpha))e^{a_0x(\alpha)+b_0Y^-(x(\alpha))} - \gamma_2(x(\alpha), Y^-(x(\alpha)))e^{a_0x(\alpha)+b_0y(\alpha)}
    \end{equation} 
    with $c_1 = \gamma_1(x(\alpha), Y^-(x(\alpha)))\gamma_2(x(\alpha), y(\alpha))$ and $c_2 = \gamma_1(x(\alpha), Y^-(x(\alpha)))\gamma_2(x(\alpha), Y^-(x(\alpha)))$. We remark that $\gamma_2(x(\alpha), Y^-(x(\alpha))) \neq 0$ since we have assumed $\alpha\neq \alpha^*$.
The right term of~\eqref{azerty} is unbounded on the set of all $(a_0, b_0)$ belonging to $\R_+^2$ since $Y^-(x(\alpha)) < y(\alpha) = Y^+(x(\alpha)) $. Then, it is sufficient to show that the supremum of the left term is bounded according to $(a_0, b_0)$. We denote by $h_1^{(a_0,b_0)}$ the density of $H_1$ according to the Lebesgue measure corresponding to the starting point $(a_0, b_0)$. We have \nb{then}
    \begin{equation}
         c_1\varphi_1^{(a_0,b_0)}(y(\alpha)) - c_2\varphi_1^{(a_0,b_0)}(Y^-(x(\alpha))) = \int_0^\infty \left(c_1e^{y(\alpha)z} - c_2e^{Y^-(x(\alpha))z}\right)h_1^{(a_0,b_0)}(z)dz =: I.
    \end{equation} 
Similarly to the proof of Lemma \ref{existencepole}, we introduce $T$ as the first hitting time of the axis $\{x = 0\}$. By the strong Markov property,
    we obtain in the same way: 
    \begin{align}
        I&= \E_{(a_0, b_0)}\left[\fc_{T < +\infty}\E_{(0, Z^2_T)}\left[\int_0^{+\infty}\fc_{\{0\}\times\R_+}(Z_t)\left(c_1e^{y(\alpha)Z^2_t} - c_2e^{Y^-(x(\alpha))Z^2_t}\right)dL^1_t\right]\right]\\
        &=\int_0^{+\infty}\int_0^{+\infty}\left(c_1e^{y(\alpha)z} - c_2e^{Y^-(x(\alpha))z}\right)h_1^{(0,y)}(z)dz\P(T < +\infty, Z^2_T = dy)\\
        &= \int_0^{+\infty} \left(c_1\varphi_1^{(0,y)}(y_\alpha) - c_2\varphi_1^{(0,y)}(Y^-(x_\alpha))\right)\P(T < +\infty, Z^2_T = dy). \label{8.10}
    \end{align}
    Using the identity (\ref{azerty}) in (\ref{8.10}) (where we see the relevance of going to the $y$-axis), we get the bound 
    \begin{equation}\label{clarinette}
    |I| \leq |\gamma_2(x(\alpha), y(\alpha))| + |\gamma_2(x(\alpha), Y^-(x(\alpha)))|
    \end{equation}
since $y(\alpha) \leq 0.$
The right term of (\ref{azerty}) is therefore bounded in $(a_0, b_0)$, and thus a contradiction has been reached. This completes the proof.
\end{proof}

\begin{lemma}[Non nullity of the constant \nr{$c(\alpha)$}\nb{$c_0(\alpha)$} for all starting points] \label{lem:c0nonnul}
    For all $\alpha \in \left(0, \frac{\pi}{2}\right)$ such that $\alpha^* < \alpha < \alpha^{**}$ and $z_0 \in \R_+^2$, we have $c_0^{z_0}(\alpha) \neq 0$.
\end{lemma}
\begin{proof}
\nr{Let's denote}\nb{Denote} $z_0 = (a_0, b_0)$ the point obtained in Lemma \ref{nonnul} such that $c_0^{z_0}(\alpha) \neq 0$. \nr{We use again the {continuity of the Laplace transform} in $z_0$ (see the proof of Lemma~\ref{lem:poleallz0}) to remark that}\nb{By continuity of the Laplace transforms $\phi_1^{z_0}$ and $\phi_2^{z_0}$ in $z_0$ (see the proof of Lemma~\ref{lem:poleallz0})} $c^{z'_0}(\alpha) \neq 0$ for all $z'_0$ in an open neighborhood $V$ of $z_0.$ Let $z''_0 \in \R_+^2$ be the starting point of the process $Z^{(z''_0)}$ and let $T = \inf\{t \geq 0, Z^{(z''_0)}_t \in V\}$ be the hitting time of $V$. We have $\P_{z''_0}(T < +\infty) = p > 0$. By the strong Markov property,
\begin{align}
        g^{z''_0}(r\cos(\alpha),r\sin(\alpha))& \geq p\inf_{z'_{0}\in V}g^{z'_0} (r\cos(\alpha),r\sin(\alpha))\\
        &\geq p\inf_{z'_{0}\in V}\left[c_0^{z'_0}(\alpha) (1 + o_{r\to\infty}(1))\right] e^{-r(\cos(\alpha)x(\alpha) + \sin(\alpha)y(\alpha))} \frac{1}{\sqrt{r}}.
    \end{align}
    Furthermore, $V$ can be chosen bounded and such that $\inf_{z'_{0}\in V}c_0^{z'_0}(\alpha) > 0$.
The issue is that the term $o_{r\to\infty}(1)$ may depend on $z'_0$. We then refer to the proof of Lemma~\ref{pp}. \nb{We remark} that the only quantity depending on the initial condition is the constant $D$ of Lemma~\ref{pp}, which is based on Lemma~\ref{phibound}. If the supremum on $z'_0 \in V$ of the quantity of Lemma~\ref{phibound} is finite, then the result holds. This fact is verified easily from the proof of this lemma, because $V$ is bounded and $\varphi_1^{z'_0}(0)$ is continuous in $z'_0.$
\end{proof}

\section{Asymptotics along the axes : \texorpdfstring{$\alpha\to 0$}{alpha tends to 0} or \texorpdfstring{$\alpha\to \frac{\pi}{2}$}{alpha tends to pi/2}}
\label{sec:asymptoticaxes}

In this section, we study the asymptotics of the Green's function $g$ along the axes. We recall the assumptions $\alpha^*\neq 0$ and $\alpha^{**}\neq \pi/2$ made in Remark~\ref{excl}.

Let us recall that for any couple $(a,b) \in {\mathbb{R}}_+^2$  we define $r=\sqrt{a^2+b^2}$ and
$\alpha(a,b)$ as the angle in $[0,\pi/2]$ such that $\cos (\alpha) = \frac{ a }{ \sqrt{a^2+b^2}}$  and  $\sin(\alpha) = \frac{ b }{\sqrt{a^2+ b^2}}$.

\begin{lemma}[Contribution of the saddle point to the asymptotics when $\alpha\to 0$ or $\pi/2$]
\label{nnn} 
\
\begin{itemize} 
\item[(i)] 
Let $a \to \infty$, $b>0$  and $\alpha(a, b) \to 0$.
Then the asymptotics of \eqref{sss} \nr{remains}\nb{remain} valid with
  $c_0(\alpha)\to 0$ as $\alpha \to 0$. Moreover, we have  
 $c_0(\alpha) \sim c'\alpha $ and
 $c_1(\alpha) \sim c''$ as $\alpha \to 0$ where $c'$ and $c''$ are non-null constants at least when \nr{$\alpha^* < 0$}\nb{$\alpha^* = -\infty$ (i.e. when there is no pole for $\phi_2$)}.
\item[(ii)] When $b \to \infty$, $a>0$ 
 and $\alpha(a, b) \to \pi/2$ the same result holds. 

\end{itemize} 
\end{lemma}
 
\begin{rem}[Competition between the two first term of the asymptotics]
The previous lemma states that when $\alpha \to 0$ and $r\to\infty$, there is a competition between the first two terms of the sum of the asymptotic development given in \eqref{sss}. Namely, the first term
$\frac{c_0(\alpha)}{\sqrt{r}} \sim \frac{c' \alpha}{\sqrt{r}}\sim \frac{c' b}{r\sqrt{r}}$ 
and the second term 
$ \frac{c_1(\alpha)}{r\sqrt{r}}\sim  \frac{c''}{r\sqrt{r}} $ 
may have the same order of magnitude.
If $b  \to 0$, the second term is dominant. If $b\to c $ where $c$ is a positive constant, they both contribute to the asymptotics.
    If $b \to \infty$ (and also $b=o(a)$ since $\alpha\to 0$), the first term is dominant. 
\end{rem}

\begin{proof} 
We first prove  (i). For any $\alpha$ close to $0$, $\Gamma_{x, \alpha}$
lies in a neighborhood of $x(\alpha)$. Using the continuation formula of $\phi_2(x)$ \eqref{zio}, the definition of $F$ \eqref{fff}, and the fact that $\Gamma_{x, \alpha}= \underrightarrow{\overleftarrow{X^+ (\Gamma_{y, \alpha})}}$ \eqref{eq:gammaXarrow}, the first integral of \eqref{sss} becomes 
$$  \frac{e^{-a x(\alpha) -b y (\alpha) } }{2i\pi } \int_{ \underrightarrow{\overleftarrow{
X^+(\Gamma_{y, \alpha} ) }  } }
\frac{ \gamma_2(x, Y^+(x)) \Big(-\gamma_1(x, Y^{-}(x)) \phi_1 (Y^-(x))-e^{a_0 x +b_0 Y^{-}(x)} \Big)}{\gamma_2(x, Y^{-}(x) ) \gamma'_y(x, Y^+(x)) } $$
$$\times \exp \big(\sqrt{a^2 + b^2} F(x, \alpha)\big) dx.
$$
  Let us make the change of variables $x=X^+(y)$. 
  Taking into account the fact that $Y^+(X^+(y))=y$,
  the relation $\gamma'_x(X^+(y), y)(X^+(y))'+ 
  \gamma'_y(X^+(y), y)\equiv 0$ and  the direction of $\underrightarrow{\overleftarrow{
X^+(\Gamma_{y, \alpha} ) }  }$, 
the first integral becomes 
 $$\frac{e^{-a x(\alpha) -b y (\alpha) } }{2i\pi } \int_{ 
\Gamma_{y, \alpha}   }
\frac{ \gamma_2(X^+(y), y) \Big(-\gamma_1(X^+(y), Y^{-}(X^{+}(y))) \phi_1 (Y^-( X^+(y)))-e^{a_0 X^+(y) +b_0 Y^{-}(X^+(y))} \Big)}{\gamma_2(X^+(y), Y^{-}(X^+(y) ) ) \gamma'_x(X^+(y), y) } $$
\begin{equation}
    \label{ruti}
\times \exp \big(\sqrt{a^2 + b^2} G(y, \alpha)\big) dy.
\end{equation} 
For the second and the third integral, we use the representation valid for $a>a_0$. 
We then have to find the asymptotics of the integral
$$\frac{e^{-a x(\alpha) -b y (\alpha) } }{2i\pi }\int_{ 
\Gamma_{y, \alpha}  }
\frac{ \gamma_2(X^+(y), Y^{-}(X^{+}(y))) H(X^+(y),y) - \gamma_2(X^+(y), y) 
  H(X^+(y), Y^{-}(X^{+}(y))) }{\gamma_2(X^+(y), Y^{-}(X^+(y) ) ) \gamma'_x(X^+(y), y) } $$
$$\times \exp \big(\sqrt{a^2 + b^2} G(y, \alpha)\big) dy
$$
   where 
$$H(X^+(y),y) = \gamma_1(X^+(y),y) \phi_1(y)+\exp(a_0 X^+(y) + b_0 y).$$
  Finally, \nr{let us note,}\nb{note} that with notation in (\ref{abc}) 
$$ Y^{-}(X^+(y)) = \frac { c(X^{+}(y))  
}{ a (X^{+}(y) )    
  \times  Y^+(X^+(y))} = \frac{\sigma_{11} (X^+(y))^2 + 2 \mu_1 X^+(y)}{ \sigma_{22} y }.$$  
 \nr{Function}\nb{The function} $X^+(y)$  is holomorphic in a neighborhood of 
 $Y^{\pm}(x_{max})$. By (\ref{dddd}) we have 
 $\gamma_x'(X^+(y), y)= \sqrt{\widetilde b^2(y)-4 \widetilde a(y) \widetilde c(y)}$  which is holomorphic in a neighborhood 
 of $Y^{\pm}(x_{max})$ and different from zero. 
   Finally, $\gamma_2(x_{\max}, Y^{\pm} (x_{max})) \ne 0$ 
 by our assumption in Remark~\ref{excl}. 
  It follows that the integrand in (\ref{ruti}) is a holomorphic function in a neighborhood of 
   $Y^{\pm}(x_{max})$. Then, \nb{we can apply} the saddle point procedure of Lemma \ref{nn} \nr{applied }to $G(y, \alpha)$ with $\alpha=0$ and where we replace the function $f(it, \alpha)$ by 
\begin{align*} f(it, \alpha) = [\gamma_2(X^+(y(it, \alpha)),& Y^{-}(X^{+}(y(it, \alpha)))) H(X^+(y(it, \alpha)),y(it, \alpha)) - \\
&\gamma_2(X^+(y(it, \alpha)), y(it, \alpha)) 
  H(X^+(y(it, \alpha)), Y^{-}(X^{+}(y(it, \alpha))))]\\
& \times\frac{y'_\omega(it, \alpha)}{\gamma_2(X^+(y(it, \alpha)), Y^{-}(X^+(y(it, \alpha)) ) ) \gamma'_x(X^+(y(it, \alpha)), y(it, \alpha))}
\end{align*}
where $y(it, \alpha)$ is the path given by the parameter-dependent Morse Lemma (\nb{see} Lemma \ref{Morse}). We get the asymptotic development \eqref{sss} as $\alpha \to 0$\nr{. We} \nb{and} then have a competition $c_0(\alpha) + \frac{c_1(\alpha)}{r} + O\left(\frac{1}{r^2}\right)$ between $c_0(\alpha) = \frac{1}{2\sqrt{\pi}} f(0, \alpha)$ and $c_1(\alpha) =  -\frac{1}{4\sqrt{\pi}}  \frac{f''_\omega(0, \alpha)}{4!}$.
When $\alpha \to 0$, we have $c_0(\alpha)\sim c' \alpha$ and $c_1(\alpha)\sim c''$ for non-null constants $c'$ and $c''$,
see Lemma~\ref{lem:c'} and Remark~\ref{rem:c''} \nr{bellow}\nb{below}.

The proof of (ii) is exactly the same, except that we use the other representation of $I_3(a,b)$. 
\end{proof}

\begin{lemma}[Non nullity of $c'$]
When $\alpha \to 0$ we have $c_0(\alpha)\sim c' \alpha$ and the constant $c'$ is non-null at least when \nr{$\alpha^* < 0$}\nb{$\alpha^* = -\infty$ (i.e. when there is no pole for $\phi_2$)}.
\label{lem:c'}
\end{lemma}
\begin{proof}
 It is clear that $c_0(0)=0$ because $c_0(\alpha)$ coincides with \eqref{eq:c0col} by uniqueness of asymptotic development, and this expression tends to $0$ as $\alpha$ goes to $0$ due to $C(\alpha)$.
Let us now consider the behaviour of $c_0(\alpha)$ when $\alpha \to 0$\nr{ remembering}\nb{. Recall} that $c_0(\alpha) = \frac{1}{2\sqrt{\pi}} f(0, \alpha)$ \nr{where we use the}\nb{with the} notation of the proof of Lemma~\ref{nnn}. \nb{Invoking}\nr{we have thanks to} Lemma \ref{alleq}, we obtain
\begin{align*}
y(\alpha)-Y^{-}(X^+(y(\alpha))) &= Y^+(X^+(y(\alpha))- Y^-(X^+(y(\alpha)))
\\ 
&= \frac{2}{\sigma_{22}}\sqrt{(\sigma_{11}\sigma_{22} -\sigma_{12}^2) (x_{max} - X^+(y(\alpha)) (X^+(y(\alpha)) -x_{min})}.
\end{align*}
\nr{We also have}\nb{We also remark that} $(X^+(y))'\Bigm|_{y=y(0)}=0$ and $(X^+(y))''\Bigm|_{y=y(0)} = -\frac{\sigma_{22}}{\gamma'_x(x_{max}, Y^{\pm}(x_{max} ))}$, so that 
$$x_{max} - X^+(y(\alpha)) = \frac{\sigma_{22}}{2\gamma'_x(x_{max}, Y^{\pm}(x_{max} ))} \alpha^2(1+o(1)), \hbox{    as   }\alpha \to 0.$$ 
Finally 
  $$y(\alpha)-Y^-(X^+(y(\alpha))) \sim  \sqrt{ \frac{ 2(\sigma_{11}\sigma_{22} -\sigma_{12}^2) (x_{max} -x_{min}) }{ 2 \sigma_{22} \gamma'_x(x_{max}, Y^{\pm} (x_{max})) }  } \times \alpha \sim \Pi \times \a ,$$
where $\Pi$ is defined as the constant in front of $\alpha$.

Since $\gamma_2(x, y)=r_{12}x + r_{22}y$ and $\gamma_2(x_{max}, Y^-(x_{max}))\gamma'_x(x_{max}, Y^-(x_{max})) \neq 0$, we obtain \nr{that when $\alpha\to 0$}
\begin{align*}c_0(\alpha) &= \frac{-r_{22} H(x_{max}, Y^{\pm} (x_{max}))\times (\Pi \alpha) + \gamma_2 (x_{max}, Y^{\pm}(x_{max})) 
  H'_y(x_{max}, Y^{\pm} (x_{max})) \times (\Pi \alpha) + o(\alpha)}{\gamma_2(x_{max}, Y^-(x_{max}))\gamma'_x(x_{max}, Y^-(x_{max})) + o(1)}  \\
&= \alpha(c' + o(1))
\end{align*}
\nb{as $\alpha\to 0$} where $c'$ is the corresponding constant.

Let us prove that $c' \neq 0$. We have to show that 
$$-r_{22} H(x_{max}, Y^{\pm} (x_{max})) + \gamma_2 (x_{max}, Y^{\pm}(x_{max})) H'_y(x_{max}, Y^{\pm} (x_{max})) \neq 0$$ i.e. that 
\begin{align*}
-r_{22} &\Big(\gamma_1(x_{max}, Y^{\pm} (x_{max}))\phi_1(Y^{\pm}(x_{max})) + e^{a_0x_{max} + b_0Y^{\pm}(x_{max})} \Big) + \gamma_2 (x_{max}, Y^{\pm}(x_{max}))\times\\
&\Big(r_{21}\phi_1(Y^{\pm}(x_{max})) + \gamma_1(x_{max}, Y^{\pm}(x_{max}))\phi_1'(Y^{\pm}(x_{max}))  + b_0e^{a_0x_{max} + b_0Y^{\pm}(x_{max})}\Big) \neq 0.
\end{align*}
The equation can be rewritten as 
\begin{equation} \label{borneoupas}
c_1\phi_1(Y^{\pm}(x_{max})) + c_2 \phi'_1(Y^{\pm}(x_{max})) \neq (c_3 + c_4 b_0)e^{a_0x_{max} + b_0Y^{\pm}(x_{max})} 
\end{equation}
with $c_1, c_2, c_3, c_4$ constants not depending on the initial conditions. Note that $c_3 = -r_{22} \neq 0$ by \eqref{eq:condexist} and $c_4 = \gamma_2(x_{max}, Y^{\pm}(x_{max})) \neq 0$ by the assumption in Remark~\ref{excl}. Furthermore, with the same method employed in the proof of Lemmas~\ref{existencepole} and~\ref{nonnul}, the left term of \eqref{borneoupas} is bounded in $(a_0, b_0)$.  
Since $x_{max} > 0$ and $Y^{\pm}(x_{max}) < 0$, the right term of \eqref{borneoupas} is not bounded in $(a_0, b_0)$. Hence, \eqref{borneoupas} holds for at least one $(a_0, b_0)$.
By a similar argument developed in the proof of Lemmas~\ref{lem:poleallz0} and~\ref{lem:c0nonnul} (using the fact that \nr{$\alpha^* < 0$}\nb{$\alpha^* = -\infty$}), \nr{we show that $c' \neq 0$ only for one starting point $(a_0, b_0)$ imply that} \nb{since $c' \neq 0$ at least for one starting point $(a_0, b_0)$,} $c' \neq 0$ for all starting points. 
Finally, \eqref{borneoupas} holds for every initial condition. This concludes the proof \nr{of the fact }that $c_0(\alpha)\sim c' \alpha$ for a non-null constant $c'$. 
\end{proof} 

\begin{rem}[Non nullity of $c''$]
\label{rem:c''}
We note here that $c'' \neq 0$. A proof inspired by what has been done in the previous lemma to show that $c' \neq 0$ would work. The same techniques have also been employed in Lemmas~\ref{existencepole} and~\ref{lem:poleallz0} to characterize the poles by showing the non nullity of a constant, and in Lemmas~\ref{nonnul} and~\ref{lem:c0nonnul} to show the non nullity of $c_0(\alpha)$.
\end{rem}

We now have everything we need to prove our second main result, which states the full asymptotic expansion of the Green's function $g$ along the edges.

\begin{theorem}[Asymptotics along the edges for the quadrant]
We now assume that $\alpha_0=0$ and let $r\to\infty$ and $\alpha\to \alpha_0=0$. In this case, we have $c_0(\alpha)\underset{\alpha\to 0}{\sim} c' \alpha$ and $c_1(\alpha)\underset{\alpha\to 0}{\sim} c''$ for some constants $c'$ and $c''$ which are non-null at least when \nr{$\alpha^* < 0$}\nb{$\alpha^* = -\infty$ (i.e. when there is no pole for $\phi_2$)}.
Then, the Green's function $g(r\cos(\alpha), r\sin(\alpha))$ has the following asymptotics: 
\begin{itemize}
\item When $\alpha^{*} < 0$ the asymptotics given by \eqref{cas_1} \nr{remains}\nb{remain} valid. In particular, we have
$$ g(r\cos(\alpha), r\sin(\alpha)) \underset{r\to\infty \atop \alpha\to 0}{\sim} e^{-r(\cos(\alpha)x(\alpha) + \sin(\alpha)y(\alpha))} \frac{1}{\sqrt{r}} \left(c'\alpha + \frac{c''}{r} 
\right).$$
\item When $\alpha^{*} > 0$ the asymptotics given by \eqref{cas_2} \nr{remains}\nb{remain} valid. In particular, we have
$$
g(r\cos(\alpha), r\sin(\alpha)) \underset{r\to\infty \atop \alpha\to 0}{\sim} c^{*}e^{-r(\cos(\alpha)x^* + \sin(\alpha)y^*)}.
$$
\end{itemize}
Therefore, when \nr{$\alpha^* < 0$}\nb{$\alpha^* = -\infty$}, there is a competition between the two first terms of the sum $\sum_{k=0}^n \frac{c_k(\alpha)}{r^{k}}$ to know which of $c'\alpha$ and $\frac{c''}{r}$ is dominant. More precisely:
\begin{itemize}
\item If $r \sin \alpha \underset{r\to\infty \atop \alpha\to 0}{\longrightarrow} \infty$ then the first term is dominant.
\item If $r \sin \alpha \underset{r\to\infty \atop \alpha\to 0}{\longrightarrow} c>0$ then both terms contribute and have the same order of magnitude.
\item If $r \sin \alpha \underset{r\to\infty \atop \alpha\to 0}{\longrightarrow} 0$ then the second term is dominant.
\end{itemize}

A symmetric result holds when we take $\alpha_0 = \frac{\pi}{2}$. The asymptotics given by \eqref{cas_1} \nr{remains}\nb{remain} valid when $\frac{\pi}{2}<\alpha^{**}$ and \eqref{cas_3} remain valid when $\alpha^{**}<\frac{\pi}{2}$ and there is a competition between the two first terms of the sum to know which one is dominant which in turn depends on the limit of $r \cos (\alpha)$.
\label{thm5}
\end{theorem}

\begin{proof}
The theorem follows directly from several lemmas put together. First, in
Lemma~\ref{3_integ} we invert the Laplace transform and we express the Green's function $g$ as the sum of three integrals. Then, in
Lemma~\ref{residus} we shift the integration contour of the integrals to reveal the contribution of the poles to the asymptotics by applying the residue theorem.
In Lemma~\ref{negl2} we show the negligibility of some integrals which implies that the asymptotic expansion is given by the integrals on the contour of steepest descent. Finally, \nr{in}
Lemma~\ref{nnn} states the \nr{asymptotics}\nb{asymptotic} expansion of these integrals given by the saddle point method.
\end{proof}

\section{Asymptotics when a pole \nr{meet}\nb{meets} the saddle point : \texorpdfstring{$\alpha \to \alpha^*$ or $\alpha \to \alpha^{**}$}{alpha tend vers alpha* ou alpha**}}
\label{sec:polemeetsaddlepoint}

In this section we study the asymptotics of the Green's function $g(r\cos\alpha,r\sin\alpha)$ when $\alpha\to\alpha_0$ in the special \nr{case}\nb{cases} where $\alpha_0=\alpha^*$ or $\alpha_0=\alpha^{**}$, that is when \emph{the pole meets the saddle point}.

We introduce the following notation
\begin{equation}
\label{RR}
R(\alpha)= x'_\omega(0,\alpha)= \sqrt{ \frac{ 2}{F''_{xx}(x(\alpha), \alpha) }  } = \sqrt{ \frac{ 2}{-\sin(\alpha) (Y^+(x))''\mid_{x(\alpha)} }  } 
\end{equation} 
$$
  = \sqrt{ \frac{2\sin(\alpha) \gamma'_y(x(\alpha), y(\alpha))  }{  \sigma_{11}\sin^2(\alpha) + 2\sigma_{12} \sin (\alpha) \cos (\alpha) + \sigma_{22} \cos^2(\alpha) }  } .$$
 
We recall that for $(a,b) \in {\mathbb{R}}_+^2$  we define $r=\sqrt{a^2+b^2}$ and we let $\alpha(a,b)$ be the angle in $[0,\pi/2]$ such that $\cos (\alpha) = \frac{ a }{ \sqrt{a^2+b^2}}$  and  $\sin(\alpha) = \frac{ b }{\sqrt{a^2+ b^2}}$.
\begin{lemma}[Asymptotics of the integral on steepest descent line when $\alpha\to\alpha^*$]  Letting $\alpha(a,b) \to \alpha^*$ as
$r=\sqrt{a^2+b^2  } \to \infty$. Then
\label{ppls} 
$$
I:= \frac{1}{2\pi i}  \int\limits_{\Gamma_{x ,\alpha}} 
 \frac{\gamma_2(x, Y^+(x)) \phi_2(x) }{
       \gamma'_y(x, Y^+(x)) } 
        \exp\Big(\sqrt{a^2+b^2} F(x, \alpha(a,b)) \Big) dx
$$
has the following asymptotics.
\begin{itemize}
\item[(i)]  If $\sqrt{a^2+b^2} (\alpha(a,b) -\alpha^*)^2 \to 0$, 
  then 
  $$I \sim \frac{-1}{2} 
\frac{ \gamma_2(x^*, y^*)  {\rm res}_{x=x^*} \phi_2 }{ \gamma'_y(x^*, y^*)}    
  \  \   \  \hbox{ if  } \alpha(a,b) >\alpha^* , $$ 
   $$I \sim  \frac{1}{2} 
\frac{ \gamma_2(x^*, y^*)  {\rm res}_{x=x^*} \phi_2 }{ \gamma'_y(x^*, y^*)}       \   \  \hbox{ if  } \alpha(a,b) <\alpha^*.$$
\item[(ii)] If $\sqrt{a^2+b^2} (\alpha(a,b) - \alpha^*)^2 \to c>0$.
Further, let
\begin{equation}
\label{AA}
  A(\alpha^*)= \frac{-x'_\alpha(\alpha^*) }{R(\alpha^*)}. 
  \end{equation}  
    Then 
  $$I \sim \frac{-1}{2}\exp(c A^2(\alpha^* )) \Big(1-\Phi\Big(\sqrt{c} A(\alpha^* )\Big)\Big) \times 
\frac{ \gamma_2(x^*, y^*)  {\rm res}_{x=x^*} \phi_2 }{ \gamma'_y(x^*, y^*)}    
 \  \   \  \hbox{ if  } \alpha(a,b) >\alpha^* ,$$ 
   $$I \sim  \frac{1}{2}\exp(c A^2(\alpha^* )) \Big(1-\Phi\Big(\sqrt{c} A(\alpha^* )\Big)\Big) \times 
\frac{ \gamma_2(x^*, y^*)  {\rm res}_{x=x^*} \phi_2}{ \gamma'_y(x^*, y^*)}    
 \  \   \  \hbox{ if  } \alpha(a,b)<\alpha^*$$
where 
  \begin{equation}
      \label{Phi}
  \Phi(z)= \frac{2}{\sqrt{\pi}} \int_0^z \exp(-t^2)dt.
  \end{equation} 
  
\item[(iii)] Let $\sqrt{a^2+b^2}( \alpha(a,b)-\alpha^*)^2 \to \infty$. Then 
$$ I \sim  \frac{ \gamma_2(x^*, y^*) R(\alpha^*) }{2\sqrt{\pi} \gamma_y'(x^*, y^*) } 
\times \frac{ { \rm res}_{x=x^*}\phi_2 }{ (x(\alpha(a,b))-x(\alpha^*))} 
 \times \frac{1} { \sqrt[4]{a^2+b^2}} .$$    
         
\end{itemize}

\end{lemma}

\begin{proof}
Proceeding as we did in Lemma~\ref{nn}, we obtain that
 \begin{equation}
     \label{qrt} 
I\sim \frac{1}{2\pi}  \int\limits_{-\epsilon }^{ \epsilon } 
 f(it, \alpha(a,b))\exp(-\sqrt{a^2+b^2} t^2)dt 
 \end{equation} 
where   
$$f(it, \alpha(a,b)) =\frac{ \gamma_2(x(it, \alpha), Y^+(x(it, \alpha)))\phi_2(x(it, \alpha))}{\gamma'_y (x(it, \alpha), Y^+(x(it, \alpha)))} \times x'_\omega (it ,\alpha).$$
The function $\phi_2$ 
 is a sum of a holomorphic function and of the term 
 $\frac{{\rm res}_{x^*}  \phi_2}{x-x^*}$ which  
 after the change of variables takes the form $\frac{{\rm res}_{x^*  } \phi_2}{x(it, \omega)-x^*}$.
  
  We have $x(0, \alpha^* )=x(\alpha^*)$.
  By the \nr{theorem of implicit function}\nb{implicit function theorem} there exists a function $\omega(\alpha)$ in the class $\mathcal{C}^\infty$ such that 
 $$ x(\omega(\alpha), \alpha) \equiv x^* \  \  
 \ \forall \alpha : |\alpha-\alpha^* |\leq \tilde \eta\  \ \  \ \  \ 
  \nr{\omega(\alpha^*)=0}$$
\nr{with}\nb{for} some $\tilde \eta$ small enough \nb{where $\omega(\alpha^*)=0$}.    
     Furthermore, differentiating this equality, we get 
 $$\omega'(\alpha)= \frac{ -x'_\alpha(\omega(\alpha), \alpha)   }{x'_\omega(\omega(\alpha), \alpha) },$$ 
 so that $$\omega'(\alpha^* )= -\frac{x'_\alpha (\alpha^*)}{x'_\omega(0, \alpha^* )}.$$
    The formula 
  \begin{equation}
      \label{omega}
  \omega(\alpha) = \frac{(x(\alpha^*)-x(\alpha))}{x'_\omega(0, \alpha^*)}(1+o(1))
  = \frac{(x(\alpha^*)-x(\alpha))}{ R(\alpha^*)}(1+o(1)) \  \  \hbox{   as 
 } \alpha \to \alpha^*
   \end{equation} 
 provides the \nr{asymptotic}\nb{asymptotics} of $\omega(\alpha)$ as $\alpha \to \alpha^* $.  
Note also that the main part of $\omega(\alpha)$ is real.

    Let us introduce the function
 $$ \Psi(\omega, \alpha) = \left \{ 
 \begin{array}{cc} 
 \frac{\omega-\omega(\alpha)}{x(\omega, \alpha) - x(\omega(\alpha), \alpha) } &\hbox{  if } 
   \omega \ne \omega(\alpha) \\
\frac{1}{x'_\omega(\omega(\alpha), \alpha) } & 
\hbox{   if } 
   \omega = \omega(\alpha).
 \end{array}
 \right.$$
 This function is holomorphic in $\omega$  
 for any fixed $\alpha$ 
and continuous as a function of three real variables.     
  \nr{We can write the integral (\ref{qrt})}\nb{Note that the integral (\ref{qrt}) can be written as}
$$\frac{1}{2\pi i}  \int\limits_{-\epsilon }^{\epsilon } f(it, \alpha)(it-\omega(\alpha)) 
\frac{\exp(-\sqrt{a^2+b^2}t^2)}{t+i\omega(\alpha)}dt. $$ 
 \nb{Furthermore, there}\nr{There} exists a constant $C>0$ such that 
  \begin{equation}
\label{ppp}
\Big|f(it, \alpha)(it-\omega(\alpha)) - f(0, \alpha)(0-\omega(\alpha))\Big| \leq C|t|.  \  \ 
\forall (it, \alpha) \in \widetilde \Omega(0, \alpha^*)
 = \{(\omega, \alpha) : |\omega| \leq  K, 
   |\alpha-\alpha^*|\leq \min(\eta, \tilde \eta)\}.
\end{equation}
   \nr{It}\nb{Indeed, it} suffices to take $C$ \nb{as} the maximum of the modulus of 
 \nr{ $ (f(\omega, \alpha)(\omega-\omega(\alpha)) - f(0, \alpha)(0-\omega(\alpha)))\omega^{-1} $ on $\{(\omega, \alpha) : |\omega|= K, 
   |\alpha-\alpha^*|\leq \min(\eta_0, \tilde \eta)\}$.}
 \nb{ $$ (f(\omega, \alpha)(\omega-\omega(\alpha)) - f(0, \alpha)(0-\omega(\alpha)))\omega^{-1} $$ on $\{(\omega, \alpha) : |\omega|= K, 
   |\alpha-\alpha^*|\leq \min(\eta, \tilde \eta)\}$ for $\eta$ small enough.} 
   Moreover since ${\Im} \omega(\alpha) =o({\Re} \omega(\alpha))$ as $\alpha \to \alpha^* $, then by Lemma 
   (\ref{technical}) (i) for any $\alpha$ close to $\alpha^* $
the inequality 
$$ \frac{|t|}{|t +i\omega(\alpha) |} \leq 2 $$
holds for all $t \in {\mathbb{R}}$. 
The integral $$\int_{\mathbb{R}} 2 
\exp(-\sqrt{a^2+b^2} t^2)dt =O(\frac{1}{\sqrt[4]{a^2+b^2} })$$
 is of smaller order than the asymptotics announced in the statement of the lemma. 
Hence, it suffices to show that the integral
$$ \frac{1}{2\pi i }\int\limits_{-\epsilon }^{\epsilon } f(0, \alpha)(0-\omega(\alpha)) 
\frac{\exp(-\sqrt{a^2+b^2}t^2)}{t+i\omega(\alpha)} dt$$
has the expected \nr{asymptotic}\nb{asymptotics}.
  Note that by (\ref{omega}) 
 $$\phi(x(\alpha)) (-\omega(\alpha)) x'_\omega(0, \alpha) \to {\rm res}_{x^*} \phi \hbox{   as  }\alpha  \to \alpha^* ,$$ 
  so that 
  $$f(0, \alpha)(-\omega(\alpha)) \to 
  \frac{ \gamma_2(x^*, y^*) }
    {\gamma'_y(x^*, y^*) } \times {\rm res}_{x^*}\phi.$$
It remains to study  
$$ \frac{1}{2\pi i} \int\limits_{-\epsilon }^{\epsilon }  \frac{\exp(-\sqrt{a^2+b^2}t^2)}{t+i \omega(\alpha)}dt.$$
  For any $t \in {\mathbb{R}} \setminus [-\epsilon, \epsilon] $ the denominator in the integral is bounded from below
   $$|t+i\omega(\alpha)|\geq ||t|-\omega(\alpha) |\geq \epsilon -\omega(\alpha) \geq \epsilon/2$$
   for any $\alpha$ close enough to $\alpha^* $
 while $$\int_{\mathbb{R}} \exp(-\sqrt{a^2+b^2}t^2)dt= O(\frac{1}{\sqrt[4]{a^2+b^2}})$$
 is of smaller order than the one stated in the lemma.
Finally, it suffices to prove that 
     \begin{equation}
     \label{dj}
     \frac{1}{2\pi i } \int_{-\infty}^\infty  
       \frac{\exp(- \sqrt{a^2+b^2}t^2)}{t+i \omega(\alpha(a,b))}dt
      \end{equation}  
 has the right \nr{asymptotic}\nb{asymptotics}.  
  By a change of variables, equation~(\ref{dj}) equals
   \begin{equation}
     \label{djj}
     \frac{1}{2\pi i }\int_{-\infty}^\infty  
       \frac{\exp(-s^2)}{s+i \omega(\alpha) \sqrt[4]{a^2+b^2}}ds.
      \end{equation}  
Now let $\alpha >\alpha^* $ [resp. $\alpha< \alpha^*$]. Then $x(\alpha)<x(\alpha^*)$ 
[resp. $x(\alpha)>x(\alpha^*)$]
and 
by (\ref{omega}) 
${\Re} \omega(\alpha)>0$
[resp. ${\Re} \omega(\alpha)<0$]. 
By Lemma~\ref{technical} (iii) this integral evaluates \nb{to}
  $$ \frac{-1}{2} \exp(\sqrt{a^2+b^2} \omega^2 (\alpha)) 
     \Big(1-\Phi(\sqrt[4] {a^2+b^2}\omega (\alpha)) \Big) \  \  \   \hbox{  if  }\alpha>\alpha^*$$
 $$ \frac{1}{2} \exp( \sqrt{a^2+b^2} \omega^2(\alpha) ) 
     \Big(1-\Phi(-\sqrt[4]{a^2+b^2}\omega(\alpha))\Big) \  \  \   \hbox{  if  }\alpha< \alpha^*.$$
 If $\sqrt{a^2+b^2} (\alpha(a,b) -\alpha^* )^2 \to c\geq 0$ 
 then by (\ref{omega}) 
   $\sqrt{a^2+b^2} \omega(\alpha(a,b))^2 
    \to c A^2(\alpha^*)$\nr{,} \nb{and} the results of (i) and (ii)
   are immediate. 
Now let $\sqrt{a^2+b^2}(\alpha(a,b)-\alpha^*)^2 \to \infty$. Then
  by Lemma \ref{technical}  (ii) 
the \nr{asymptotic}\nb{asymptotics} of this integral \nr{is}\nb{are} $$\frac{\sqrt{\pi} }{2\pi i \times (i \omega(\alpha(a,b)))\sqrt[4]{a^2+b^2}}$$  where the \nr{asymptotic}\nb{asymptotics} of $\omega(\alpha(a,b))$ have been stated in (\ref{omega}). The result follows. 
 \end{proof}

It is useful to note that
\begin{multline}
\label{AR}
 -\cos (\alpha) x^*- \sin(\alpha) y^* =
  -\cos (\alpha) x(\alpha) -\sin (\alpha) Y^+(x(\alpha)) 
   + R^{-2}(\alpha^*)(x(\alpha)-x^*)^2 (1+o(1))
   \\
    = -\cos (\alpha) x(\alpha) -\sin (\alpha) Y^+(x(\alpha)) 
   + A^2(\alpha^*)(\alpha -\alpha^*)^2 (1+o(1)), \ \  
    \ \alpha \to \alpha^*
    \end{multline}      
with the notation $R(\alpha)$ and $A(\alpha)$ above in (\ref{RR}) \nr{et}\nb{and} (\ref{AA}).

By Taylor expansion at $x(\alpha)$ and 
and by the definition of a saddle point (the first derivative is zero): 
$$-\cos (\alpha) x^* -\sin (\alpha) y^* = 
  -\cos (\alpha) x(\alpha) -\sin (\alpha) Y^{+}(x(\alpha)) 
     - \frac{1}{2}
     \sin (\alpha) (Y^+(x))''\mid_{x=x(\alpha)}(x(\alpha)-x^*)^2 (1+o(1)),  \  \ \alpha \to \alpha^*.$$
We remind the reader that
 $$  - \frac{1}{2}
     \sin (\alpha) (Y^+(x))''\mid_{x=x(\alpha)}  = (R(\alpha))^{-2} = R(\alpha^*)^{-2} (1+o(1)), \ \ \alpha \to \alpha^* .$$

The following lemma is useful in determining the asymptotics of the value of $I_1$ found in Lemma~\ref{residus}.
\begin{lemma}[Combined contribution of the pole and saddle point to the asymptotics when $\alpha\to\alpha^*$]
\label{nnnn}
Let $r=\sqrt{a^2+b^2} \to \infty$ and $\alpha(a,b) \to \alpha^*$. 
The sum 
\begin{equation}
    \label{uio}
I:= \frac{ \big({\rm -res}_{x=x^*} \phi_2(x) \big ) \gamma_2(x^*, y^*)}{ \gamma'_y(x^*, y^*)} \exp(-ax^* -b y^* ) \times {\bf 1}_{ \alpha <\alpha^*} + \nb{\frac{1}{2i\pi}} \int_{\Gamma_{x, \alpha}} 
  \frac{\gamma_2(x, Y^+(x))\phi_2(x)}{\gamma_y'(x, Y^+(x))}
  \exp(-a x -b Y^+(x))dx
  \end{equation} 
has the following \nr{asymptotic}\nb{asymptotics}.

\begin{itemize} 

\item[(i)] If $\alpha>\alpha^*$ and $\sqrt{a^2+b^2} (\alpha(a,b) -\alpha^*)^2 \to \infty$. Then 
$$I \sim \frac{\exp(-a x(\alpha(a,b)) -b y(\alpha(a,b)))}{\sqrt[4]{a^2+b^2}  }  
\frac{\gamma_2(x^*, y^*)}
{ \sqrt{2\pi}  \sqrt{\sigma_{11}\sin^2(\alpha^*) + 2\sigma_{12} \sin (\alpha^*) \cos (\alpha^*) + \sigma_{22} \cos^2(\alpha^*)} } \times
\frac{ {\rm res}_{x^*}\phi_2 }{x(\alpha(a,b))-x^* } \times C(\alpha^*),$$
where $$C(\alpha^*)= \sqrt{ \frac{ \sin(\alpha^*)}{ \gamma'_y(x^*, y^*)   }} = \sqrt{ \frac{ \cos(\alpha^*) }{ \gamma'_x( x^*, y^*) } }. $$   

\item[(ii)]  If $\alpha >\alpha^*$ and $\sqrt{a^2+b^2} (\alpha(a,b) -\alpha^*)^2 \to c>0$, then 
  $$I \sim \frac{-1}{2} \exp (-a x^* -b y^*) \Big(1-\Phi\Big(\sqrt{c} A(\alpha^* )\Big)\Big) \times 
\frac{ \gamma_2(x^*, y^*)  {\rm res}_{x=x^*} \phi_2 }{ \gamma'_y(x^*, y^*)}.$$ 

where $A(\alpha^*)$ and $\Phi$ are defined in (\ref{AA}), (\ref{RR}) 
and (\ref{Phi}).

\item[(iii)] If $\sqrt{a^2+b^2} (\alpha(a,b) -\alpha^*)^2 \to 0$, then 
      $$I \sim \frac{-1}{2} \exp (-a x^* -b y^*)  \times 
\frac{ \gamma_2(x^*, y^*)  {\rm res}_{x=x^*} \phi_2 }{ \gamma'_y(x^*, y^*)}.$$ 

\item[(iv)]  If $\alpha <\alpha^*$ and $\sqrt{a^2+b^2} (\alpha(a,b) -\alpha^*)^2 \to c>0$, then 
     $$I \sim \frac{-1}{2} \exp (-a x^* -b y^*) \Big(1+\Phi\Big(\sqrt{c} A(\alpha^* )\Big)\Big) \times 
\frac{ \gamma_2(x^*, y^*)  {\rm res}_{x=x^*} \phi_2 }{ \gamma'_y(x^*, y^*)}.$$ 

\item[(v)]  If  $\alpha<\alpha^*$ and $\sqrt{a^2+b^2} (\alpha(a,b) -\alpha^*)^2 \to \infty$, then
      $$I \sim  \exp (-a x^* -b y^*)  \times 
\frac{ - \gamma_2(x^*, y^*)  {\rm res}_{x=x^*} \phi }{ \gamma'_y(x^*, y^*)}.$$ 

\end{itemize} 

\end{lemma} 

\begin{proof} 
 Let us note that  
 \begin{equation}
     \label{uip} 
 \int_{\Gamma_{x, \alpha}} 
  \frac{\gamma_2(x, Y^+(x))\phi_2(x)}{\gamma_y'(x, Y^+(x))}
  \exp(-a x -b Y^+(x))dx 
  \end{equation}
\nr{
 $$ = \exp( - a x(\alpha) -b y (\alpha) )
   \int\limits_{\Gamma_{x,\alpha}} \frac{\gamma_2(x, Y^+(x))\phi_2(x)}{\gamma_y'(x, Y^+(x))}
  \exp\big(\nr{-}\sqrt{a^2+b^2} F(x, \alpha(a,b)) \big) dx  
  $$
}
\nb{ 
$$ = \exp( - a x(\alpha) -b y (\alpha) )
   \int\limits_{\Gamma_{x,\alpha}} \frac{\gamma_2(x, Y^+(x))\phi_2(x)}{\gamma_y'(x, Y^+(x))}
  \exp\big(\sqrt{a^2+b^2} F(x, \alpha(a,b)) \big) dx  
  $$
}
\begin{itemize}

\item[(i)]  The result follows from the representation (\ref{uip}) and 
Lemma \ref{ppls} (iii) with $R(\alpha^*)$ defined in (\ref{RR}). 

\item[(ii)] Invoking (\ref{AR}) the representation (\ref{uip}) can be also written as 
\nr{
    \begin{equation}
        \label{uip1}
     \exp( - a x^* -b y^* - \sqrt{a^2+b^2} 
            A^2(\alpha^*)(\alpha(a,b) -\alpha^*)^2(1+o(1)))
   \int\limits_{\Gamma_{x, \alpha} } \frac{\gamma_2(x, Y^+(x))\phi_2(x)}{\gamma_y'(x, Y^+(x))}
  \exp(\nr{-}\sqrt{a^2+b^2} F(x,\alpha(a,b)) )dx   \end{equation} 
}
\nb{
    \begin{equation}
        \label{uip1}
     \exp( - a x^* -b y^* - \sqrt{a^2+b^2} 
            A^2(\alpha^*)(\alpha(a,b) -\alpha^*)^2(1+o(1)))
   \int\limits_{\Gamma_{x, \alpha} } \frac{\gamma_2(x, Y^+(x))\phi_2(x)}{\gamma_y'(x, Y^+(x))}
  \exp(\sqrt{a^2+b^2} F(x,\alpha(a,b)) )dx   \end{equation} 
}

The result follows from Lemma \ref{ppls} (ii).

\item[(iii)]   
We will consider three subcases. 

\nr{Let first $\alpha(a,b)>\alpha^*$. 
  Then}\nb{If $\alpha(a,b)>\alpha^*$,} the announced result follows from Lemma \ref{ppls} (i) and from the representation (\ref{uip1}) where $\sqrt{a^2+b^2} 
            A^2(\alpha^*)(\alpha(a,b) -\alpha^*) \to 0$.

 If $\alpha=\alpha^*$, then the contour $\Gamma_{x,\alpha}$     
 is special and a direct computation leads to the result. We refer to Lemma~19 of \cite{ernst_franceschi_asymptotic_2021} which deals with a similar case. 
 
    If $\alpha(a,b)<\alpha^*$, 
 then by Lemma \ref{ppls} (i)
 the \nr{asymptotic}\nb{asymptotics} of the second term of (\ref{uio}) 
     \nr{is}\nb{are} the same as in the 
   case $\alpha(a,b)>\alpha^*$ but with the opposite sign. It should 
     be summed with the first term. The sum of their constants
     $1/2-1$  provides the result.   

\item[(iv)]  By the representation (\ref{uip1}) and Lemma \ref{ppls} (ii)
  the \nr{asymptotic}\nb{asymptotics} of the second term of (\ref{uio}) \nr{is}\nb{are} the same as in the case (ii) but with opposite sign. It should be summed with the first term. The sum of their constants 
  $\frac12 (1-\Phi(\sqrt{c} A(\alpha^*)))-1$ leads to the result.  

\item[(v)] By Lemma (\ref{ppls}) (iii) and the representation  (\ref{uip1}) the second term of (\ref{uio}) has the \nr{asymptotic}\nb{asymptotics}
$$  \exp\Big( - a x^* -b y^* - \sqrt{a^2+b^2} 
            A^2(\alpha^*)(\alpha(a,b) -\alpha^*)^2(1+o(1)) \Big)$$
            $$ \times 
          \frac{\gamma_2(x^*, y^*) R(\alpha^*) }{2\sqrt{\pi} \gamma_y'(x^*, y^*) } 
\times \frac{ {\rm res}_{x=x^*}\phi }{ (x(\alpha(a,b))-x(\alpha^*))} 
 \times \frac{1} { \sqrt[4]{a^2+b^2}}.$$
 Since $\frac{\exp(-\sqrt{a^2+b^2} 
            A^2(\alpha^*)(\alpha(a,b) -\alpha^*)^2)}{ (\alpha(a,b)-\alpha^* ) \sqrt[4]{a^2+b^2}}$  converges to $0$ 
    in this case,  the order of the second term in (\ref{uio}) 
 is clearly smaller than the one of the first term which dominates 
 the asymptotics.
\end{itemize} 
\end{proof} 

\begin{rem}[Consistency of the results]
The results of (i) and (v) are perfectly ``continuous'' with asymptotics along directions $\alpha \to \alpha^*, \alpha < \alpha^*$ and $\alpha \to \alpha^*, \alpha > \alpha^*$. Namely, if in (i) we substitute $\phi(x(\alpha))$ instead of $\frac{res_{x^*} \phi}{x(\alpha(a,b)) - x^*}$, 
  we obtain the \nr{asymptotic}\nb{asymptotics} for angles greater than $\alpha^*$. 
  The result (v) remains valid for angles less than $\alpha^*$. 
\end{rem}

We now summarize the previous results to obtain our final main result.
\begin{theorem}[Asymptotics in the quadrant when the saddle point \nr{\emph{meet}}\nb{\emph{meets}} a pole]
We now assume that $\alpha_0=\alpha^{*}$ and let $\alpha\to\alpha^{*}$ and $r\to\infty$. Then, the Green's density function $g(r\cos \alpha, r \sin \alpha)$ has the following asymptotics:
\begin{itemize}
\item When $r(\alpha - \alpha^{*})^2\to 0$ then the {principal term of the} asymptotics is given by \eqref{cas_2} but the constant $c^{*}$ of the first term has to be replaced by $\frac{1}{2}c^{*}$.
\item When $r(\alpha-\alpha^{*})^2\to c>0$ for some constant $c$ then:
\begin{itemize}
\item If $\alpha<\alpha^{*}$ the {principal term of the} asymptotics is still given by \eqref{cas_2} but the constant $c^{*}$ of the first term has to be replaced by $\frac{1}{2}c^{*}(1+\Phi(\sqrt{c}A))$ for some constant $A$.
\item If $\alpha>\alpha^{*}$ the {principal term of the} asymptotics is still given by \eqref{cas_2} but the constant $c^{*}$ of the first term has to be replaced by $\frac{1}{2}c^{*}(1-\Phi(\sqrt{c}A))$ for some constant $A$.
\end{itemize}
Note that above $\Phi(z)= \frac{2}{\sqrt{\pi}} \int_0^z \exp(-t^2)dt$.
\item When $r(\alpha-\alpha^{*})^2\to \infty$ then:
\begin{itemize}
\item If $\alpha<\alpha^{*}$ the {principal term of the} asymptotics is given by \eqref{cas_2}.
\item If $\alpha>\alpha^{*}$ the {principal term of the} asymptotics is given by \eqref{cas_1} and we have ${c_0}(\alpha)\underset{\alpha\to\alpha^{*}}{\sim}\frac{c}{\alpha-\alpha^{*}}$ for some constant $c$.
\end{itemize}
\end{itemize}

A symmetric result holds when we assume that $\alpha_0=\alpha^{**}$.
\label{thm6}
\end{theorem}

\begin{proof}
The theorem follows directly from several lemmas put together.
The Green's function $g$ is still given by the sum $I_1 + I_2 + I_3$, see
Lemma~\ref{3_integ}. We again apply Lemma \ref{residus} to take into account the contribution of the poles and Lemma \ref{pp} which shows the negligibility of some integrals in the final asymptotics. Furthermore, by the proof of Lemma \ref{nn}, $I_2 + I_3 = O\left(\frac{e^{-r\cos(\alpha^*)x(\alpha^*) - r\sin(\alpha^*)y(\alpha^*))}}{\sqrt r}\right)$ when $r \to \infty$ and $\alpha \to \alpha^* $ (recall that $\alpha^*< \alpha^{**}$). With Lemma \ref{nnnn}, we see in each case that $I_2 + I_3$ is negligible compared to $I_1$ when $r\to\infty$ and $\alpha\to\alpha^*$. Indeed, in the case $\alpha > \alpha ^*$ and $r(\alpha - \alpha^*)^2 \to \infty$, the domination of $I_1$ is due to the term $\frac{1}{x(\alpha) - x^*}$. For the other cases, the domination of $I_1$ is due to the factor $\frac{1}{\sqrt r}$ in the asymptotics of $I_2 + I_3$. 

The proof is similar for $\alpha_0 = \alpha^{**}$.
\end{proof}

\section{Asymptotics in a cone}
\label{sec:asymptcone}

\subsection*{From the quadrant to the cone}
Let us describe the linear transformation which maps the reflected Brownian motion in the quarter plane (of covariance matrix $\Sigma$ and reflecting vectors $R^1$ and $R^2$) to a reflected Brownian motion in a wedge with identity covariance matrix. We take
\begin{equation}
\beta=\arccos \left( -\frac{\sigma_{12}}{\sqrt{\sigma_{11}\sigma_{22}}} \right)
\in(0,\pi)
\end{equation}
and we define 
\begin{equation}
T= \left(
  \begin{array}{cc}
    \displaystyle     \frac 1 {\sin \beta} & \cot \beta\\
    0 & 1
  \end{array}
\right)
\left(
  \begin{array}{cc}
    \displaystyle    \frac 1 {\sqrt{\sigma_{11}}} & 0\\
    0 & \displaystyle  \frac 1 {\sqrt{\sigma_{22}}}
  \end{array}\right)
  \label{eq:lineartransform}
\end{equation}
which satisfies $T\Sigma T^\top= \text{Id}$.
Then, if $Z_t$ is a reflected Brownian motion in the quadrant of parameters $(\Sigma,\mu,R)$, the process $\widetilde{Z_t}=TZ_t$ is a reflected Brownian motion in the cone of angle $\beta$ and of parameters 
$(T\Sigma T^\top,T\mu,TR)=(\text{Id},\widetilde{\mu},TR)$. The new reflection matrix $TR$ correspond to reflections of angles $\delta$ and $\epsilon$ defined in $(0,\pi)$ by
\begin{equation}
\tan\delta=
\frac{\sin\beta}{\frac{r_{12}}{r_{22}}\sqrt{\frac{\sigma_{22}}{\sigma_{11}}}+\cos\beta} 
\qquad
\text{and}
\qquad
\tan\varepsilon=
\frac{\sin\beta}{\frac{r_{21}}{r_{11}}\sqrt{\frac{\sigma_{11}}{\sigma_{22}}}+\cos\beta}.
\end{equation}
The new drift has an angle $\theta=\arg \widetilde{\mu}$ with the horizontal axis and satisfies
\begin{equation}
\tan \theta=\frac{\sin \beta}{\frac{\mu_1}{\mu_2}\sqrt{\frac{\sigma_{22}}{\sigma_{11}}} +\cos \beta}.
\end{equation}
The assumption $\mu_1>0$ and $\mu_2>0$ is equivalent to $\theta\in(0,\beta)$.

\subsection*{Green's functions in the cone}
Let us denote $g^{z_0}$ the density of $G(z_0,\cdot)$. For $z\in\mathbb{R}_+^2 $ we have
$$
g^{z_0}(z)=\int_0^\infty p_t(z_0,z)\mathrm{d}t.
$$
Let us recall that we have denoted $\widetilde{G}(\widetilde{z_0},\widetilde{A})$ the Green measure of $\widetilde{Z_t}$ and $\widetilde{g}^{\widetilde{z_0}}(\widetilde{z})$ its density. It is straightforward to see that for $A\in\mathbb{R}_+^2$ we have $G(z_0,A)=\widetilde{G}(T{z_0},TA)$ and then
\begin{equation}
g^{z_0}(z)= |\det T| \widetilde{g}^{T{z_0}}(Tz)=\frac{1}{\sqrt{\det \Sigma}}\widetilde{g}^{\widetilde{z_0}}(\widetilde{z})
\label{eq:ggtildelink}
\end{equation}
where $\widetilde{z_0}=Tz_0$ and $\widetilde{z}=Tz$.

\subsection*{Polar coordinates}
For any $z=(a,b)\in\mathbb{R}_+^2$ we may define the polar coordinate in the quadrant $(r,\alpha)\in \mathbb{R}_+ \times [0,\frac{\pi}{2}]$ by
\begin{equation}
z=(a,b)=(r\cos\alpha,r\sin\alpha).
\end{equation}
We now define the polar coordinates in the $\beta$-cone $(\rho,\omega)$ by
\begin{equation}
\widetilde{z}=(\rho \cos \omega, \rho \sin \omega).
\end{equation}
For $\widetilde{z}=Tz$ we obtain by a direct computation that
\begin{equation}
(r\cos\alpha,r\sin\alpha)=(\rho \sqrt{\sigma_{11}} \cos (\beta-\omega), \rho \sqrt{\sigma_{22}}\sin \omega).
\label{eq:ralpharho}
\end{equation}
and that 
\begin{equation}
\tan \omega = \frac{\sin \beta}{\frac{1}{\tan \alpha}\sqrt{\frac{\sigma_{22}}{\sigma_{11}}}+\cos \beta}.
\label{eq:alphaomega}
\end{equation}
We deduce that
\begin{equation}
\widetilde{g}^{\widetilde{z_0}}(\rho \cos \omega, \rho \sin \omega)=\sqrt{\det \Sigma}
\ g^{z_0}(\rho \sqrt{\sigma_{11}} \cos (\beta-\omega), \rho \sqrt{\sigma_{22}}\sin \omega).
\end{equation}

\subsection*{Saddle point}
The ellipse $\mathcal{E}=\{(x,y)\in\mathbb{R}^2 : \gamma(x,y)=0\}$ can be easily parametrized by the following, 
$$\mathcal{E}=\left\{ (\widetilde{x}(t),\widetilde{y}(t)) : t\in[0,2\pi]  \right\},$$
where
\begin{equation}
\begin{cases}
\widetilde{x}(t)=\frac{x_{max}+x_{min}}{2}+\frac{x_{max}-x_{min}}{2}\cos(t),
\\
\widetilde{y}(t)=\frac{y_{max}+y_{min}}{2}+\frac{y_{max}-y_{min}}{2}\cos(t-\beta) .
\end{cases}
\label{eq:param1}
\end{equation}
see Proposition~5 of~\cite{franceschi_asymptotic_2016}. Noticing that 
$$
-\cos \theta = \frac{x_{max}+x_{min}}{x_{max}-x_{min}},
\quad \text{and} \quad
-\cos (\beta-\theta) = \frac{y_{max}+y_{min}}{y_{max}-y_{min}}
$$
and that 
$$
2|\widetilde{\mu}|=\sqrt{\sigma_{11}}(x_{max}-x_{min})\sin \beta=\sqrt{\sigma_{22}}(y_{max}-y_{min})\sin \beta
$$
we obtain
\begin{equation}
\begin{cases}
\widetilde{x}(t)
= \frac{|\widetilde{\mu}|}{\sqrt{\sigma_{11}} \sin \beta}(\cos t-\cos \theta )
= \frac{2|\widetilde{\mu}|}{\sqrt{\sigma_{11}}\sin \beta} \sin(\frac{\theta-t}{2})\sin(\frac{t+\theta}{2})
\\
\widetilde{y}(t)
= \frac{|\widetilde{\mu}|}{\sqrt{\sigma_{22}}\sin \beta}(\cos(t-\beta)-\cos (\theta -\beta))
= \frac{2|\widetilde{\mu}|}{\sqrt{\sigma_{22}}\sin \beta} \sin(\frac{\theta-t}{2})\sin(\frac{t+\theta-2\beta}{2})
.
\end{cases}
\label{eq:param2}
\end{equation}
The following result gives an expression of the saddle point in terms of the polar coordinate in the cone.
\begin{prop}[Saddle point in polar coordinate]
For $\alpha\in(0,\frac{\pi}{2})$ and $\omega\in(0,\beta)$ previously defined and linked by \eqref{eq:alphaomega} we have
\begin{equation}
(x(\alpha),y(\alpha))=(\widetilde{x}(\omega),\widetilde{y}(\omega))
\label{eq:saddlecone}
\end{equation}
where $(x(\alpha),y(\alpha))$ is the saddle point defined in \eqref{eq:defsaddlepoint}.
\label{prop:saddlepolar}
\end{prop}
\begin{proof}
Letting $\alpha\in(0,\frac{\pi}{2})$, we are looking for the point $(x(\alpha),y(\alpha))$ which maximizes the quantity $x\cos \alpha +y \sin \alpha$ for $(x,y)$ in the ellipse $\mathcal{E}=\{(x,y)\in\mathbb{R}^2 : \gamma(x,y)=0\}$. 
We search for a $t\in(0,\beta)$ cancelling the derivative of $\widetilde{x}(t)\cos \alpha + \widetilde{y}(t)\sin \alpha$ w.r.t $t$. By \eqref{eq:param2} we obtain that $\widetilde{x}'(t)\cos \alpha + \widetilde{y}'(t)\sin \alpha = 0$ if and only if
$$
-\frac{1}{\sqrt{\sigma_{11}}}\sin t \cos \alpha -\frac{1}{\sqrt{\sigma_{22}}} \sin(t-\beta)\sin\alpha  =0.
$$
Writing $\sin(t-\beta)=\sin t \cos \beta -\cos t \sin \beta$ it directly leads to
$\tan t = \frac{\sin \beta}{\frac{1}{\tan \alpha}\sqrt{\frac{\sigma_{22}}{\sigma_{11}}}+\cos \beta}$. Then
by \eqref{eq:alphaomega} we obtain $\tan t=\tan \omega$ and we deduce that $t=\omega$ maximizes $\widetilde{x}(t)\cos \alpha + \widetilde{y}(t)\sin \alpha$ and therefore $(x(\alpha),y(\alpha))=(\widetilde{x}(\omega),\widetilde{y}(\omega))$.
\end{proof}

\subsection*{Poles}
Let us recall that $x^*$ is the pole of $\phi_2(x)$ (when $x^* > 0$), and $y^{**}$ is the pole of $\phi_1(y)$ (when $y^{**}> 0$), see Proposition~\ref{poles_expr}.
We defined $\alpha^*$ and $\alpha^{**}$ such that $x(\alpha^*)=x^*$ and $y(\alpha^{**})=y^{**}$.
Now, we may define the corresponding $\omega^*$ and $\omega^{**}$ \nr{linked by}linked by formula \eqref{eq:alphaomega} and such that
\begin{equation}
x^*=\widetilde{x}(\omega^*)=\widetilde{x}(-\omega^*)
\quad\text{and}\quad
y^{**}=\widetilde{y}(\omega^{**})=\widetilde{y}(2\beta-\omega^{**}).
\end{equation}
\begin{prop}[Poles in polar coordinate]
We have
\begin{equation}
\omega^*=\theta-2\delta
\quad\text{and}\quad
\omega^{**}=\theta+2\epsilon.
\end{equation}
We have,
$\alpha<\alpha^*$ if and only if $\omega<\omega^*$,
and $\alpha>\alpha^{**}$ if and only if $\omega>\omega^{**}$.
Then, $x^*$ is the pole of $\phi_2(x)$ if and only if $\theta-2\delta>0$, and $y^{**}$ is a pole of $\phi_1(y)$ if and only if $\theta+2\epsilon<\beta$.
\label{prop:polepolar}
\end{prop}
\begin{proof}
When the pole of $\phi_2$ exists, we have
$\gamma_2(x^*,Y^-(x^*))=0$.
Let us recall that in~\eqref{eq:x**y*} we defined $y^*:=Y^+(x^*)=\widetilde{y}(\omega^*)$. Therefore, we have $Y^-(x^*)=\widetilde{y}(-\omega^*)$.
We are looking for the solutions of the equation 
\begin{equation}
\gamma_2(\widetilde{x}(t),\widetilde{y}(t))=0,
\label{eq:gam1}
\end{equation}
which is the intersection of the ellipse $\mathcal{E}$ and the line  $\gamma_2=0$.
There are two solutions, the first one is elementary and is given by $t=\theta$, that is $(\widetilde{x}(t),\widetilde{y}(t))=(0,0)$. The second one is by definition $(\widetilde{x}(-\omega^{*}),\widetilde{y}(-\omega^{*}))=(x^{*},Y^-(x^{*}))$. 
By \eqref{eq:param2}, the equation \eqref{eq:gam1} gives
$$
r_{12}\frac{1}{\sqrt{\sigma_{11}}}\sin \left( \frac{-\omega^{*}+\theta}{2}\right)+r_{22} \frac{1}{\sqrt{\sigma_{22}}}\sin \left( \frac{-\omega^{*}+\theta}{2} -\beta \right)
=0
$$
With some basic trigonometry, we obtain that
$$
\tan \frac{-\omega^{*}+\theta}{2} = \frac{\sin \beta}{\frac{r_{12}}{r_{22}} \sqrt{\frac{\sigma_{22}}{\sigma_{11}}}+\cos \beta}=\tan(\delta).
$$
We deduce that $\omega^{*}=\theta-2\delta$. A symmetric computation leads to $\omega^{**}=\theta+2\epsilon$. 
The necessary and sufficient condition for the existence of the poles comes from Proposition~\ref{poles_expr}.
The inequalities on $\alpha$ transfer to $\omega$ by equation~\eqref{eq:alphaomega}. 
\end{proof}

\subsection*{Asymptotics in the cone}
We now compute the exponential decay rate in terms of the polar coordinate in the cone.
\begin{prop}[Exponential decay rate]
For $\alpha$ and $\omega$ previously defined and linked by \eqref{eq:alphaomega} we have
\begin{equation}
r\cos(\alpha) x(\alpha)+ r\sin (\alpha) y(\alpha)=2\rho|\widetilde{\mu}| \sin^2 \left( \frac{\omega-\theta}{2} \right)
\end{equation}
and
\begin{equation}
r\cos(\alpha) x(\alpha^*)+ r\sin (\alpha) y(\alpha^*)=2\rho|\widetilde{\mu}| \sin^2 \left( \frac{2\omega-\omega^*-\theta}{2} \right).
\end{equation}
\label{prop:decayrate}
\end{prop}
\begin{proof}
By Equations~\eqref{eq:ralpharho} and \eqref{eq:saddlecone} we obtain the desired result.
\end{proof}

\begin{proof}[Proofs of Theorems~\ref{thm1}, \ref{thm2} and \ref{thm3}]
Equation~\eqref{eq:ggtildelink} and
Propositions~\ref{prop:saddlepolar},~\ref{prop:polepolar},~\ref{prop:decayrate} combined to Theorem~\ref{thm4} (resp. Theorems \ref{thm5} and \ref{thm6}), lead to Theorem~\ref{thm1} (resp. Theorems~\ref{thm2} and \ref{thm3}).
\end{proof}

\clearpage
\appendix

\section{Parameter-dependent Morse Lemma}
\label{sec:morse}

The following lemma is a parameter-dependent Morse lemma. Although it is an intuitive result, we could not find it in the existing literature.

\begin{lemma} 
\label{Morse}
  Assume that $\alpha_0\in\mathbb{R}$ is a constant, $\alpha\mapsto x(\alpha)$ is a function which is $\mathcal{C}^\infty$ near $\alpha_0$, and $(x,\alpha)\mapsto F(x,\alpha)$ is a function which is analytic as a function of the first variable $x$ and $\mathcal{C}^\infty$ as a function of the second variable $\alpha$ near $(x(\alpha_0),\alpha_0)$. Furthermore, assume that for all $\alpha$ near $\alpha_0$ we have
  $$F(x(\alpha), \alpha)=0 ,
  \quad
  F'_x(x(\alpha), \alpha)=0,
  \quad
  F''(x(\alpha), \alpha)>0.
  $$
  There exists a neighborhood of 
  $(0, \alpha_0)$ in ${\mathbb{C}} \times {\mathbb{R}}$
  $$ \Omega(0, \alpha_0)=\{ (\omega, \alpha) \in {\mathbb{C} \times \mathbb{R}} :
     |\omega|\leq K, |\alpha-\alpha_0|\leq \eta \}$$
with some $K, \eta >0$  
 and a function $x(\omega, \alpha)$ defined in $\Omega(0, \alpha_0)$
such that 
$$F(x(\omega, \alpha),  \alpha) =\omega^2, \ \  \forall \omega : |\omega|\leq K $$
$$ x(0, \alpha) = x(\alpha) \ \  \forall \alpha : |\alpha-\alpha_0| \leq \eta.$$
Furthermore $x(\omega, \alpha)$ is in the class $\mathcal{C}^\infty$  as function of three real variables ${\Re}\omega, {\Im} \omega, \alpha$ and holomorphic of $\omega$ for any fixed $\alpha$. 
Finally 
  \begin{equation}
      \label{zop}
  x'_\omega(0, \alpha) = \sqrt{\frac{2}{F''_x(x(\alpha), \alpha)) } }  .
   \end{equation} 
 \end{lemma} 

\begin{proof} 
   This is an adaptation of Morse’s lemma to the dependence of the parameter $\alpha$.  
Consider $T(z,\alpha)=F(z+x(\alpha), \alpha)$. 
 Then $T(0, \alpha)=0$, $T'_z(0, \alpha) =0$ 
 and $T''_z(0, \alpha) = F''_x(x(\alpha), \alpha)>0 $
   for any $\alpha$ close to $\alpha_0$.
  Then the following representation holds
  \begin{equation}
      \label{TTT} 
  T(z, \alpha)=z^2 F''_x(x(\alpha), \alpha)/2  +z^3 h(z, \alpha) 
  \end{equation} 
which allows us to define 
$$S(z,\alpha)= z\sqrt{  F''_x(x(\alpha), \alpha)/2   + z h(z,\gamma)}$$ 
   with one of two branches of the square root. Let us choose the one that takes the value $+F''_x(x(\alpha), \alpha)/2 $ at $z=0$.
      Due to elementary properties of the function $F$
  and the fact that $x(\alpha)$ is in class $\mathcal{C}^\infty$, the function $h(z,\alpha)$ in the representation of $T$ above
        is in class $\mathcal{C}^\infty$ in a neighborhood 
  of ${\cal O}(0, \alpha_0) \subset {\mathbb{C}} \times {\mathbb{R}}$ 
  as a function of three real variables 
    and also holomorphic in $z$ for any fixed $\alpha$. Furthermore, 
 \begin{equation}
     \label{FF}
  S'_z (0, \alpha_0) = F''_x(x(\alpha_0), \alpha_0)/2  \ne 0.
  \end{equation} 
   Then by the \nr{theorem of implicit function}\nb{implicit function theorem} (the real one to establish the announced properties in ${\mathbb{R}}^3$ and the complex one to show the holomorphicity), there exists a function 
    $z(\omega, \alpha)$ in a neighborhood of $(0, \alpha_0)$
     which is in the class $\mathcal{C}^\infty$ in three variables and holomorphic in $\omega$ such that 
     \begin{equation}
         \label{zet}
     S(z(\omega, \alpha), \alpha) \equiv \omega, \ \  \  \   z(0, \alpha_0)=0.
     \end{equation}
    This means that $T(z(\omega, \alpha), \alpha)\equiv \omega^2$
      for any couple $(\omega, \alpha)$  in this neighborhood.
 In particular, \nb{the} function 
$z(0, \alpha)$ solves the equation $S(z, \alpha) \equiv 0$ in the variable $z$. 
     Since $S'_z(0, \alpha_0)\ne 0$, a function in the class $\mathcal{C}^\infty$  of a real variable $\alpha$
         satisfying this equation  
 and vanishing at $\alpha_0$ is 
    unique by the \nr{theorem of implicit function}\nb{implicit function theorem}.  
  But we know already that $S(0, \alpha)=0$ for any $\alpha$ close to $\alpha_0$. Hence, $z(0, \alpha)\equiv 0$
      for any $\alpha$ close to $\alpha_0$.
        
    Now, let
     $$x(\omega, \alpha)=z(\omega, \alpha)+ x(\alpha),$$
       where $x(\alpha)$ is in the class $\mathcal{C}^\infty$.   
          It satisfies all expected properties. Furthermore
    $F(x(\omega, \alpha), \alpha)\equiv \omega^2$.
     Differentiating this identity twice, we obtain (\ref{zop}). 
     \end{proof}

\section{Technical Results} 
\label{app:tech}
This following lemma is useful in Section~\ref{sec:polemeetsaddlepoint} for finding out how the asymptotics behave as the saddle point approaches the pole.

\begin{lemma} 
\begin{itemize}
\item[(i)] If $C>0$ is such that 
$C^2 \geq 1 + \frac{B^2}{A^2}$, then 
$$  \frac{|s|}{|s+i (A+i B)|} \leq C \  \  \forall s \in {\mathbb{R}}.$$
\item[(ii)] 
   Let $|A| \to \infty$ and $B=o(A)$ as $|A| \to \infty$.
 Then 
 $$ \int_{-\infty}^{\infty} \frac{\exp(-s^2)}{s+i(A+iB)} ds 
   \sim \frac{\sqrt{\pi}}{i(A+iB)}.$$ 
\item[(iii)] 
Let 
$$\Pi(w)= \int_{-\infty}^{\infty} \frac{\exp (-s^2)}{s+iw}ds$$
 with $\Re w \ne 0$. 
 This function is holomorphic in each half plane
 $\{ w : {\Re}\, w>0\}$ and $\{w: {\Re}\, w<0\}$ and can be made explicit:
$$ \Pi(w)=\pi i \exp(w^2)(1-\Phi(-w)) \ \  \forall w : \ {\Re} w<0$$
$$\Pi(w)= -\pi i \exp(w^2)(1-\Phi(w)) \ \  \forall z : \ {\Re} w>0$$
  where $\Phi(w)=\frac{2}{\sqrt{\pi}} \int_0^w \exp(-s^2)ds.$
\end{itemize} 
\label{technical}
\end{lemma}

\begin{proof} 
\begin{itemize}
\item[(i)] Elementary computation.
\item[(ii)]
    We have $\int_{-\infty}^{\infty}\frac{\exp(-s^2)}{i(A+ i B)}ds= 
    \frac{\sqrt{\pi} }{i(A+iB )}. $
  It suffices to show that $$\int_{\mathbb{R}}  \frac{ |s|}{|s+i(A+iB) |} \exp(-s^2)ds$$  converges to $0$ for any $A$ with absolute value large enough to have $\frac{|A|}{|B|}\geq 1$.
   Then by (i) 
     $\frac{|s|}{|s +i(A+iB)|} \leq 2$ for any $s \in {\mathbb{R}}$. 
   Since the integral $\int_{\mathbb{R}} 2 \exp(-s^2)ds$  converges, the dominated convergence theorem applies and we get the stated asymptotics.  
  
\item[(iii)] 
Let us define for any $z>0$ and $w>0$
$$ \Pi(z, w)= \int_{-\infty}^\infty \frac{\exp(-zs^2)}{s+i w} ds.$$
Then 
\begin{align*}
\Pi'_z(z,w) &= \int_{-\infty}^{\infty} \frac{-s^2 \exp(-z s^2)  }{s+iw } ds = \int_{-\infty}^\infty \frac{((iw)^2-s^2-(iw)^2)\exp(-z s^2)}{s+iw  } ds\\
&= \int_{-\infty}^{\infty} (iw-s) \exp(-z s^2)ds + w^2\int_{-\infty}^\infty 
  \frac{ \exp(-z s^2) }{s+iw}ds \\
&=iw \sqrt{\frac{\pi}{z}} + w^2 \Pi(w,z).
\end{align*}

Solving this differential equation, we get that 
$\Pi(w,z)=c(w,z)\exp(w^2z)$ where $c'_z(w,z)= iw \sqrt{\frac{\pi}{z}}\exp(-w^2z)$. Taking into account the fact that $\Pi(+\infty, w)=0$, we obtain
$$ \Pi(z,w)= -i w \sqrt{\pi} \exp (w^2 z) \int_z^\infty t^{-1/2} \exp(-w^2 t )dt= 
   - i w \sqrt{\pi} \exp(w^2z) \int_{w\sqrt{z}}^\infty \exp (-s^2)ds$$
   $$= 
      -i w \pi \exp (w^2 z) \Big(1-\Phi (w \sqrt{z})\Big).$$
Now let $z=1$. 
Then $$ \Pi(1, w)= -\pi i \exp(w^2) (1-\Phi(w))$$ for any real positive $w$. The holomorphicity of $\Phi(w)$ in $\{w \in {\mathbb{C}} : \Re w>0 \} $ allows us to prove statement (iii). 
  Finally, we note that for any $w$ with $\Re w<0$, $\Pi(-w)=-\Pi(w)$.
    
\end{itemize}
\end{proof} 

\section{Green's functions near zero and Laplace transforms near infinity}
\label{appendixC}

We introduce the parameter 
$$\lambda =\frac{\delta+\epsilon-\pi}{\beta}$$
where $\beta$ is the angle of the cone, and $\epsilon$ and $\delta$ are the angles of reflection which can be expressed in terms of the covariance matrix $\Sigma$ and the reflection matrix $R$, see Section~\ref{sec:asymptcone}. This parameter $\lambda$ is well known in the SRBM literature and is usually denoted by $\alpha$ but to avoid any confusion of notation we have called it $\lambda$ in this article. It is well known that existence conditions of the SRBM stated in \eqref{eq:condexist} are equivalent to 
$$\lambda<1.$$

\begin{lemma}[Laplace transforms behaviour near infinity and Green's functions near zero]
For some constants $C_1$ and $C_2$, the Laplace transforms $\phi_1$ and $\phi_2$ satisfy 
\begin{equation}
\phi_1(y)\sim C_1y^{\lambda-1} \text{ when } |y|\to\infty
\quad\text{and}\quad
\phi_2(x)\sim C_2x^{\lambda-1} \text{ when } |x|\to\infty
\label{eq:equivphiinfini}
\end{equation}
and their derivatives satisfy
\begin{equation}
\phi_1'(y)\sim C_1(\lambda-1)y^{\lambda-2} \text{ when } |y|\to\infty
\quad\text{and}\quad
\phi_2'(x)\sim C_2(\lambda-1)x^{\lambda-2} \text{ when } |x|\to\infty.
\label{eq:equivderiveinfini}
\end{equation}
Furthermore, the Green's functions on the boundaries $h_1$ and $h_2$ satisfy
\begin{equation}
h_1(v)\sim C_1 \Gamma(-\lambda+1) v^{-\lambda} \text{ when } |v|\to 0
\quad\text{and}\quad
h_2(u)\sim C_2 \Gamma(-\lambda+1) u^{-\lambda} \text{ when } |u|\to 0,
\label{eq:equivgreen0}
\end{equation}
where $\Gamma$ is the gamma function.
\label{lem:C1}
\end{lemma}
We give the sketch of the proof of the previous Lemma which relies on the resolution of Boundary Value Problem studied in~\cite{franceschi_green_2021}. This lemma is not crucial for establishing the results of this article. It is only used to simplify the proof of Lemma~\ref{negl2} which is useful only in the special case where we are looking for the asymptotics along the axes.
\begin{proof}[Sketch of proof]
The article~\cite{franceschi_green_2021} states in Theorem 11 an explicit expression for the Laplace transform $\phi_1$. 
This result is obtained by solving a Carleman Boundary Value Problem coming from the functional equation~\eqref{Equation fonctionnelle}. 
The solution is the product of the solution of the corresponding homogeneous problem and an integral, namely,
$$
\phi_1(y)=X(W(y))\left( \frac{1}{2\pi}\int_{\mathcal{R}^-} \frac{g(t)}{X^+(t)}\frac{dt}{W(y)-W(t)}+C \right),
$$
where we have taken the notation of Theorem 11 in \cite{franceschi_green_2021} and its proof. 
Since $\frac{g(t)}{X^+(s)}$ converges to $0$ when $t$ tends to infinity, the integral $\frac{1}{2\pi}\int_0^1 \frac{g(t)}{X^+(t)}\frac{dt}{W(y)-W(t)}$ converges to a constant when $y\to\infty$ by classical complex analysis results, see (5.2.17) of \cite{FIM17}. The function $X(W(y))$ is the solution to the corresponding homogeneous BVP which is studied in detail in the recurrent case in \cite{franceschi_explicit_2017}. Proposition 19 of \cite{franceschi_explicit_2017} shows that $X(W(y)))\sim y^{\lambda-1}$ when $y$ tends to infinity, which concludes the proof of~\eqref{eq:equivphiinfini}. 

Integral Hardy–Littlewood Tauberian theorems (see for example Karamata's theorem and Ikehara's theorem \cite[\S 7.4 \& 7.5]{tenenbaum_95} and \cite[Thm 33.3 \& 33.7]{doetsch_introduction_1974}) state that, with some hypotheses, for a function $f$ and its Laplace transform $\mathcal{L}(f)$, for $\lambda\geqslant -1$, $f(t)\sim C t^{-\lambda}$ when $t\to 0$ is equivalent to $\mathcal{L}(f)(x)\sim C \Gamma (-\lambda+1) x^{\lambda-1}$ when $x\to\infty$. Equation~\eqref{eq:equivgreen0} follows from a Tauberian theorem and from~\eqref{eq:equivphiinfini}.

The proof of \eqref{eq:equivderiveinfini} follows from~\eqref{eq:equivgreen0}, from a Tauberian theorem and from the properties of the derivative of the Laplace transform, namely $\mathcal{L}(tf(t))=\frac{d}{dx}\mathcal{L}(f)(x)$.
\end{proof}

\bibliographystyle{abbrv} 

\end{document}